\documentclass[reqno]{amsart}
\usepackage{setspace,tikz,xcolor,mathrsfs,listings,multicol}
\usepackage{rotating}
\usepackage{booktabs, afterpage}
\usepackage[vcentermath]{youngtab}
\usepackage{enumerate}
\usepackage[all,cmtip]{xy}
\usetikzlibrary{arrows,matrix}

\usepackage[colorlinks=true, pdfstartview=FitV, linkcolor=blue, citecolor=blue, urlcolor=blue]{hyperref}

\newcommand{\g}{\mathfrak{g}}
\newcommand{\HH}{\mathcal{H}}
\newcommand{\clfw}{\overline{\Lambda}} 
\newcommand{\clsr}{\overline{\alpha}} 
\newcommand{\inner}[2]{\left\langle #1, #2 \right\rangle}
\newcommand{\iso}{\cong}

\newcommand{\sgn}{\!\operatorname{sgn}}

\newcommand{\Sol}{\mathcal{S}} 
\newcommand{\modgamma}{\check{\gamma}} 
\newcommand{\vdomino}[2]{\begin{array}{|c|} \hline #1 \\\hline #2 \\\hline \end{array}} 

\newcommand{\vace}{ {\color{gray} \phantom{1} 1} } 

\DeclareMathOperator{\RC}{RC} 
\DeclareMathOperator{\wt}{wt} 
\DeclareMathOperator{\cc}{cc} 
\DeclareMathOperator{\id}{id} 

\newcommand{\hwRC}{\RC^{\nabla}} 

\newcommand{\ZZ}{\mathbb{Z}}

\newcommand{\bze}{\overline{0}}
\newcommand{\bon}{\overline{1}}
\newcommand{\btw}{\overline{2}}
\newcommand{\bth}{\overline{3}}
\newcommand{\bfo}{\overline{4}}
\newcommand{\bfive}{\overline{5}}
\newcommand{\bsix}{\overline{6}}
\newcommand{\bseven}{\overline{7}}

\newcommand{\bn}{\overline{n}}

\definecolor{darkred}{rgb}{0.7,0,0} 
\newcommand{\defn}[1]{{\color{darkred}\emph{#1}}} 

\usepackage{listings}
\lstdefinelanguage{Sage}[]{Python}
{morekeywords={False,sage,True},sensitive=true}
\definecolor{dblackcolor}{rgb}{0.0,0.0,0.0}
\definecolor{dbluecolor}{rgb}{0.01,0.02,0.7}
\definecolor{dgreencolor}{rgb}{0.2,0.4,0.0}
\definecolor{dgraycolor}{rgb}{0.30,0.3,0.30}

\usepackage{xparse}

\makeatletter
\protected\def\specialmergetwolists{%
  \begingroup
  \@ifstar{\def\cnta{1}\@specialmergetwolists}
    {\def\cnta{0}\@specialmergetwolists}%
}
\def\@specialmergetwolists#1#2#3#4{%
  \def\tempa##1##2{%
    \edef##2{%
      \ifnum\cnta=\@ne\else\expandafter\@firstoftwo\fi
      \unexpanded\expandafter{##1}%
    }%
  }%
  \tempa{#2}\tempb\tempa{#3}\tempa
  \def\cnta{0}\def#4{}%
  \foreach \x in \tempb{%
    \xdef\cnta{\the\numexpr\cnta+1}%
    \gdef\cntb{0}%
    \foreach \y in \tempa{%
      \xdef\cntb{\the\numexpr\cntb+1}%
      \ifnum\cntb=\cnta\relax
        \xdef#4{#4\ifx#4\empty\else,\fi\x#1\y}%
        \breakforeach
      \fi
    }%
  }%
  \endgroup
}
\makeatother

\usepackage{xparse}
\DeclareDocumentCommand\rpp{ m m g }{
	\foreach \x [count=\s from 1] in {#1}{
	        {\ifnum\s=1
	                \draw (0,-\s)--(\x,-\s);
	                \fi}
	   \draw (0,-\s-1) to (\x,-\s-1);
	   \foreach \y in {0, ..., \x} {\draw (\y,-\s)--(\y,-\s-1);}
	}
	\specialmergetwolists{/}{#1}{#2}\ziplist
	\foreach \x/\y [count=\yi from 1] in \ziplist{
	    \node[anchor=west,font=\small] at (\x,-\yi - .5) {$\y$};
	}
	\IfValueT {#3}
	{\foreach \z [count=\zi from 1] in {#3} {\node[anchor=east,font=\small] at (0,-\zi - .5) {$\z$};}}
	{}
}

\theoremstyle{plain}
\newtheorem{thm}{Theorem}[section]
\newtheorem{lemma}[thm]{Lemma}
\newtheorem{conj}[thm]{Conjecture}
\newtheorem{prop}[thm]{Proposition}

\theoremstyle{definition}
\newtheorem{dfn}[thm]{Definition}
\newtheorem{ex}[thm]{Example}
\newtheorem{remark}[thm]{Remark}
\numberwithin{equation}{section}

\usepackage[colorinlistoftodos]{todonotes}

\begin{document}
\title[Uniform approach to SCA using rigged configurations]{A uniform approach to soliton cellular automata using rigged configurations}

\author[X.~Liu]{Xuan Liu}
\address[X.~Liu]{School of Mathematics, University of Minnesota, 206 Church St.\ SE, Minneapolis, MN 55455}
\curraddr[X.~Liu]{Department of Statistics, North Carolina State University, 2311 Stinson Dr., Raleigh, NC 27695}
\email{xliu65@ncsu.edu} 
\urladdr{https://sites.google.com/ncsu.edu/xliu65/home}

\author[T.~Scrimshaw]{Travis Scrimshaw}
\address[T.~Scrimshaw]{School of Mathematics and Physics, University of Queensland, St.\ Lucia, QLD 4072, Australia}
\email{tcscrims@gmail.com}
\urladdr{https://people.smp.uq.edu.au/TravisScrimshaw/}


\thanks{TS was partially supported by the National Science Foundation RTG grant NSF/DMS-1148634.}

\begin{abstract}
For soliton cellular automata, we give a uniform description and proofs of the solitons, the scattering rule of two solitons, and the phase shift using rigged configurations in a number of special cases.
In particular, we prove these properties for the soliton cellular automata using $B^{r,1}$ when $r$ is adjacent to $0$ in the Dynkin diagram or there is a Dynkin diagram automorphism sending $r$ to $0$.
\end{abstract}

\maketitle

\section{Introduction}
\label{sec:introduction}

A soliton cellular automaton (SCA) is a discrete dynamical system on a one-dimensional lattice that evolves according to a particular deterministic rule.
Some of the key properties of SCA are that they possess stable configurations called solitons, isolated solitons move proportional to their length, and the number of solitons and their lengths do not change after collisions.
A classic and, despite its simplicity, surprisingly rich example is the box-ball system of Takahashi and Satsuma~\cite{TS90}.

The box-ball system is an integrable nonlinear dynamical system and is related to the difference analog of the Lotka--Volterra equation under tropicalization~\cite{TNS96,TTMS96}.
It is also an ultradiscrete version of the Korteweg--de Vries (KdV) equation~\cite{Boussinesq1877, KdV1895}: a nonlinear partial differential equation that models shallow water waves in 1D (such as a thin channel).
Solutions of the KdV equation were shown to separate in solitary waves, where they retained their shape after interaction, by Kruskal and Zabusky~\cite{KZ64}.
These solitary waves are called solitons, and solitons in the box-ball system are the ultradiscrete analog.
The inverse scattering transform was constructed in~\cite{GGKM74} and applied to the KdV equation, showing $m$-soliton solutions exist and that the KdV equation is an exactly solvable model.

The next breakthrough in studying SCA came from using the theory of Kashiwara's crystal bases~\cite{K90, K91} and certain finite crystals for affine Kac--Moody algebras called Kirillov--Reshetikhin (KR) crystals~\cite{FOS09,HKOTY99,HKOTT02,JS10,KKMMNN92,KMOY07,OS08,Yamane98}.
SCA were reformulated using tensor products of KR crystals $(B^{r,1})^{\otimes \infty}$,\footnote{This could be given more generally as $\bigotimes_{i=1}^{\infty} B^{r_i,s_i}$ for any sequence $\{ (r_i, s_i) \in I_0 \times \ZZ_{\geq 0} \}_{i=1}^{\infty}$. However, we will not consider this level of generality.} where the time evolution was given using the combinatorial $R$-matrix with a carrier $B^{r,s}$ (typically for $s \gg 1$) and the invariants were described using the local energy function~\cite{binMohamad12,FOY00,HHIKTT01,HKOTY02,HKT00,MOW12,MW13,TNS96,Yamada04,Yamada07}.
Based on these results, it is conjectured that solitons are parameterized by ``decoupled'' KR crystals (removing nodes $0$ and $r$ from the Dynkin diagram and taking the appropriate affinization), the scattering rule is determined by the ``decoupled'' combinatorial $R$-matrix, and the phase shift is given by the local energy function for all SCA.
We note that the phase shift corresponds to the change in the coefficient of the null root $\delta$ if we consider the affinization crystal (we refer the reader to~\cite[Ch.~10]{HK02} for details) of ``decoupled'' KR crystals.
Additionally, the time evolution can also be described using the row-to-row transfer matrix of integrable 2D lattice models at $q = 0$.
In particular, the box-ball system is given using the KR crystal $B^{1,1}$ in type $A_1^{(1)}$.

SCA are also intimately connected to the Bethe ansatz~\cite{B31} of Heisenberg spin chains.
The connection comes from the fact that the Hamiltonian of the Heisenberg spin chain commutes with the row-to-row transfer matrix of the 2D lattice model and can be simultaneously diagonalized~\cite{KKMMNN91}.
The analysis of the 2D lattice model and the Bethe ansatz led to the $X = M$ conjecture~\cite{HKOTY99, HKOTT02}, but there is a more direct combinatorial interpretation.
Baxter's corner transfer matrix~\cite{B89} can be to solve the 2D lattice model, which naturally corresponds to classically highest weight elements in a tensor product of KR crystals.
In~\cite{KKR86, KR86}, Kerov, Kirillov, and Reshetikhin introduced combinatorial objects called rigged configurations to parameterize solutions to the Bethe ansatz and developed a bijection $\Phi$ between rigged configurations and classically highest weight elements in $\bigotimes_{i=1}^N B^{1,s_i}$ in type $A_n^{(1)}$.
The bijection $\Phi$ was later extended to an affine crystal isomorphism with the full tensor product in general (\textit{i.e.}, for $\bigotimes_{i=1}^N B^{r_i, s_i}$) in type $A_n^{(1)}$~\cite{DS06, KSS02, SW10}.

The SCA for type $A_n^{(1)}$ has similar dynamics to the box-ball system, but now there are $n$ colored balls (for $r = 1$).
In the bijection $\Phi$, the addition of vacuum states to the end of the tensor product does not change the rigged configuration, and thus we can extend $\Phi$ to be a bijection between rigged configurations and states of the SCA.
Moreover, the combinatorial $R$-matrix intertwines with the identity map on rigged configurations under $\Phi$, and thus time evolution acts by increasing the riggings $J_i^{(r)}$ by $i$.
Hence, the corresponding rigged configuration under $\Phi^{-1}$ encodes the action-angle variables of the SCA, and in particular, the partition $\nu^{(r)}$ of the rigged configuration encodes the sizes of the solitons (with no interactions)~\cite{KNTW18,KS09,KOSTY06,Takagi05}.
The scattering of two solitons has been identified with the affine combinatorial $R$-matrices~\cite{KSY07,Sakamoto08}, which ensures the Yang--Baxter property and is also valid for all intermediate states during multiple scatterings.
In~\cite{KSY07}, it was shown that $\Phi$ can be described by a tropicalization of the $\tau$ function from the Kadomtsev--Petviashvili (KP) hierarchy (we refer the reader to~\cite{JM83} for details).
Additionally, time evolution is a tropicalization of the nonautonomous discrete KP equation~\cite{HHIKTT01}.

In type $A_n^{(1)}$, there is an intermediate geometric (or rational) model between the KdV equation and the box-ball system introduced by Hirota~\cite{Hirota81} called the discrete KdV equation.
For the discrete KdV equation, Kashiwara's crystals are replaced by geometric crystals~\cite{BK00, BK07}; more specifically, with the affine geometric crystals of~\cite{KNO08,LP12}.
The geometric $R$-matrix was described by Yamada~\cite{Yamada01}, where the ring of geometric $R$-matrix invariants was studied in~\cite{LP13}.
This led to a conjectural description of a geometric version of $\Phi$ in~\cite{LPS15, Scrimshaw17II} using these invariants, where it is known that it tropicalizes to describe $\nu^{(1)}$~\cite{LPS15}.

Transitioning back to the more general setting, SCA have been well studied through explicit description of the action of the combinatorial $R$-matrix; \textit{e.g.},~\cite{binMohamad12,FOY00,HHIKTT01,HKOTY02,HKT00,MOW12,MW13,TNS96,Yamada04,Yamada07}.
Moreover, rigged configurations have been used to describe properties of SCA in types $A_n^{(1)}$~\cite{KOSTY06,KSY07,Sakamoto08} and $D_n^{(1)}$~\cite{KSY11,Sakamoto17}.
There exists an analogous (conjectural) bijection $\Phi$ for all types~\cite{DS06, JS10, KSS02, OSSS16, OS12, OSS17, OSS03, OSS03III, OSS03II, SchillingS15, SS2006, SW10, Scrimshaw15, Scrimshaw17}.
Thus, it is expected that similar results hold for all types; in particular, that rigged configurations and the (conjectural) bijection $\Phi$ can be applied to show properties about SCA.
The main goal of this paper is to do this in as general as possible with uniform methods.
In particular, our results give an interpretation of~\cite[App.~B]{HKOTT02} as ``decoupling'' rules on the level of rigged configurations by forgetting $\nu^{(r)}$ and instead using $\nu^{(r)}$ to describe the ``decoupled'' KR crystals.
Moreover, our results subsume the results on the scattering and phase shift of SCA (see Conjecture~\ref{conj:SCA}) from~\cite{binMohamad12,FOY00,HHIKTT01,HKOTY02,HKT00,MOW12,MW13,TNS96,Yamada04,Yamada07} as the requisite properties of $\Phi$ are known~\cite{KSS02,OSSS16,OSS17,Scrimshaw17}.
Our results also include other types that have not previously been considered; \textit{e.g.}, $E_{6,7}^{(1)}$ and $F_4^{(1)}$.

One key aspect of our results is that they are largely type-independent, typically relying on the properties of the KR crystals.
Moreover, the proofs are typically short and straightforward, relying on (expected) properties of $\Phi$ and utilizing the natural information of rigged configurations.
An important advantage to this approach is that we do not require an understanding of the often intricate combinatorial $R$-matrix as it (conjecturally) becomes the identity map on rigged configurations under $\Phi$.
In contrast, directly studying the SCA using KR crystals requires detailed descriptions of the combinatorial $R$-matrix or the evolution rules of~\cite{HKT:2001}, as well as doing multiple applications, in order to prove properties about the SCA, leading to complicated (and type-dependent) proofs.
See, \textit{e.g.},~\cite{binMohamad12,FOY00,HHIKTT01,HKOTY02,HKT00,MOW12,MW13,TNS96,Yamada04,Yamada07}.

We note that in order to give a uniform description and proofs of SCA, we require a formal, uniform definition of solitons and their length.
As far as the authors are aware, no such definition has been given in the literature.
Provided the conjectured properties of $\Phi$ are true, we mostly achieve this, but our description of length is somewhat \textit{ad-hoc}.
Our description of length requires a special considerations for elements not in the maximal component $B(\clfw_r) \subseteq B^{r,1}$ despite being entirely determined by the crystal.
It is based upon the (conjectural) algorithm for $\Phi$ and that $\nu^{(r)}$ should describe the lengths of the solitons (which naturally shows their speed corresponds to their length).
The latter property is expected as all rows of $\nu^{(r)}$ are connected with solitons of the KP hierarchy~\cite{KSY07}.
Hence, we are not certain our definition of length will fully generalize.
Specifically in type $C_n^{(1)}$, we can contrast this with the description of the solitons given in~\cite[Sec.~2.3]{HKOTY02} and~\cite[Sec.~3.2]{HKT00}, where a vacuum element gets bounded as part of the soliton.

Our approach also allows us to give a simple description of the phase shift, the shift to the left of the larger soliton after scattering compared to its movement without scattering, in terms of the vacancy numbers of the rigged configuration.
In particular, the phase shift is measuring the change in the vacancy numbers, where we discard any contribution from the tensor factors $(B^{r,1})^{\otimes \infty}$, of $\nu^{(r)}$ after adding a smaller soliton.
Furthermore, the larger class of SCA that we can examine from our results allows us to construct a number of examples where the phase shift is negative, a phenomenon only previously observed in types $D_4^{(3)}$~\cite{Yamada07} and $G_2^{(1)}$~\cite{MOW12}.
We construct such examples not just in nonexceptional types, but the simply-laced type $D_4^{(1)}$.
Moreover, our results suggest that a large class of SCA can exhibit negative phase shifts.
In addition, our results also connect the phase shift to the local energy function of the decoupled in certain special cases, recovering previous known results.

Rigged configurations and the bijection $\Phi$ are known to be well behaved under the virtualization map~\cite{OSS17, OSS03III, OSS03II, SchillingS15, Scrimshaw15, Scrimshaw17}, an embedding of a crystal of non-simply-laced type into one of simply-laced type.
Furthermore, it is known that an SCA constructed using $B^{1,1}$ in every nonexceptional type can be embedded in a type $D_n^{(1)}$ SCA~\cite{KTT04}.
Therefore, we expect that our results could be applied to obtain analogous embedding results for other SCA and SCA in exceptional types.

We note that some of the KR crystals we consider in this paper have not been shown to exist (more specifically, in the exceptional types).
Yet, this is implicit in our assumption that the combinatorial $R$-matrix corresponds to the identity map on rigged configurations under $\Phi$.
Thus, our results give further evidence that these KR crystals should exist.
Furthermore, they are additional evidence for some of the conjectural properties of KR crystals.

This paper is organized as follows.
In Section~\ref{sec:background}, we give the necessary background on crystals, KR crystals, SCA, and rigged configurations.
In Section~\ref{sec:solitons}, we describe solitons in a number of cases using the properties of the KR crystals.
In Section~\ref{sec:SCA_RC}, we give our main results under some natural conjectures, where we show rigged configurations encode the sizes of the solitons and give simple uniform proofs of scattering and the phase shift.
In Section~\ref{sec:summary}, we summarize the cases when our results are not based on any conjectures.

\section{Background}
\label{sec:background}

Let $\g$ be an affine Kac--Moody Lie algebra with index set $I$, Cartan matrix $(A_{ij})_{i,j \in I}$, simple roots $(\alpha_i)_{i \in I}$, fundamental weights $(\Lambda_i)_{i \in I}$, weight lattice $P$, simple coroots $(\alpha_i^{\vee})_{i \in I}$, and canonical pairing $\langle\ ,\ \rangle \colon P^{\vee} \times P \to \ZZ$ given by $\inner{\alpha_i^{\vee}}{\alpha_j} = A_{ij}$.
Let $U_q(\g)$ denote the corresponding (Drinfel'd--Jimbo) quantum group.
Define $t_i^{\vee} := \max(c_i^{\vee}/c_i, c_0)$, where $c_i$ and $c_i^{\vee}$ are the Kac and dual Kac labels respectively~\cite[Table Aff1-3]{kac90}.
We write $i \sim j$ if $A_{ij} \neq 0$ and $i \neq j$.

Let $\g_0$ (resp.~$\g_{0,r}$) denote the canonical semisimple Lie algebra given by the index set $I_0 = I \setminus \{0\}$ (resp.~$I_{0,r} = I \setminus \{0,r\}$) and quantum group $U_q(\g_0)$ (resp.~$U_q(\g_{0,r})$).
Let $\clfw_i$ and $\clsr_i$ denote the natural projection of $\Lambda_i$ and $\alpha_i$, respectively, onto the weight lattice $\overline{P}$ of $\g_0$.
Note that $(\clsr_i)_{i \in I_0}$ are the simple roots in $\g_0$.
Denote the fundamental weights of $\g_{0,r}$ by $\{\varpi_i\}_{i \in I_{0,r}}$.

We say an $r \in I_0$ is \defn{special} if it is in the orbit of $0$ in $I$ under Dynkin diagram automorphisms.
We say $r$ is \defn{minuscule} if it is special and $\g$ is of dual untwisted type (\textit{i.e.}, it is the dual type to an untwisted type).
In particular, if a $r$ node is minuscule, then $B(\clfw_r)$ is a minuscule $U_q(\g_0)$-representation.

\begin{figure}[t]
\begin{center}
\begin{tikzpicture}[scale=0.43, baseline=-20]
\draw (-0.8,0) node[anchor=east] {$E_6^{(1)}$};
\draw (0 cm,0) -- (8 cm,0);
\draw (4 cm, 0 cm) -- +(0,2 cm);
\draw (4 cm, 2 cm) -- +(0,2 cm);
\draw[fill=white] (4 cm, 4 cm) circle (.25cm) node[right=3pt]{$0$};
\draw[fill=white] (0 cm, 0 cm) circle (.25cm) node[below=4pt]{$1$};
\draw[fill=white] (2 cm, 0 cm) circle (.25cm) node[below=4pt]{$3$};
\draw[fill=white] (4 cm, 0 cm) circle (.25cm) node[below=4pt]{$4$};
\draw[fill=white] (6 cm, 0 cm) circle (.25cm) node[below=4pt]{$5$};
\draw[fill=white] (8 cm, 0 cm) circle (.25cm) node[below=4pt]{$6$};
\draw[fill=white] (4 cm, 2 cm) circle (.25cm) node[right=3pt]{$2$};
\begin{scope}[xshift=11.8cm]
\draw (-0.8,0) node[anchor=east] {$E_7^{(1)}$};
\draw (0 cm,0) -- (12 cm,0);
\draw (6 cm, 0 cm) -- +(0,2 cm);
\draw[fill=white] (0 cm, 0 cm) circle (.25cm) node[below=4pt]{$0$};
\draw[fill=white] (2 cm, 0 cm) circle (.25cm) node[below=4pt]{$1$};
\draw[fill=white] (4 cm, 0 cm) circle (.25cm) node[below=4pt]{$3$};
\draw[fill=white] (6 cm, 0 cm) circle (.25cm) node[below=4pt]{$4$};
\draw[fill=white] (8 cm, 0 cm) circle (.25cm) node[below=4pt]{$5$};
\draw[fill=white] (10 cm, 0 cm) circle (.25cm) node[below=4pt]{$6$};
\draw[fill=white] (12 cm, 0 cm) circle (.25cm) node[below=4pt]{$7$};
\draw[fill=white] (6 cm, 2 cm) circle (.25cm) node[right=3pt]{$2$};
\end{scope}
\end{tikzpicture}

\begin{tikzpicture}[scale=0.43]
\draw (-1,0) node[anchor=east] {$E_8^{(1)}$};
\draw (0 cm,0) -- (12 cm,0);
\draw (4 cm, 0 cm) -- +(0,2 cm);
\draw (12 cm,0) -- +(2 cm,0);
\draw[fill=white] (14 cm, 0 cm) circle (.25cm) node[below=4pt]{$0$};
\draw[fill=white] (0 cm, 0 cm) circle (.25cm) node[below=4pt]{$1$};
\draw[fill=white] (2 cm, 0 cm) circle (.25cm) node[below=4pt]{$3$};
\draw[fill=white] (4 cm, 0 cm) circle (.25cm) node[below=4pt]{$4$};
\draw[fill=white] (6 cm, 0 cm) circle (.25cm) node[below=4pt]{$5$};
\draw[fill=white] (8 cm, 0 cm) circle (.25cm) node[below=4pt]{$6$};
\draw[fill=white] (10 cm, 0 cm) circle (.25cm) node[below=4pt]{$7$};
\draw[fill=white] (12 cm, 0 cm) circle (.25cm) node[below=4pt]{$8$};
\draw[fill=white] (4 cm, 2 cm) circle (.25cm) node[right=3pt]{$2$};
\end{tikzpicture}

\begin{tikzpicture}[scale=0.43,baseline=10]
\draw (-1,0) node[anchor=east] {$F_4^{(1)}$};
\draw (0 cm,0) -- (2 cm,0);
{
\pgftransformxshift{2 cm}
\draw (0 cm,0) -- (2 cm,0);
\draw (2 cm, 0.1 cm) -- +(2 cm,0);
\draw (2 cm, -0.1 cm) -- +(2 cm,0);
\draw (4.0 cm,0) -- +(2 cm,0);
\draw[shift={(3.2, 0)}, rotate=0] (135 : 0.45cm) -- (0,0) -- (-135 : 0.45cm);
\draw[fill=white] (0 cm, 0 cm) circle (.25cm) node[below=4pt]{$1$};
\draw[fill=white] (2 cm, 0 cm) circle (.25cm) node[below=4pt]{$2$};
\draw[fill=white] (4 cm, 0 cm) circle (.25cm) node[below=4pt]{$3$};
\draw[fill=white] (6 cm, 0 cm) circle (.25cm) node[below=4pt]{$4$};
}
\draw[fill=white] (0 cm, 0 cm) circle (.25cm) node[below=4pt]{$0$};

\begin{scope}[xshift=15cm]
\draw (-1,0) node[anchor=east] {$E_6^{(2)}$};
\draw (0 cm,0) -- (2 cm,0);
{
\pgftransformxshift{2 cm}
\draw (0 cm,0) -- (2 cm,0);
\draw (2 cm, 0.1 cm) -- +(2 cm,0);
\draw (2 cm, -0.1 cm) -- +(2 cm,0);
\draw (4.0 cm,0) -- +(2 cm,0);
\draw[shift={(2.8, 0)}, rotate=180] (135 : 0.45cm) -- (0,0) -- (-135 : 0.45cm);
\draw[fill=white] (0 cm, 0 cm) circle (.25cm) node[below=4pt]{$1$};
\draw[fill=white] (2 cm, 0 cm) circle (.25cm) node[below=4pt]{$2$};
\draw[fill=white] (4 cm, 0 cm) circle (.25cm) node[below=4pt]{$3$};
\draw[fill=white] (6 cm, 0 cm) circle (.25cm) node[below=4pt]{$4$};
}
\draw[fill=white] (0 cm, 0 cm) circle (.25cm) node[below=4pt]{$0$};
\end{scope}
\end{tikzpicture}

\begin{tikzpicture}[scale=0.43,baseline=20]
\begin{scope}
\draw (-1,0) node[anchor=east] {$G_2^{(1)}$};
\draw (2 cm,0) -- (4.0 cm,0);
\draw (0, 0.15 cm) -- +(2 cm,0);
\draw (0, -0.15 cm) -- +(2 cm,0);
\draw (0,0) -- (2 cm,0);
\draw (0, 0.15 cm) -- +(2 cm,0);
\draw (0, -0.15 cm) -- +(2 cm,0);
\draw[shift={(0.8, 0)}, rotate=180] (135 : 0.45cm) -- (0,0) -- (-135 : 0.45cm);
\draw[fill=white] (0 cm, 0 cm) circle (.25cm) node[below=4pt]{$1$};
\draw[fill=white] (2 cm, 0 cm) circle (.25cm) node[below=4pt]{$2$};
\draw[fill=white] (4 cm, 0 cm) circle (.25cm) node[below=4pt]{$0$};
\end{scope}

\begin{scope}[xshift=15cm]
\draw (-1,0) node[anchor=east] {$D_4^{(3)}$};
\draw (2 cm,0) -- (4.0 cm,0);
\draw (0, 0.15 cm) -- +(2 cm,0);
\draw (0, -0.15 cm) -- +(2 cm,0);
\draw (0,0) -- (2 cm,0);
\draw (0, 0.15 cm) -- +(2 cm,0);
\draw (0, -0.15 cm) -- +(2 cm,0);
\draw[shift={(1.2, 0)}, rotate=0] (135 : 0.45cm) -- (0,0) -- (-135 : 0.45cm);
\draw[fill=white] (0 cm, 0 cm) circle (.25cm) node[below=4pt]{$2$};
\draw[fill=white] (2 cm, 0 cm) circle (.25cm) node[below=4pt]{$1$};
\draw[fill=white] (4 cm, 0 cm) circle (.25cm) node[below=4pt]{$0$};
\end{scope}
\end{tikzpicture}
\end{center}
\caption{Dynkin diagrams for the exceptional affine types with labels from~\cite{kac90}.}
\label{figure:exceptional_types}
\end{figure}
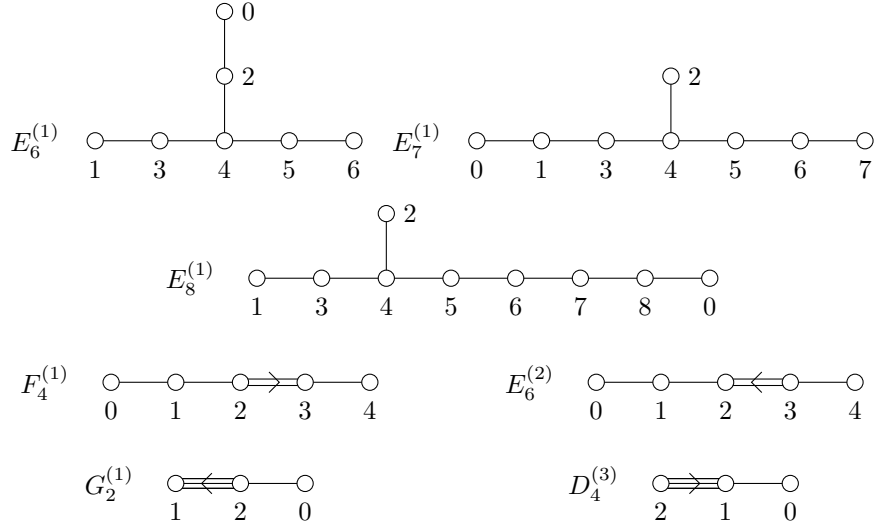

An \defn{abstract $U_q(\g)$-crystal} is a set $B$ with \defn{crystal operators} $e_i, f_i \colon B \to B \sqcup \{0\}$, for $i \in I$, and \defn{weight function} $\wt \colon B \to P$ that satisfy the following conditions:
\begin{itemize}
\item[(1)] $\varphi_i(b) = \varepsilon_i(b) + \inner{h_i}{\wt(b)}$ for all $b \in B$ and $a \in I$,
\item[(2)] $f_i b = b'$ if and only if $b = e_i b'$ for $b, b' \in B$ and $a \in I$,
\end{itemize}
where $\varepsilon_i, \varphi_i \colon  B \to \ZZ_{\geq 0}$ are the statistics
\[
\varepsilon_i(b) := \max \{ k \mid e_i^k b \neq 0 \},
\qquad \qquad \varphi_i(b) := \max \{ k \mid f_i^k b \neq 0 \}.
\]
Hence, we can express an entire $i$-string through an element $b \in B$ diagrammatically by
\[
e_i^{\varepsilon_i(b)}b \xrightarrow[\hspace{12pt}]{i}
\cdots \xrightarrow[\hspace{12pt}]{i}
e_i^2 b \xrightarrow[\hspace{12pt}]{i}
e_i b \xrightarrow[\hspace{12pt}]{i}
b \xrightarrow[\hspace{12pt}]{i}
f_i b \xrightarrow[\hspace{12pt}]{i}
f_i^2 b \xrightarrow[\hspace{12pt}]{i}
\cdots \xrightarrow[\hspace{12pt}]{i}
f_i^{\varphi_i(b)}b.
\]
More generally, we identify the abstract crystal $B$ with its crystal graph, an $I$-edge-colored weighted directed graph, where there is an edge $b \to b'$ if $f_i b = b'$.
We say an element $b \in B$ is \defn{highest weight} if $e_i b = 0$ for all $a \in I$.

\begin{remark}
The definition of an abstract crystal given in this paper is sometimes called a \defn{regular} or \defn{seminormal} abstract crystal in the literature.
\end{remark}

We call an abstract $U_q(\g)$-crystal $B$ a \defn{$U_q(\g)$-crystal} if $B$ is the crystal basis of some $U_q(\g)$-module.
Kashiwara has shown that the irreducible highest weight module $V(\lambda)$ admits a crystal basis~\cite{K91}.
We denote this crystal basis by $B(\lambda)$, and let $u_{\lambda} \in B(\lambda)$ denote the unique highest weight element, which is the unique element of weight $\lambda$.
Since the crystal graph of $B(\lambda)$ is acyclic, we regard $B(\clfw_r)$ as a poset with $b \leq b'$ if there exists a path $f_{i_L} \cdots f_{i_1} b = b'$.
In particular, the crystal graph is the Hasse diagram of this poset.

We define the \defn{tensor product} of abstract $U_q(\g)$-crystals $B_1$ and $B_2$ as the crystal $B_2 \otimes B_1$ that is the Cartesian product $B_2 \times B_1$ with the crystal structure
\begin{align*}
e_i(b_2 \otimes b_1) & = \begin{cases}
e_i b_2 \otimes b_1 & \text{if } \varepsilon_i(b_2) > \varphi_i(b_1), \\
b_2 \otimes e_i b_1 & \text{if } \varepsilon_i(b_2) \leq \varphi_i(b_1),
\end{cases}
\\ f_i(b_2 \otimes b_1) & = \begin{cases}
f_i b_2 \otimes b_1 & \text{if } \varepsilon_i(b_2) \geq \varphi_i(b_1), \\
b_2 \otimes f_i b_1 & \text{if } \varepsilon_i(b_2) < \varphi_i(b_1),
\end{cases}
\\ \varepsilon_i(b_2 \otimes b_1) & = \max(\varepsilon_i(b_1), \varepsilon_i(b_2) - \inner{h_i}{\wt(b_1)}),
\\ \varphi_i(b_2 \otimes b_1) & = \max(\varphi_i(b_2), \varphi_i(b_1) + \inner{h_i}{\wt(b_2)}),
\\ \wt(b_2 \otimes b_1) & = \wt(b_2) + \wt(b_1).
\end{align*}

\begin{remark}
Our tensor product convention is opposite of Kashiwara~\cite{K91}.
\end{remark}

Consider $U_q(\g)$-crystals $B_1, \dotsc, B_L$.
The action of the crystal operators on the tensor product $B = B_L \otimes \cdots \otimes B_2 \otimes B_1$ can be computed by the \defn{signature rule}.
Let $b = b_L \otimes \cdots \otimes b_2 \otimes b_1 \in B$, and for $i \in I$, we write
\[
\underbrace{-\cdots-}_{\varphi_i(b_L)}\
\underbrace{+\cdots+}_{\varepsilon_i(b_L)}\
\cdots\
\underbrace{-\cdots-}_{\varphi_i(b_1)}\
\underbrace{+\cdots+}_{\varepsilon_i(b_1)}\ .
\]
Then by successively deleting any $+-$-pairs (in that order) in the above sequence, we obtain a sequence
\[
\sgn_i(b) :=
\underbrace{-\cdots-}_{\varphi_i(b)}\
\underbrace{+\cdots+}_{\varepsilon_i(b)}
\]
called the \defn{reduced signature}.
Suppose $1 \leq j_-, j_+ \leq L$ are such that $b_{j_-}$ contributes the rightmost $-$ in $\sgn_i(b)$ and $b_{j_+}$ contributes the leftmost $+$ in $\sgn_i(b)$. 
Then, we have
\begin{align*}
e_i b &= b_L \otimes \cdots \otimes b_{j_++1} \otimes e_ib_{j_+} \otimes b_{j_+-1} \otimes \cdots \otimes b_1, \\
f_i b &= b_L \otimes \cdots \otimes b_{j_-+1} \otimes f_ib_{j_-} \otimes b_{j_--1} \otimes \cdots \otimes b_1.
\end{align*}

Let $B_1$ and $B_2$ be two abstract $U_q(\g)$-crystals.
A \defn{crystal morphism} $\psi \colon B_1 \to B_2$ is a map $B_1 \sqcup \{0\} \to B_2 \sqcup \{0\}$ with $\psi(0) = 0$ such that the following properties hold for all $b \in B_1$ and $i \in I$:
\begin{itemize}
\item[(1)] If $\psi(b) \in B_2$, then $\wt\bigl(\psi(b)\bigr) = \wt(b)$, $\varepsilon_i\bigl(\psi(b)\bigr) = \varepsilon_i(b)$, and $\varphi_i\bigl(\psi(b)\bigr) = \varphi_i(b)$.
\item[(2)] We have $\psi(e_i b) = e_i \psi(b)$ if $\psi(e_i b) \neq 0$ and $e_i \psi(b) \neq 0$.
\item[(3)] We have $\psi(f_i b) = f_i \psi(b)$ if $\psi(f_i b) \neq 0$ and $f_i \psi(b) \neq 0$.
\end{itemize}
An \defn{embedding} (resp.~\defn{isomorphism}) is a crystal morphism such that the induced map $B_1 \sqcup \{0\} \to B_2 \sqcup \{0\}$ is an embedding (resp.~bijection).

Let $r$ be a minuscule node.
Following~\cite{JS10, Scrimshaw17}, we first note that $b \in B(\clfw_r)$ is determined by the subset
\[
X_b := \{i \mid \varphi_i(b) = 1\} \sqcup \{\overline{\imath} \mid \varepsilon_i(b) = 1\}
\]
where $i \in I$.
To ease notation, we will write this simply as a word, which we call the \defn{minuscule word} of $b$.
See Figure~\ref{fig:crystal_graphs} for two examples.
For $r \sim 0$, we decompose the crystal $B(\clfw_r) = \{x_{\alpha}\} \sqcup \{y_i\}$, where
\begin{itemize}
\item $x_{\alpha}$ is the unique element weight $\alpha$, which is a root of the root system of $\g$;
\item $y_i$ satisfies $\varepsilon_j(y_i) = \varphi_j(y_i) = \delta_{ij}$ (note that $\wt(y_i) = 0$).
\end{itemize}
See Figure~\ref{fig:adjoint_ex_crystal} for an example.
We will also represent elements of $B(\clfw_1)$ in nonexceptional types by the common $1, \dotsc, \bon$ from~\cite{KM94,KN94}.

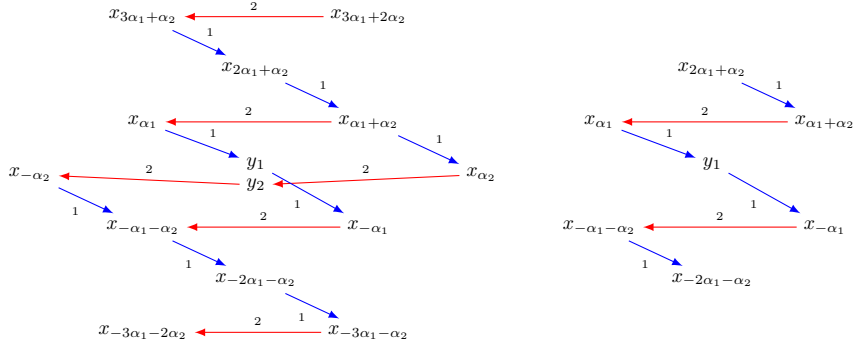
\begin{figure}
\[
\begin{tikzpicture}[>=latex,line join=bevel,xscale=1.5,yscale=0.7,every node/.style={scale=0.75},baseline=0]
\node (3a2b) at (1,3) {$x_{3\alpha_1+2\alpha_2}$};
\node (3ab) at (-1,3) {$x_{3\alpha_1+\alpha_2}$};
\node (2ab) at (0,2) {$x_{2\alpha_1+\alpha_2}$};
\node (ab) at (1,1) {$x_{\alpha_1+\alpha_2}$};
\node (a) at (-1,1) {$x_{\alpha_1}$};
\node (b) at (2,0) {$x_{\alpha_2}$};
\node (y1) at (0,0.2) {$y_1$};
\node (y2) at (0,-0.2) {$y_2$};
\node (mb) at (-2,0) {$x_{-\alpha_2}$};
\node (ma) at (1,-1) {$x_{-\alpha_1}$};
\node (mab) at (-1,-1) {$x_{-\alpha_1-\alpha_2}$};
\node (m2ab) at (0,-2) {$x_{-2\alpha_1-\alpha_2}$};
\node (m3ab) at (1,-3) {$x_{-3\alpha_1-\alpha_2}$};
\node (m3a2b) at (-1,-3) {$x_{-3\alpha_1-2\alpha_2}$};
\draw[->,red] (3a2b) -- node[midway,above,black] {\tiny $2$} (3ab);
\draw[->,blue] (3ab) -- node[midway,above right,black] {\tiny $1$} (2ab);
\draw[->,blue] (2ab) -- node[midway,above right,black] {\tiny $1$} (ab);
\draw[->,red] (ab) -- node[midway,above,black] {\tiny $2$} (a);
\draw[->,blue] (ab) -- node[midway,above right,black] {\tiny $1$} (b);
\draw[->,blue] (a) -- node[midway,above right,black] {\tiny $1$} (y1);
\draw[->,red] (b) -- node[midway,above,black] {\tiny $2$} (y2);
\draw[->,blue] (y1) -- node[midway,below left,black] {\tiny $1$} (ma);
\draw[->,red] (y2) -- node[midway,above,black] {\tiny $2$} (mb);
\draw[->,red] (ma) -- node[midway,above,black] {\tiny $2$} (mab);
\draw[->,blue] (mb) -- node[midway,below left,black] {\tiny $1$} (mab);
\draw[->,blue] (mab) -- node[midway,below left,black] {\tiny $1$} (m2ab);
\draw[->,blue] (m2ab) -- node[midway,below left,black] {\tiny $1$} (m3ab);
\draw[->,red] (m3ab) -- node[midway,above,black] {\tiny $2$} (m3a2b);
\end{tikzpicture}
\qquad
\begin{tikzpicture}[>=latex,line join=bevel,xscale=1.5,yscale=0.7,every node/.style={scale=0.75},baseline=0]
\node (2ab) at (0,2) {$x_{2\alpha_1+\alpha_2}$};
\node (ab) at (1,1) {$x_{\alpha_1+\alpha_2}$};
\node (a) at (-1,1) {$x_{\alpha_1}$};
\node (y1) at (0,0.2) {$y_1$};
\node (ma) at (1,-1) {$x_{-\alpha_1}$};
\node (mab) at (-1,-1) {$x_{-\alpha_1-\alpha_2}$};
\node (m2ab) at (0,-2) {$x_{-2\alpha_1-\alpha_2}$};
\draw[->,blue] (2ab) -- node[midway,above right,black] {\tiny $1$} (ab);
\draw[->,red] (ab) -- node[midway,above,black] {\tiny $2$} (a);
\draw[->,blue] (a) -- node[midway,above right,black] {\tiny $1$} (y1);
\draw[->,blue] (y1) -- node[midway,below left,black] {\tiny $1$} (ma);
\draw[->,red] (ma) -- node[midway,above,black] {\tiny $2$} (mab);
\draw[->,blue] (mab) -- node[midway,below left,black] {\tiny $1$} (m2ab);
\end{tikzpicture}
\]
\caption{The crystal $B(\clfw_2) \subseteq B^{2,1}$ in type $G_2^{(1)}$ (left) and $B(\clfw_1) \subseteq B^{1,1}$ in type $D_4^{(3)}$ (right) (this is the ``little'' adjoint representation of $G_2$).}
\label{fig:adjoint_ex_crystal}
\end{figure}

Next, we want to characterize the elements in the unique connected component $B(s\clfw_r) \subseteq B(\clfw_r)^{\otimes s}$.

\begin{prop}[{\cite[Prop.~7.29]{Scrimshaw17}}]
\label{prop:minuscule_description}
Let $r$ be such that $\clfw_r$ is a minuscule weight. Then $B(s\clfw_r) \subseteq B(\clfw_r)^{\otimes s}$ is characterized by
\[
\{ b_1 \otimes \cdots \otimes b_s \mid b_1 \leq \cdots \leq b_s \}.
\]
\end{prop}

We will write the elements of $B(s\clfw_r)$ as the single row tableaux
\[
\begin{array}{|c|c|c|c|}
\hline b_1 & b_2 & \cdots & b_s
\\\hline 
\end{array}\ .
\]
This was also used to describe the elements of $B(s\clfw_1)$ for types $A_n$, $B_n$, $C_n$, and $D_n$ in~\cite{KN94} and for type $G_2$ in~\cite{KM94}.
However, when there is an element $y_1 \in B(\clfw_1)$ of weight $0$ (and is the unique such element and denoted $0$ in~\cite{KM94,KN94}), it can only appear once in a tableau.
For example, a typical element in $B(s\clfw_1)$ in type $B_n$ is of the form:
\[
\begin{array}{|c|c|c|c|c|c|c|c|c|c|c|c|c|}
\hline
\raisebox{-1pt}{$1$} & \cdots & \raisebox{-1pt}{$1$} & \raisebox{-1pt}{$2$} & \cdots & \raisebox{-1pt}{$n$} & \raisebox{-1pt}{$0$} & \raisebox{-1pt}{$\bn$} & \cdots & \raisebox{-1pt}{$\btw$} & \raisebox{-1pt}{$\bon$} & \cdots & \raisebox{-1pt}{$\bon$}
\\\hline
\end{array}
\]

\afterpage{\clearpage}
\begin{figure}[p]
\[
\begin{tikzpicture}[>=latex,line join=bevel,xscale=0.9,yscale=0.45,every node/.style={scale=0.65}]
\node (node_26) at (77.0bp,9.5bp) [draw,draw=none] {$\overline{6}$};
  \node (node_24) at (77.0bp,155.5bp) [draw,draw=none] {$\overline{4}5$};
  \node (node_25) at (77.0bp,82.5bp) [draw,draw=none] {$\overline{5}6$};
  \node (node_22) at (103.0bp,301.5bp) [draw,draw=none] {$\overline{1}\overline{2}3$};
  \node (node_23) at (77.0bp,228.5bp) [draw,draw=none] {$\overline{2}\overline{3}4$};
  \node (node_20) at (50.0bp,447.5bp) [draw,draw=none] {$\overline{1}\overline{5}4$};
  \node (node_21) at (61.0bp,374.5bp) [draw,draw=none] {$\overline{1}\overline{4}23$};
  \node (node_9) at (130.0bp,812.5bp) [draw,draw=none] {$\overline{6}2$};
  \node (node_8) at (39.0bp,666.5bp) [draw,draw=none] {$\overline{3}16$};
  \node (node_7) at (50.0bp,739.5bp) [draw,draw=none] {$\overline{4}36$};
  \node (node_6) at (61.0bp,812.5bp) [draw,draw=none] {$\overline{2}\overline{5}46$};
  \node (node_5) at (103.0bp,885.5bp) [draw,draw=none] {$\overline{5}26$};
  \node (node_4) at (50.0bp,885.5bp) [draw,draw=none] {$\overline{2}5$};
  \node (node_3) at (77.0bp,958.5bp) [draw,draw=none] {$\overline{4}25$};
  \node (node_2) at (77.0bp,1031.5bp) [draw,draw=none] {$\overline{3}4$};
  \node (node_1) at (77.0bp,1104.5bp) [draw,draw=none] {$\overline{1}3$};
  \node (node_0) at (77.0bp,1176.5bp) [draw,draw=none] {$1$};
  \node (node_19) at (130.0bp,374.5bp) [draw,draw=none] {$\overline{2}1$};
  \node (node_18) at (50.0bp,301.5bp) [draw,draw=none] {$\overline{3}2$};
  \node (node_17) at (119.0bp,447.5bp) [draw,draw=none] {$\overline{4}12$};
  \node (node_16) at (135.0bp,593.5bp) [draw,draw=none] {$\overline{5}3$};
  \node (node_15) at (15.0bp,593.5bp) [draw,draw=none] {$\overline{1}6$};
  \node (node_14) at (110.0bp,520.5bp) [draw,draw=none] {$\overline{3}\overline{5}14$};
  \node (node_13) at (39.0bp,520.5bp) [draw,draw=none] {$\overline{1}\overline{6}5$};
  \node (node_12) at (75.0bp,593.5bp) [draw,draw=none] {$\overline{3}\overline{6}15$};
  \node (node_11) at (110.0bp,666.5bp) [draw,draw=none] {$\overline{4}\overline{6}35$};
  \node (node_10) at (119.0bp,739.5bp) [draw,draw=none] {$\overline{2}\overline{6}4$};
  \draw [black,->] (node_20) ..controls (48.848bp,428.85bp) and (48.615bp,414.29bp)  .. (51.0bp,402.0bp) .. controls (51.543bp,399.2bp) and (52.354bp,396.31bp)  .. (node_21);
  \definecolor{strokecol}{rgb}{0.0,0.0,0.0};
  \pgfsetstrokecolor{strokecol}
  \draw (59.5bp,411.0bp) node {$4$};
  \draw [red,->] (node_3) ..controls (71.052bp,943.57bp) and (68.302bp,936.95bp)  .. (66.0bp,931.0bp) .. controls (62.662bp,922.38bp) and (59.164bp,912.73bp)  .. (node_4);
  \draw (74.5bp,922.0bp) node {$2$};
  \draw [black,->] (node_6) ..controls (54.699bp,797.8bp) and (52.199bp,791.18bp)  .. (51.0bp,785.0bp) .. controls (49.379bp,776.65bp) and (48.968bp,767.25bp)  .. (node_7);
  \draw (59.5bp,776.0bp) node {$4$};
  \draw [black,->] (node_5) ..controls (110.44bp,864.94bp) and (117.61bp,846.08bp)  .. (node_9);
  \draw (128.5bp,849.0bp) node {$6$};
  \draw [blue,->] (node_14) ..controls (93.106bp,499.51bp) and (76.272bp,479.59bp)  .. (node_20);
  \draw (96.5bp,484.0bp) node {$1$};
  \draw [black,->] (node_24) ..controls (77.0bp,135.04bp) and (77.0bp,116.45bp)  .. (node_25);
  \draw (85.5bp,119.0bp) node {$5$};
  \draw [blue,->] (node_8) ..controls (23.853bp,652.07bp) and (18.634bp,645.82bp)  .. (16.0bp,639.0bp) .. controls (12.86bp,630.87bp) and (12.292bp,621.27bp)  .. (node_15);
  \draw (24.5bp,630.0bp) node {$1$};
  \draw [black,->] (node_10) ..controls (116.53bp,719.04bp) and (114.18bp,700.45bp)  .. (node_11);
  \draw (124.5bp,703.0bp) node {$4$};
  \draw [red,->] (node_9) ..controls (126.99bp,792.04bp) and (124.1bp,773.45bp)  .. (node_10);
  \draw (134.5bp,776.0bp) node {$2$};
  \draw [blue,->] (node_0) ..controls (77.0bp,1157.4bp) and (77.0bp,1138.8bp)  .. (node_1);
  \draw (85.5bp,1141.0bp) node {$1$};
  \draw [blue,->] (node_12) ..controls (65.026bp,572.83bp) and (55.331bp,553.71bp)  .. (node_13);
  \draw (70.5bp,557.0bp) node {$1$};
  \draw [green,->] (node_22) ..controls (95.835bp,280.94bp) and (88.931bp,262.08bp)  .. (node_23);
  \draw (102.5bp,265.0bp) node {$3$};
  \draw [black,->] (node_15) ..controls (12.307bp,574.76bp) and (11.433bp,559.82bp)  .. (16.0bp,548.0bp) .. controls (17.399bp,544.38bp) and (19.528bp,540.92bp)  .. (node_13);
  \draw (24.5bp,557.0bp) node {$6$};
  \draw [green,->] (node_16) ..controls (128.11bp,572.94bp) and (121.47bp,554.08bp)  .. (node_14);
  \draw (134.5bp,557.0bp) node {$3$};
  \draw [black,->] (node_3) ..controls (84.165bp,937.94bp) and (91.069bp,919.08bp)  .. (node_5);
  \draw (102.5bp,922.0bp) node {$5$};
  \draw [black,->] (node_13) ..controls (37.848bp,501.85bp) and (37.615bp,487.29bp)  .. (40.0bp,475.0bp) .. controls (40.543bp,472.2bp) and (41.354bp,469.31bp)  .. (node_20);
  \draw (48.5bp,484.0bp) node {$5$};
  \draw [red,->] (node_5) ..controls (91.3bp,864.72bp) and (79.833bp,845.34bp)  .. (node_6);
  \draw (95.5bp,849.0bp) node {$2$};
  \draw [black,->] (node_12) ..controls (84.697bp,572.83bp) and (94.122bp,553.71bp)  .. (node_14);
  \draw (105.5bp,557.0bp) node {$5$};
  \draw [black,->] (node_7) ..controls (66.894bp,718.51bp) and (83.728bp,698.59bp)  .. (node_11);
  \draw (96.5bp,703.0bp) node {$6$};
  \draw [black,->] (node_8) ..controls (48.974bp,645.83bp) and (58.669bp,626.71bp)  .. (node_12);
  \draw (70.5bp,630.0bp) node {$6$};
  \draw [black,->] (node_6) ..controls (77.33bp,791.51bp) and (93.604bp,771.59bp)  .. (node_10);
  \draw (105.5bp,776.0bp) node {$6$};
  \draw [black,->] (node_4) ..controls (48.848bp,866.85bp) and (48.615bp,852.29bp)  .. (51.0bp,840.0bp) .. controls (51.543bp,837.2bp) and (52.354bp,834.31bp)  .. (node_6);
  \draw (59.5bp,849.0bp) node {$5$};
  \draw [black,->] (node_2) ..controls (77.0bp,1011.0bp) and (77.0bp,992.45bp)  .. (node_3);
  \draw (85.5bp,995.0bp) node {$4$};
  \draw [green,->] (node_11) ..controls (100.3bp,645.83bp) and (90.878bp,626.71bp)  .. (node_12);
  \draw (105.5bp,630.0bp) node {$3$};
  \draw [red,->] (node_21) ..controls (72.7bp,353.72bp) and (84.167bp,334.34bp)  .. (node_22);
  \draw (95.5bp,338.0bp) node {$2$};
  \draw [blue,->] (node_19) ..controls (122.56bp,353.94bp) and (115.39bp,335.08bp)  .. (node_22);
  \draw (128.5bp,338.0bp) node {$1$};
  \draw [green,->] (node_1) ..controls (77.0bp,1084.0bp) and (77.0bp,1065.4bp)  .. (node_2);
  \draw (85.5bp,1068.0bp) node {$3$};
  \draw [green,->] (node_7) ..controls (43.699bp,724.8bp) and (41.199bp,718.18bp)  .. (40.0bp,712.0bp) .. controls (38.379bp,703.65bp) and (37.968bp,694.25bp)  .. (node_8);
  \draw (48.5bp,703.0bp) node {$3$};
  \draw [red,->] (node_17) ..controls (122.01bp,427.04bp) and (124.9bp,408.45bp)  .. (node_19);
  \draw (134.5bp,411.0bp) node {$2$};
  \draw [black,->] (node_11) ..controls (116.89bp,645.94bp) and (123.53bp,627.08bp)  .. (node_16);
  \draw (134.5bp,630.0bp) node {$5$};
  \draw [black,->] (node_23) ..controls (77.0bp,208.04bp) and (77.0bp,189.45bp)  .. (node_24);
  \draw (85.5bp,192.0bp) node {$4$};
  \draw [red,->] (node_18) ..controls (56.067bp,282.93bp) and (61.199bp,268.4bp)  .. (66.0bp,256.0bp) .. controls (67.043bp,253.31bp) and (68.178bp,250.47bp)  .. (node_23);
  \draw (74.5bp,265.0bp) node {$2$};
  \draw [blue,->] (node_17) ..controls (102.67bp,426.51bp) and (86.396bp,406.59bp)  .. (node_21);
  \draw (105.5bp,411.0bp) node {$1$};
  \draw [black,->] (node_14) ..controls (112.47bp,500.04bp) and (114.82bp,481.45bp)  .. (node_17);
  \draw (124.5bp,484.0bp) node {$4$};
  \draw [black,->] (node_25) ..controls (77.0bp,62.042bp) and (77.0bp,43.449bp)  .. (node_26);
  \draw (85.5bp,46.0bp) node {$6$};
  \draw [green,->] (node_21) ..controls (54.699bp,359.8bp) and (52.199bp,353.18bp)  .. (51.0bp,347.0bp) .. controls (49.379bp,338.65bp) and (48.968bp,329.25bp)  .. (node_18);
  \draw (59.5bp,338.0bp) node {$3$};
\end{tikzpicture}
\hspace{80pt}
\begin{tikzpicture}[>=latex,line join=bevel,xscale=1.0,yscale=0.435,every node/.style={scale=0.65}]
\node (node_13) at (50.0bp,273.0bp) [draw,draw=none] {$\btw\bfive3$};
  \node (node_14) at (50.0bp,153.0bp) [draw,draw=none] {$\bth4$};
  \node (node_9) at (30.0bp,633.0bp) [draw,draw=none] {$\bon\bfo3$};
  \node (node_8) at (110.0bp,513.0bp) [draw,draw=none] {$\bfive1$};
  \node (node_7) at (90.0bp,633.0bp) [draw,draw=none] {$\bth15$};
  \node (node_6) at (70.0bp,753.0bp) [draw,draw=none] {$\btw\bfo13$};
  \node (node_5) at (90.0bp,873.0bp) [draw,draw=none] {$\bfo2$};
  \node (node_4) at (10.0bp,753.0bp) [draw,draw=none] {$\bon4$};
  \node (node_3) at (50.0bp,873.0bp) [draw,draw=none] {$\btw14$};
  \node (node_2) at (70.0bp,993.0bp) [draw,draw=none] {$\bth24$};
  \node (node_1) at (70.0bp,1113.0bp) [draw,draw=none] {$\bfive3$};
  \node (node_10) at (50.0bp,513.0bp) [draw,draw=none] {$\bon\bth25$};
  \node (node_11) at (30.0bp,393.0bp) [draw,draw=none] {$\btw5$};
  \node (node_0) at (70.0bp,1233.0bp) [draw,draw=none] {$5$};
  \node (node_15) at (50.0bp,33.0bp) [draw,draw=none] {$\bfo$};
  \node (node_12) at (70.0bp,393.0bp) [draw,draw=none] {$\bon\bfive2$};
  \draw [blue,->] (node_6) ..controls (54.885bp,707.41bp) and (48.749bp,689.31bp)  .. (node_9);
  \definecolor{strokecol}{rgb}{0.0,0.0,0.0};
  \pgfsetstrokecolor{strokecol}
  \draw (61.5bp,693.0bp) node {$1$};
  \draw [red,->] (node_2) ..controls (56.424bp,961.9bp) and (52.849bp,951.67bp)  .. (51.0bp,942.0bp) .. controls (49.424bp,933.76bp) and (48.646bp,924.85bp)  .. (node_3);
  \draw (59.5bp,933.0bp) node {$2$};
  \draw [black,->] (node_10) ..controls (57.736bp,466.36bp) and (60.403bp,450.62bp)  .. (node_12);
  \draw (70.5bp,453.0bp) node {$5$};
  \draw [green,->] (node_13) ..controls (50.0bp,226.36bp) and (50.0bp,210.62bp)  .. (node_14);
  \draw (58.5bp,213.0bp) node {$3$};
  \draw [black,->] (node_7) ..controls (97.736bp,586.36bp) and (100.4bp,570.62bp)  .. (node_8);
  \draw (109.5bp,573.0bp) node {$5$};
  \draw [black,->] (node_2) ..controls (77.736bp,946.36bp) and (80.403bp,930.62bp)  .. (node_5);
  \draw (89.5bp,933.0bp) node {$4$};
  \draw [green,->] (node_6) ..controls (77.736bp,706.36bp) and (80.403bp,690.62bp)  .. (node_7);
  \draw (89.5bp,693.0bp) node {$3$};
  \draw [blue,->] (node_8) ..controls (94.885bp,467.41bp) and (88.749bp,449.31bp)  .. (node_12);
  \draw (100.5bp,453.0bp) node {$1$};
  \draw [black,->] (node_14) ..controls (50.0bp,106.36bp) and (50.0bp,90.625bp)  .. (node_15);
  \draw (58.5bp,93.0bp) node {$4$};
  \draw [green,->] (node_1) ..controls (70.0bp,1066.4bp) and (70.0bp,1050.6bp)  .. (node_2);
  \draw (78.5bp,1053.0bp) node {$3$};
  \draw [red,->] (node_12) ..controls (62.264bp,346.36bp) and (59.597bp,330.62bp)  .. (node_13);
  \draw (70.5bp,333.0bp) node {$2$};
  \draw [black,->] (node_11) ..controls (28.145bp,348.51bp) and (28.783bp,335.6bp)  .. (31.0bp,324.0bp) .. controls (32.213bp,317.66bp) and (34.17bp,311.07bp)  .. (node_13);
  \draw (39.5bp,333.0bp) node {$5$};
  \draw [green,->] (node_9) ..controls (28.145bp,588.51bp) and (28.783bp,575.6bp)  .. (31.0bp,564.0bp) .. controls (32.213bp,557.66bp) and (34.17bp,551.07bp)  .. (node_10);
  \draw (39.5bp,573.0bp) node {$3$};
  \draw [blue,->] (node_7) ..controls (74.885bp,587.41bp) and (68.749bp,569.31bp)  .. (node_10);
  \draw (80.5bp,573.0bp) node {$1$};
  \draw [blue,->] (node_3) ..controls (35.827bp,842.08bp) and (31.383bp,831.67bp)  .. (28.0bp,822.0bp) .. controls (25.101bp,813.72bp) and (22.41bp,804.72bp)  .. (node_4);
  \draw (36.5bp,813.0bp) node {$1$};
  \draw [red,->] (node_5) ..controls (82.264bp,826.36bp) and (79.597bp,810.62bp)  .. (node_6);
  \draw (89.5bp,813.0bp) node {$2$};
  \draw [black,->] (node_0) ..controls (70.0bp,1186.4bp) and (70.0bp,1170.6bp)  .. (node_1);
  \draw (78.5bp,1173.0bp) node {$5$};
  \draw [black,->] (node_3) ..controls (48.145bp,828.51bp) and (48.783bp,815.6bp)  .. (51.0bp,804.0bp) .. controls (52.213bp,797.66bp) and (54.17bp,791.07bp)  .. (node_6);
  \draw (59.5bp,813.0bp) node {$4$};
  \draw [red,->] (node_10) ..controls (36.424bp,481.9bp) and (32.849bp,471.67bp)  .. (31.0bp,462.0bp) .. controls (29.424bp,453.76bp) and (28.646bp,444.85bp)  .. (node_11);
  \draw (39.5bp,453.0bp) node {$2$};
  \draw [black,->] (node_4) ..controls (8.145bp,708.51bp) and (8.7826bp,695.6bp)  .. (11.0bp,684.0bp) .. controls (12.213bp,677.66bp) and (14.17bp,671.07bp)  .. (node_9);
  \draw (19.5bp,693.0bp) node {$4$};
\end{tikzpicture}
\]
\caption{The crystal $B(\clfw_1)$ of type $E_6$ (left) and $B(\clfw_5)$ of type $D_5$ (right).}
\label{fig:crystal_graphs}
\end{figure}
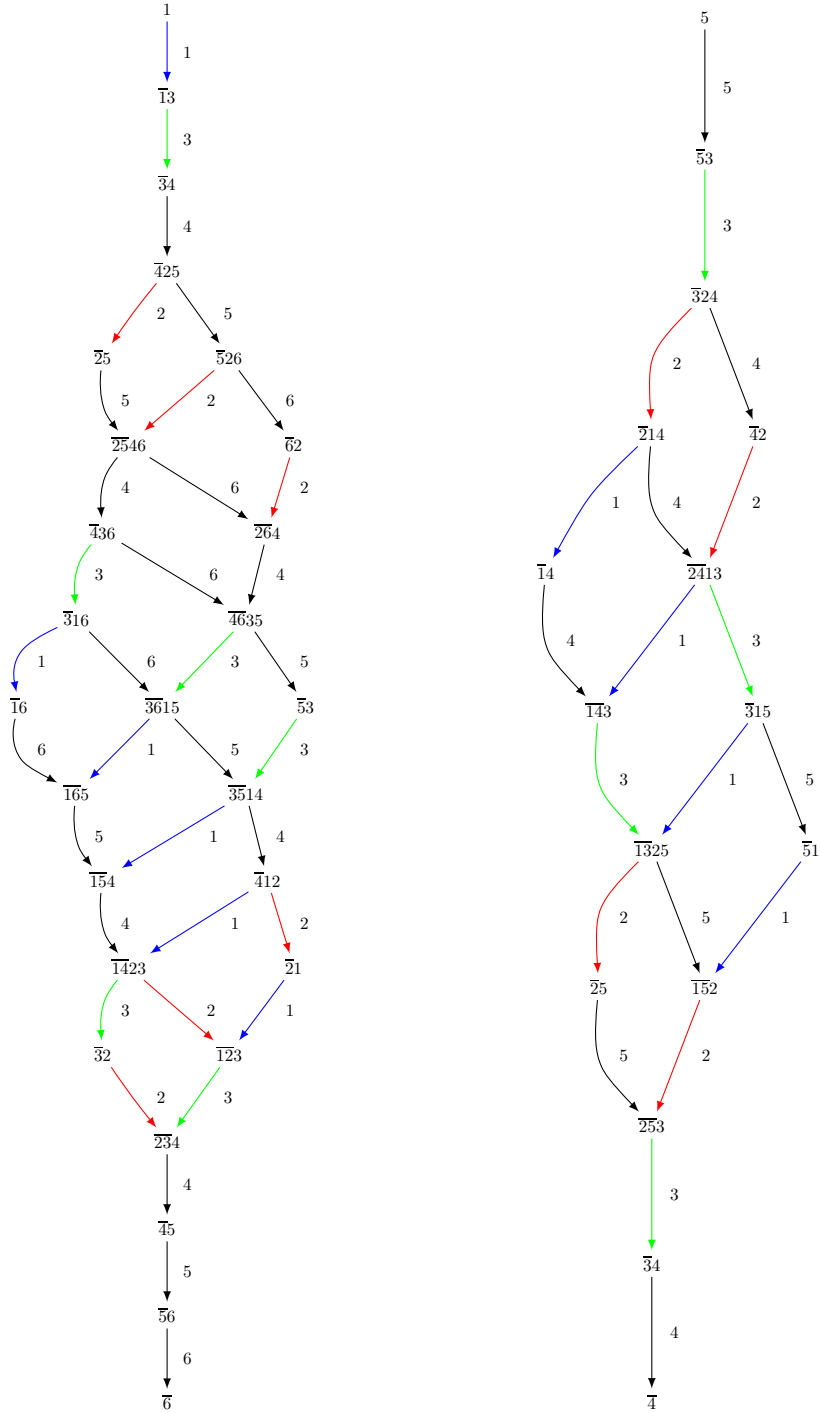

\subsection{Kirillov--Reshetikhin crystals}

Let $c_a$ denote the Kac labels~\cite[Table~Aff1-3]{kac90}.
Let $U_q'(\g) := U_q([\g, \g])$, and we note that the weight lattice is given by $P' = P / \ZZ\delta$, where $\delta = \sum_{a \in I} c_a \alpha_a$ is the null root.
In particular, the simple roots in $P'$ have a linear dependence.
We will not be considering $U_q(\g)$-crystals in this paper, and so we abuse notation and denote the $U_q'(\g)$-weight lattice by $P$.
For a $U_q'(\g)$-crystal $B$, we say $b \in B$ is \defn{classically highest weight} if $e_i b = 0$ for all $i \in I_0$.

An important class of finite-dimensional irreducible $U_q'(\g)$-modules are \defn{Kirillov--Reshetikhin (KR) modules}, which are characterized by their Drinfel'd polynomials~\cite{CP95, CP98}.
A (conjectural~\cite{HKOTY99, HKOTT02}) remarkable property of KR modules is that they admit crystal bases, which are known as \defn{Kirillov--Reshetikhin (KR) crystals}.
A KR crystal is denoted by $B^{r, s}$, where $r \in I_0$ and $s \in \ZZ_{>0}$.
KR crystals were shown to exist in all nonexceptional types in~\cite{OS08} and a combinatorial model given in~\cite{FOS09}.
Jones and Schilling showed KR crystals exist and gave a combinatorial model for $B^{r,s}$, where $r = 1, 2, 6$, in type $E_6^{(1)}$~\cite{JS10}.
The cases $B^{1,s}$ in type $D_4^{(3)}$~\cite{KMOY07} and $B^{2,s}$ in type $G_2^{(1)}$~\cite{Yamane98} are also known to exist and have a combinatorial model.
The KR crystal $B^{r,1}$ in all types was constructed uniformly using projected level-zero LS paths by the work of Naito and Sagaki~\cite{NS08II, NS08}.

We will also describe the elements of $B^{r,1}$ when $r \sim 0$ following~\cite{BFKL06}.
Specifically for $\g$ not of type $A_n^{(1)}$ or $A_{2n}^{(2)\dagger}$, we have $B^{r,1} \iso B(\clfw_r) \oplus B(0)$ as $U_q(\g_0)$-crystals.
We denote the unique element $\emptyset \in B(0)$, which plays the role of $y_0$.
For type $A_{2n}^{(2)\dagger}$, we have $B^{r,1} \iso B(\clfw_r)$ as $U_q(\g_0)$-crystals.

There exists a unique classical component $B(s\clfw_r) \subseteq B^{r,s}$.
Moreover, we have $B(s\clfw_r) \iso B^{r,s}$ as $U_q(\g_0)$-crystals when $r$ is a minuscule node.
Let $u(B^{r,s}) = u_{r\clfw_s} \in B(s\clfw_r) \subseteq B^{r,s}$ denote the \defn{maximal element}, which is the unique element in $B^{r,s}$ of classical weight $s \clfw_r$.
For $B = \bigotimes_{i=1}^N B^{r_i, s_i}$, the maximal element is $u(B) = u(B^{r_1,s_1}) \otimes \cdots \otimes u(B^{r_N,s_N})$, which is the unique element of classical weight $\sum_{i=1}^N s_i \clfw_{r_i}$.
Let $v(B)$ denote the \defn{minimal element} of $B$, which is the unique element of classical weight $-\sum_{i=1}^N s_i \clfw_{r_i}$.

It is conjectured that tensor products of KR crystals are connected, which is known in nonexceptional types~\cite{Okado13} and when the KR crystals are perfect of the same level~\cite{FSS07, ST12}.
Since tensor products of KR crystals are (conjecturally) connected, there exists a unique $U_q'(\g)$-crystal morphism $R \colon B \otimes B' \to B' \otimes B$, called the \defn{combinatorial $R$-matrix}, defined by
\[
R\bigl(u(B) \otimes u(B')\bigr) = u(B') \otimes u(B).
\]
We denote the combinatorial $R$-matrix $R(b \otimes b') = \widetilde{b}' \otimes \widetilde{b}$ pictorially by
\[
\begin{tikzpicture}[baseline=-10]
\node (bp) at (1,0) {$b'$};
\node (b) at (0,1) {$b$};
\node (bt) at (0,-1) {$\widetilde{b}$};
\node (btp) at (-1,0) {$\widetilde{b}'$};
\draw[->] (b) -- (bt);
\draw[->] (bp) -- (btp);
\end{tikzpicture}\ .
\]

We now describe an important statistic that arises from mathematical physics called the \defn{local energy function}.
Let $\widetilde{b}' \otimes \widetilde{b} = R(b \otimes b')$, and define the following conditions:
\begin{itemize}
\item[(LL)] $e_0(b \otimes b') = e_0 b \otimes b' \text{ and } e_0(\widetilde{b}' \otimes \widetilde{b}) = e_0 \widetilde{b}' \otimes \widetilde{b}$;
\item[(RR)] $e_0(b \otimes b') = b \otimes e_0 b' \text{ and } e_0(\widetilde{b}' \otimes \widetilde{b}) = \widetilde{b}' \otimes e_0 \widetilde{b}$.
\end{itemize}
The local energy function $H \colon B^{r,s} \otimes B^{r',s'} \to \ZZ$ is defined by
\begin{equation}
\label{eq:local_energy}
H\bigl( e_i(b \otimes b') \bigr) = H(b \otimes b') + \begin{cases}
-1 & \text{if } i = 0 \text{ and } (LL), \\
1 & \text{if } i = 0 \text{ and } (RR), \\
0 & \text{otherwise,}
\end{cases}
\end{equation}
and it is known $H$ is defined up to an additive constant~\cite{KKMMNN91}.
We normalize $H$ by setting $H\bigl( u(B^{r ,s}) \otimes u(B^{r',s'}) \bigr) = 0$.
Note that $H$ is constant on classical components.

For special nodes, the local energy has the following form.

\begin{thm}[{\cite[Thm.~7.5]{Scrimshaw17}}]
\label{thm:minuscule_local_energy}
Let $r$ be a special node.
Then the classically highest weight elements of $B^{r,s} \otimes B^{r,s'}$ are of the form $b \otimes u(B^{r,s'})$ and $H\bigl(b \otimes u(B^{r,s'})\bigr)$ equals the number of $r$-arrows in the path from $b$ to $u_{s\clfw_r}$ in $B(s\clfw_r)$.
\end{thm}

For $r \sim 0$, the local energy has the following form.

\begin{thm}[{\cite[Thm.~6.2]{BFKL06}}]
\label{thm:adjoint_local_energy}
Let $r \sim 0$. Then the local energy of $B^{r,1} \otimes B^{r,1}$ is given in Table~\ref{table:adjoint_local_energy} (it is specified on the $I_0$-highest weight elements).
\end{thm}

\begin{table}
\[
\scalebox{0.93}{$
\begin{array}{ccccccccc}
\toprule
b' \otimes b & u \otimes u_{\clfw_r} & (f_r u) \otimes u & \emptyset \otimes u & u \otimes \emptyset & x \otimes u_{\clfw_r} & y_r \otimes u & v \otimes u & \emptyset \otimes \emptyset
\\ \midrule
H(b' \otimes b) & 0 & -A_{0,r} & 1 & 1 & 2 & 2 & 2 & 2
\\ \bottomrule
\end{array}
$}
\]
\caption{Local energy on $B^{r,1} \otimes B^{r,1}$ for $r \sim 0$, where $u = u(B^{r,1})$, $v = v(B^{r,1})$, $x \in B(\clfw_r) \setminus \{ u, f_r u \}$.}
\label{table:adjoint_local_energy}
\end{table}

\begin{remark}
The local energy function on $B^{r,1} \otimes B^{r,1}$ from~\cite{BFKL06} in Table~\ref{table:adjoint_local_energy} is renormalized to our convention.
Let $u = u(B^{r,1})$.
We note that there are two minor errors in~\cite{BFKL06}, where it is stated $H(x \otimes u) = 1$ and $H(f_r u \otimes u) = 1$ (the only difference is for types $D_{n+1}^{(2)}$ and $A_{2n}^{(2)}$).
\end{remark}

%

Define $\varphi(b) = \sum_{i \in I} \varphi_i(b) \Lambda_i$.
Let $b^{\sharp} \in B^{r,s}$ be the unique element such that $\varphi(b^{\sharp}) = \ell \Lambda_0$, where $\ell = \min \{ \inner{c}{\varphi(b)} \mid b \in B^{r,s} \}$ and $c$ is the canonical central element.
For example, when $r$ is a special node, we have $b^{\sharp} = v(B^{r,s})$.
Following~\cite{HKOTT02}, we then define $D_{B^{r,s}} \colon B^{r,s} \to \ZZ$ by
\[
D_{B^{r,s}}(b) = H(b \otimes b^{\sharp}) - H(u(B^{r,s}) \otimes b^{\sharp}).
\]
Let $B = \bigotimes_{k=1}^N B^{r_k,s_k}$. We define \defn{energy}~\cite{HKOTY99} $D \colon B \to \ZZ$ by
\begin{equation}
\label{eq:energy_function}
D = \sum_{1 \leq j < k \leq N} H_j R_{j+1} R_{j+2} \cdots R_{k-1} + \sum_{j=1}^N D_{B^{r_j,s_j}} R_1 R_2 \cdots R_{j-1},
\end{equation}
where $R_j$ and $H_j$ are the combinatorial $R$-matrix and local energy function, respectively, acting on the $j$-th and $(j+1)$-th factors and $D_{B^{r_j,s_j}}$ acts on the rightmost factor. Note that $D$ is constant on classical components since $H$ is and $R$ is a $U_q'(\g)$-crystal isomorphism.

If we restrict ourselves to the case when $B = (B^{r,s})^{\otimes N}$ for $r$ a minuscule node, then we can simplify the energy function to
\begin{equation}
\label{eq:energy_function_simple}
D(b_1 \otimes \cdots \otimes b_N) = \sum_{j=1}^{N-1} j H(b_j \otimes b_{j+1}).
\end{equation}

\subsection{Soliton cellular automata}

A \defn{soliton cellular automaton (SCA)} using $B^{r,1}$ of type $\g$ is a discrete dynamical system, where a \defn{state} is an element in $\bigotimes_{k=-\infty}^{0} B^{r,1}$ of the form
\[
\cdots \otimes u_1 \otimes u_1 \otimes b_L \otimes \cdots \otimes b_1 \otimes b_0,
\]
where $u_s$ is the maximal element of $B^{r,s}$, for some $L \gg 1$.
The element $u_1$ in a state is called a \defn{vacuum element}.
Given a state $b$, define the \defn{time evolution operator} $T_{\ell}(b)$ by
\[
u_{\ell} \otimes T_{\ell}(b) = \cdots R_{-3} R_{-2} R_{-1} R_0(b \otimes u_{\ell}).
\]
Note that this is well defined because eventually we have $R(u_1 \otimes u_{\ell}) = u_{\ell} \otimes u_1$.
We depict this by
\[
\begin{tikzpicture}[baseline=-20]
\node (u0) at (0.4,0) {$u^{(0)} = u_{\ell}$};
\foreach \x in {-4, ..., 0} {
    \node (b\x) at (2*\x-1,0.9) {$b_{\x}$};
    \node (bt\x) at (2*\x-1,-1) {$\widetilde{b}_{\x}$};
    \draw[->] (b\x) -- (bt\x);
}
\foreach \x [evaluate=\x as \xm using {int(\x-1)}] in {1, ..., 4} {
    \node (u\x) at (-2*\x, 0) {$u^{(\x)}$};
    \draw[->] (u\xm) -- (u\x);
}
\node (u5) at (-10, 0) {$\cdots$};
\draw[->] (u4) -- (u5);
\end{tikzpicture}.
\]
The \defn{state energy} is defined as
\[
E_{\ell}(b) = \sum_{k=0}^{\infty} H(b_{-k} \otimes u^{(k)}).
\]
When $\ell = 1$, the state energy may be simplified as
\[
E_1(b) = \sum_{k=0}^{L+1} H(b_k \otimes b_{k+1}),
\]
where $b_1 = b_{-L-1} = u_1$, since the combinatorial $R$-matrix $R \colon B^{r,s} \otimes B^{r,s} \to B^{r,s} \otimes B^{r,s}$ is the identity map.

\begin{dfn}
\label{def:soliton}
Consider an SCA using $B^{r,1}$ of type $\g$.
A \defn{soliton} is a tensor product $b = b_{-L} \otimes \cdots \otimes b_0$ such that $b_k \neq u_1$, for all $k$, and $E_1(b) = H(f_r u_1 \otimes u_1)$ (with $b$ considered as a state by $\cdots \otimes u_1 \otimes u_1 \otimes b$).
The \defn{length} of a soliton is $\sum_{k=0}^L N_r(b_{-k})$, where
\[
N_r(b_{-k}) = \begin{cases}
1 & \text{if } b_{-k} = \emptyset, \\
\lvert \{a \mid i_a = r \} \rvert & \text{otherwise},
\end{cases}
\]
for $e_{i_1} \dotsm e_{i_m} b_{-k} = u_{\clfw_r}$.
\end{dfn}

\begin{ex}
\label{ex:carrier_type_C}
Consider $r = 1$ in type $C_3^{(1)}$.
Then $2 \otimes \bon$ is a soliton of length $3$ as $e_1 e_2 e_3 e_2 e_1 \bon = 1 = u_1$ and $E_1(2 \otimes \bon) = H(2 \otimes 1) = 1$.
Therefore, $\ell$ is not necessarily equal $L$.
In particular, we apply $T_5(\cdots \otimes 1 \otimes 1 \otimes 2 \otimes \bon)$:
\[
\begin{tikzpicture}[baseline=-20,scale=0.9]
\node (u0) at (0,0) {$1^5$};
\node (u1) at (-2,0) {$1^3$};
\node (u2) at (-4,0) {$1^2 2$};
\node (u3) at (-6,0) {$1^3 2\bon$};
\node (u4) at (-8,0) {$1^4 2$};
\node (u5) at (-10,0) {$1^5$};
\node (u6) at (-12,0) {$\cdots$};
\node (b0) at (-1,0.9) {$\bon$};
\node (b-1) at (-3,0.9) {$2$};
\foreach \x in {-5, ..., -2} {
    \node (b\x) at (2*\x-1,0.9) {$1$};
}
\foreach \x in {-5, -2, -1, 0} {
    \node (bt\x) at (2*\x-1,-1) {$1$};
}
\node (bt-3) at (-7,-1) {$\bon$};
\node (bt-4) at (-9,-1) {$2$};
\foreach \x in {-5, ..., 0} {
    \draw[->] (b\x) -- (bt\x);
}
\foreach \x [evaluate=\x as \xm using {int(\x-1)}] in {1, ..., 6} {
    \draw[->] (u\xm) -- (u\x);
}
\end{tikzpicture}.
\]
Note that the soliton moved $3$ steps to the left, which agrees with its length.
\end{ex}

Let $\Sol^{(r)}(\ell_1, \ell_2, \dotsc, \ell_m)$ denote the set of states consisting of a solitons of length $\ell_1, \dotsc, \ell_m$, in that order from right to left, in the SCA using $B^{r,1}$ of type $\g$ that are ``far apart;'' \textit{i.e.}, the solitons (pairwise) will not be interacting after applying $T_{\infty}$.
Note that the set of solitons of length $\ell$ is equivalent to $\Sol^{(r)}(\ell)$ up to removing vacuum states.

\begin{ex}
Let $\g$ be of type $D_5^{(2)}$.
In this example and all subsequent examples, unless otherwise stated, we will evolve the SCA by $T_{\infty}$ (or $T_{\ell}$ for $\ell \gg 1$) and we omit the tensor products.
We give the evolution of the SCA starting with a state in $\Sol^{(1)}(2,3,4)$:
\[
{\begin{array}{c|c}
t = 0 & \cdots {\color{gray} 1} {\color{gray} 1} {\color{gray} 1} {\color{gray} 1} {\color{gray} 1} {\color{gray} 1} {\color{gray} 1} {\color{gray} 1} {\color{gray} 1} {\color{gray} 1} {\color{gray} 1} {\color{gray} 1} {\color{gray} 1} {\color{gray} 1} {\color{gray} 1} {\color{gray} 1} {\color{gray} 1} {\color{gray} 1} {\color{gray} 1} {\color{gray} 1} {\color{gray} 1} {\color{gray} 1} {\color{gray} 1} {\color{gray} 1} {\color{gray} 1} {\color{gray} 1} {\color{gray} 1} {\color{gray} 1} {\color{gray} 1} {\color{gray} 1} {\color{gray} 1} {\color{gray} 1} {\color{gray} 1} {\color{gray} 1} {\color{gray} 1} 2 \overline{3} {\color{gray} 1} {\color{gray} 1} {\color{gray} 1} 4 0 \overline{2} {\color{gray} 1} {\color{gray} 1} {\color{gray} 1} {\color{gray} 1} 3 \overline{4} \overline{3} \overline{3} {\color{gray} 1} \\
t = 1 & \cdots {\color{gray} 1} {\color{gray} 1} {\color{gray} 1} {\color{gray} 1} {\color{gray} 1} {\color{gray} 1} {\color{gray} 1} {\color{gray} 1} {\color{gray} 1} {\color{gray} 1} {\color{gray} 1} {\color{gray} 1} {\color{gray} 1} {\color{gray} 1} {\color{gray} 1} {\color{gray} 1} {\color{gray} 1} {\color{gray} 1} {\color{gray} 1} {\color{gray} 1} {\color{gray} 1} {\color{gray} 1} {\color{gray} 1} {\color{gray} 1} {\color{gray} 1} {\color{gray} 1} {\color{gray} 1} {\color{gray} 1} {\color{gray} 1} {\color{gray} 1} {\color{gray} 1} {\color{gray} 1} {\color{gray} 1} 2 \overline{3} {\color{gray} 1} {\color{gray} 1} 4 0 \overline{2} {\color{gray} 1} {\color{gray} 1} {\color{gray} 1} 3 \overline{4} \overline{3} \overline{3} {\color{gray} 1} {\color{gray} 1} {\color{gray} 1} {\color{gray} 1} {\color{gray} 1} \\
t = 2 & \cdots {\color{gray} 1} {\color{gray} 1} {\color{gray} 1} {\color{gray} 1} {\color{gray} 1} {\color{gray} 1} {\color{gray} 1} {\color{gray} 1} {\color{gray} 1} {\color{gray} 1} {\color{gray} 1} {\color{gray} 1} {\color{gray} 1} {\color{gray} 1} {\color{gray} 1} {\color{gray} 1} {\color{gray} 1} {\color{gray} 1} {\color{gray} 1} {\color{gray} 1} {\color{gray} 1} {\color{gray} 1} {\color{gray} 1} {\color{gray} 1} {\color{gray} 1} {\color{gray} 1} {\color{gray} 1} {\color{gray} 1} {\color{gray} 1} {\color{gray} 1} {\color{gray} 1} 2 \overline{3} {\color{gray} 1} 4 0 \overline{2} {\color{gray} 1} {\color{gray} 1} 3 \overline{4} \overline{3} \overline{3} {\color{gray} 1} {\color{gray} 1} {\color{gray} 1} {\color{gray} 1} {\color{gray} 1} {\color{gray} 1} {\color{gray} 1} {\color{gray} 1} {\color{gray} 1} \\
t = 3 & \cdots {\color{gray} 1} {\color{gray} 1} {\color{gray} 1} {\color{gray} 1} {\color{gray} 1} {\color{gray} 1} {\color{gray} 1} {\color{gray} 1} {\color{gray} 1} {\color{gray} 1} {\color{gray} 1} {\color{gray} 1} {\color{gray} 1} {\color{gray} 1} {\color{gray} 1} {\color{gray} 1} {\color{gray} 1} {\color{gray} 1} {\color{gray} 1} {\color{gray} 1} {\color{gray} 1} {\color{gray} 1} {\color{gray} 1} {\color{gray} 1} {\color{gray} 1} {\color{gray} 1} {\color{gray} 1} {\color{gray} 1} 2 4 \overline{3} {\color{gray} 1} 0 \overline{2} {\color{gray} 1} 3 \overline{4} \overline{3} \overline{3} {\color{gray} 1} {\color{gray} 1} {\color{gray} 1} {\color{gray} 1} {\color{gray} 1} {\color{gray} 1} {\color{gray} 1} {\color{gray} 1} {\color{gray} 1} {\color{gray} 1} {\color{gray} 1} {\color{gray} 1} {\color{gray} 1} \\
t = 4 & \cdots {\color{gray} 1} {\color{gray} 1} {\color{gray} 1} {\color{gray} 1} {\color{gray} 1} {\color{gray} 1} {\color{gray} 1} {\color{gray} 1} {\color{gray} 1} {\color{gray} 1} {\color{gray} 1} {\color{gray} 1} {\color{gray} 1} {\color{gray} 1} {\color{gray} 1} {\color{gray} 1} {\color{gray} 1} {\color{gray} 1} {\color{gray} 1} {\color{gray} 1} {\color{gray} 1} {\color{gray} 1} {\color{gray} 1} {\color{gray} 1} 2 4 0 \overline{3} {\color{gray} 1} {\color{gray} 1} \overline{4} \overline{2} 3 \overline{3} \overline{3} {\color{gray} 1} {\color{gray} 1} {\color{gray} 1} {\color{gray} 1} {\color{gray} 1} {\color{gray} 1} {\color{gray} 1} {\color{gray} 1} {\color{gray} 1} {\color{gray} 1} {\color{gray} 1} {\color{gray} 1} {\color{gray} 1} {\color{gray} 1} {\color{gray} 1} {\color{gray} 1} {\color{gray} 1} \\
t = 5 & \cdots {\color{gray} 1} {\color{gray} 1} {\color{gray} 1} {\color{gray} 1} {\color{gray} 1} {\color{gray} 1} {\color{gray} 1} {\color{gray} 1} {\color{gray} 1} {\color{gray} 1} {\color{gray} 1} {\color{gray} 1} {\color{gray} 1} {\color{gray} 1} {\color{gray} 1} {\color{gray} 1} {\color{gray} 1} {\color{gray} 1} {\color{gray} 1} {\color{gray} 1} 2 4 0 \overline{3} {\color{gray} 1} {\color{gray} 1} {\color{gray} 1} \overline{4} \overline{4} \overline{2} 4 \overline{3} {\color{gray} 1} {\color{gray} 1} {\color{gray} 1} {\color{gray} 1} {\color{gray} 1} {\color{gray} 1} {\color{gray} 1} {\color{gray} 1} {\color{gray} 1} {\color{gray} 1} {\color{gray} 1} {\color{gray} 1} {\color{gray} 1} {\color{gray} 1} {\color{gray} 1} {\color{gray} 1} {\color{gray} 1} {\color{gray} 1} {\color{gray} 1} {\color{gray} 1} \\
t = 6 & \cdots {\color{gray} 1} {\color{gray} 1} {\color{gray} 1} {\color{gray} 1} {\color{gray} 1} {\color{gray} 1} {\color{gray} 1} {\color{gray} 1} {\color{gray} 1} {\color{gray} 1} {\color{gray} 1} {\color{gray} 1} {\color{gray} 1} {\color{gray} 1} {\color{gray} 1} {\color{gray} 1} 2 4 0 \overline{3} {\color{gray} 1} {\color{gray} 1} {\color{gray} 1} {\color{gray} 1} \overline{4} \overline{4} \overline{2} {\color{gray} 1} 4 \overline{3} {\color{gray} 1} {\color{gray} 1} {\color{gray} 1} {\color{gray} 1} {\color{gray} 1} {\color{gray} 1} {\color{gray} 1} {\color{gray} 1} {\color{gray} 1} {\color{gray} 1} {\color{gray} 1} {\color{gray} 1} {\color{gray} 1} {\color{gray} 1} {\color{gray} 1} {\color{gray} 1} {\color{gray} 1} {\color{gray} 1} {\color{gray} 1} {\color{gray} 1} {\color{gray} 1} {\color{gray} 1} \\
t = 7 & \cdots {\color{gray} 1} {\color{gray} 1} {\color{gray} 1} {\color{gray} 1} {\color{gray} 1} {\color{gray} 1} {\color{gray} 1} {\color{gray} 1} {\color{gray} 1} {\color{gray} 1} {\color{gray} 1} {\color{gray} 1} 2 4 0 \overline{3} {\color{gray} 1} {\color{gray} 1} {\color{gray} 1} {\color{gray} 1} {\color{gray} 1} \overline{4} \overline{4} \overline{2} {\color{gray} 1} {\color{gray} 1} 4 \overline{3} {\color{gray} 1} {\color{gray} 1} {\color{gray} 1} {\color{gray} 1} {\color{gray} 1} {\color{gray} 1} {\color{gray} 1} {\color{gray} 1} {\color{gray} 1} {\color{gray} 1} {\color{gray} 1} {\color{gray} 1} {\color{gray} 1} {\color{gray} 1} {\color{gray} 1} {\color{gray} 1} {\color{gray} 1} {\color{gray} 1} {\color{gray} 1} {\color{gray} 1} {\color{gray} 1} {\color{gray} 1} {\color{gray} 1} {\color{gray} 1} \\
t = 8 & \cdots {\color{gray} 1} {\color{gray} 1} {\color{gray} 1} {\color{gray} 1} {\color{gray} 1} {\color{gray} 1} {\color{gray} 1} {\color{gray} 1} 2 4 0 \overline{3} {\color{gray} 1} {\color{gray} 1} {\color{gray} 1} {\color{gray} 1} {\color{gray} 1} {\color{gray} 1} \overline{4} \overline{4} \overline{2} {\color{gray} 1} {\color{gray} 1} {\color{gray} 1} 4 \overline{3} {\color{gray} 1} {\color{gray} 1} {\color{gray} 1} {\color{gray} 1} {\color{gray} 1} {\color{gray} 1} {\color{gray} 1} {\color{gray} 1} {\color{gray} 1} {\color{gray} 1} {\color{gray} 1} {\color{gray} 1} {\color{gray} 1} {\color{gray} 1} {\color{gray} 1} {\color{gray} 1} {\color{gray} 1} {\color{gray} 1} {\color{gray} 1} {\color{gray} 1} {\color{gray} 1} {\color{gray} 1} {\color{gray} 1} {\color{gray} 1} {\color{gray} 1} {\color{gray} 1}\\
t = 9 & \cdots {\color{gray} 1} {\color{gray} 1} {\color{gray} 1} {\color{gray} 1} 2 4 0 \overline{3} {\color{gray} 1} {\color{gray} 1} {\color{gray} 1} {\color{gray} 1} {\color{gray} 1} {\color{gray} 1} {\color{gray} 1} \overline{4} \overline{4} \overline{2} {\color{gray} 1} {\color{gray} 1} {\color{gray} 1} {\color{gray} 1} 4 \overline{3} {\color{gray} 1} {\color{gray} 1} {\color{gray} 1} {\color{gray} 1} {\color{gray} 1} {\color{gray} 1} {\color{gray} 1} {\color{gray} 1} {\color{gray} 1} {\color{gray} 1} {\color{gray} 1} {\color{gray} 1} {\color{gray} 1} {\color{gray} 1} {\color{gray} 1} {\color{gray} 1} {\color{gray} 1} {\color{gray} 1} {\color{gray} 1} {\color{gray} 1} {\color{gray} 1} {\color{gray} 1} {\color{gray} 1} {\color{gray} 1} {\color{gray} 1} {\color{gray} 1} {\color{gray} 1} {\color{gray} 1}\\
\end{array}}
\]
\end{ex}

Next, we recall some conjectures about solitons.
First, we define $\widehat{\g}_{0,r}$ roughly by removing the node $r$ from the classical Dynkin diagram and taking the affine version of each of the remaining components, taking the twisted version if $\g$ was twisted.
More specifically, suppose $\g$ is of type $X_{\mathfrak{R}(n)}^{(t)}$.
If $\g$ is of nonexceptional type except for $r = n-2,n-1$ in type $D_n^{(1)}$, then $\widehat{\g}_{0,r}$ is of type $A_{r-1}^{(1)} \otimes X_{\mathfrak{R}(n-r)}^{(t)}$.\footnote{We consider $D_3^{(1)} = A_3^{(1)}$ and $D_2^{(1)} = A_1^{(1)} \times A_1^{(1)}$. Furthermore, we consider $X_{\mathfrak{R}(1)}^{(t)} = A_1^{(1)}$ whenever $X_{\mathfrak{R}(1)}^{(t)} \neq A_2^{(2)}$.}
Because of the trivalent node in the classical Dynkin diagram of $D_n$, for type $D_n^{(1)}$, we have $\widehat{\g}_{0,n-1} = \widehat{\g}_{0,n}$, which has type $A_{n-1}^{(1)}$, and $\widehat{\g}_{0,n-2}$ of type $A_{n-3}^{(1)} \times A_1^{(1)} \times A_1^{(1)}$.
For type $E_{6,7,8}^{(1)}$, we consider the untwisted affine version of $\g_{0,r}$ (see Table~\ref{table:restrictions_E}).
Otherwise, $\widehat{\g}_{0,r}$ for the $r$ we consider is given by Table~\ref{table:restrictions}.
Let  $(\gamma_i)_{i \in I}$ be given by Table~\ref{table:scaling_factors} and
\[
\modgamma_i = \begin{cases}
1 & \text{if } \g = A_{2n}^{(2)} \text{ and } i = n, \\
2 & \text{if } \g = A_{2n}^{(2)\dagger} \text{ and } i = n, \\
\gamma_i & \text{otherwise.}
\end{cases}
\]

\begin{table}
\[
\begin{array}{|c|c|}
\hline
r & \widehat{\g}_{0,r} \\\hline
1 &  D_{n-1}^{(1)} \\
2 & A_{n-1}^{(1)} \\
3 & A_{n-2}^{(1)} \times A_1^{(1)} \\
4 & A_{n-4}^{(1)} \times A_2^{(1)} \times A_1^{(1)} \\
5 & A_4^{(1)} \times A_{n-5}^{(1)} \\
6 & D_5^{(1)} \times A_{n-6}^{(1)} \\
7 & E_6^{(1)} \times A_{n-7}^{(1)} \\
8 & E_7^{(1)}
\\\hline
\end{array}
\]
\caption{The restrictions for type $E_n^{(1)}$, where we disregard any $A_k^{(1)}$ with $k \leq 0$.}
\label{table:restrictions_E}
\end{table}

\begin{table}
\[
\begin{array}{cccccccccc}
\toprule
\g & F_4^{(1)} & F_4^{(1)} & F_4^{(1)} & E_6^{(2)} & E_6^{(2)} & E_6^{(2)} & G_2^{(1)} & D_4^{(3)} \\ \midrule
r & 1 & 2,3 & 4 & 1 & 2,3 & 4 & 1,2 & 1,2 \\ \midrule
\widehat{\g}_{0,r} & C_3^{(1)} & A_2^{(1)} \times A_1^{(1)} & B_3^{(1)} & D_4^{(2)} & A_2^{(1)} \times A_1^{(1)} & A_5^{(2)} & A_1^{(1)} & A_1^{(1)} \\
\bottomrule
\end{array}
\]
\caption{The restrictions given by removing node $r$ in type $\g$ considered.}
\label{table:restrictions}
\end{table}

Three of the fundamental questions about SCA are given as the following conjecture.\footnote{Two other fundamental questions include determining the soliton equations and an ultradiscrete description.}
This has not been explicitly stated in the literature as far as the authors are aware, but it is known to experts.

\begin{conj}
\label{conj:SCA}
Consider an SCA using $B^{r,1}$ of type $\g$, and fix some integers $\ell_1 < \ell_2$.
\begin{enumerate}[{\rm (1)}]
\item \label{conj:SCA_bijection} There exists a bijection $\Psi$ between solitons of length $\ell$ and elements in
\[
\mathcal{B}^{(r)}(\ell) := \bigotimes_{r' \sim r} \begin{cases}
B^{r', \ell \modgamma_r / \gamma_{r'}} & \text{if $\modgamma_r / \gamma_{r'} \in \ZZ$}, \\
B^{r', \lfloor \ell / 2 \rfloor} \otimes B^{r', \lceil \ell / 2 \rceil} & \text{if $\modgamma_r / \gamma_{r'} = 1/2$}, \\
B^{r', \lfloor \ell / 3 \rfloor} \otimes B^{r', \lfloor \ell / 3 \rfloor + \sigma} \otimes B^{r', \lfloor \ell / 3 \rfloor + \tau} & \text{if $\modgamma_r / \gamma_{r'} = 1/3$},
\end{cases}
\]
of type $\widehat{\g}_{r,0}$, where $\sigma = 1$ (resp.~$\tau = 1$) if the remainder of $\ell / 3$ is at least $1$ (resp.\ equals $2$) and $0$ otherwise.
\item \label{conj:SCA_scattering} The \defn{scattering rule} of two solitons, a state in $\Sol^{(r)}(\ell_1, \ell_2)$ evolving to $\Sol^{(r)}(\ell_2, \ell_1)$ under sufficiently many time evolutions, is given by the combinatorial $R$-matrix
\[
R \colon \mathcal{B}^{(r)}(\ell_1) \otimes \mathcal{B}^{(r)}(\ell_2) \to \mathcal{B}^{(r)}(\ell_2) \otimes \mathcal{B}^{(r)}(\ell_1)
\]
of type $\widehat{\g}_{0,r}$.
\item \label{conj:SCA_phase_shift} The \defn{phase shift}, the shift of the soliton positions after scattering compared to if there was no interaction, is given by
\[
\delta = 2\ell_1 - C_r H\bigl(x \otimes y\bigr),
\]
where $x \in \mathcal{B}^{(r)}(\ell_1)$ and $y \in \mathcal{B}^{(r)}(\ell_2)$ correspond to the solitons given by Part~{\rm(\ref{conj:SCA_bijection})} and $C_r \in \ZZ$.
\end{enumerate}
\end{conj}

\begin{remark}
Conjecture~\ref{conj:SCA}(\ref{conj:SCA_bijection}) can be considered an interpretation of the formulas given in~\cite[App.~B]{HKOTT02}.
\end{remark}

We make some remarks about Conjecture~\ref{conj:SCA}.
First, the ordering of the tensor factors in Part~(\ref{conj:SCA_bijection}) does not matter by the combinatorial $R$-matrix.
Additionally, we can extend Part~(\ref{conj:SCA_bijection}) to a more general case.
\begin{dfn}[Decoupling rule]
\label{def:decoupling}
Define
\[
\mathcal{B}^{(r)}(\ell_1, \dotsc, \ell_k) := \bigotimes_{i=1}^k \mathcal{B}^{(r)}(\ell_i).
\]
We construct a bijection between $\Sol^{(r)}(\ell_1, \dotsc, \ell_k)$ and $\mathcal{B}^{(r)}(\ell_1, \dotsc, \ell_k)$ by applying the bijection $\Psi$ from Part~(\ref{conj:SCA_bijection}) on each soliton of a fixed state $p \in \Sol^{(r)}(\ell_1, \dotsc, \ell_k)$.
We call the resulting bijection the \defn{decoupling rule}.
\end{dfn}
By slight abuse of notation, we also denote the decoupling rule by $\Psi$.
More explicitly, consider solitons $b_1, \dotsc, b_k$ (in that order), then
\[
\Psi(p) = \Psi(b_1) \otimes \cdots \otimes \Psi(b_k).
\]
We also note that we have $\modgamma_r / \gamma_{r'} \neq 1$ only if $\widehat{\g}_{0,r} = A_{n-1}^{(1)}$.
Furthermore, the phase shift of one soliton is the negative phase shift of the other in two body scattering.

\begin{ex}
Consider an SCA starting with a state in $\Sol^{(1)}(2,3)$ of type $A_3^{(1)}$:
\[
{\begin{array}{c|c}
t = 0 & \cdots {\color{gray} 1} {\color{gray} 1} {\color{gray} 1} {\color{gray} 1} {\color{gray} 1} {\color{gray} 1} {\color{gray} 1} {\color{gray} 1} {\color{gray} 1} {\color{gray} 1} {\color{gray} 1} {\color{gray} 1} {\color{gray} 1} {\color{gray} 1} {\color{gray} 1} {\color{gray} 1} {\color{gray} 1} {\color{gray} 1} {\color{gray} 1} {\color{gray} 1} {\color{gray} 1} {\color{gray} 1} {\color{gray} 1} {\color{gray} 1} {\color{gray} 1} {\color{gray} 1} {\color{gray} 1} {\color{gray} 1} {\color{gray} 1} {\color{gray} 1} {\color{gray} 1} {\color{gray} 1} \underline{3 4} {\color{gray} 1} {\color{gray} 1} {\color{gray} 1} \underline{2 3 4} {\color{gray} 1} \\
t = 1 & \cdots {\color{gray} 1} {\color{gray} 1} {\color{gray} 1} {\color{gray} 1} {\color{gray} 1} {\color{gray} 1} {\color{gray} 1} {\color{gray} 1} {\color{gray} 1} {\color{gray} 1} {\color{gray} 1} {\color{gray} 1} {\color{gray} 1} {\color{gray} 1} {\color{gray} 1} {\color{gray} 1} {\color{gray} 1} {\color{gray} 1} {\color{gray} 1} {\color{gray} 1} {\color{gray} 1} {\color{gray} 1} {\color{gray} 1} {\color{gray} 1} {\color{gray} 1} {\color{gray} 1} {\color{gray} 1} {\color{gray} 1} {\color{gray} 1} {\color{gray} 1} \underline{3 4} {\color{gray} 1} {\color{gray} 1} \underline{2 3 4} {\color{gray} 1} {\color{gray} 1} {\color{gray} 1} {\color{gray} 1} \\
t = 2 & \cdots {\color{gray} 1} {\color{gray} 1} {\color{gray} 1} {\color{gray} 1} {\color{gray} 1} {\color{gray} 1} {\color{gray} 1} {\color{gray} 1} {\color{gray} 1} {\color{gray} 1} {\color{gray} 1} {\color{gray} 1} {\color{gray} 1} {\color{gray} 1} {\color{gray} 1} {\color{gray} 1} {\color{gray} 1} {\color{gray} 1} {\color{gray} 1} {\color{gray} 1} {\color{gray} 1} {\color{gray} 1} {\color{gray} 1} {\color{gray} 1} {\color{gray} 1} {\color{gray} 1} {\color{gray} 1} {\color{gray} 1} \underline{3 4} {\color{gray} 1} \underline{2 3 4} {\color{gray} 1} {\color{gray} 1} {\color{gray} 1} {\color{gray} 1} {\color{gray} 1} {\color{gray} 1} {\color{gray} 1}\\
t = 3 & \cdots {\color{gray} 1} {\color{gray} 1} {\color{gray} 1} {\color{gray} 1} {\color{gray} 1} {\color{gray} 1} {\color{gray} 1} {\color{gray} 1} {\color{gray} 1} {\color{gray} 1} {\color{gray} 1} {\color{gray} 1} {\color{gray} 1} {\color{gray} 1} {\color{gray} 1} {\color{gray} 1} {\color{gray} 1} {\color{gray} 1} {\color{gray} 1} {\color{gray} 1} {\color{gray} 1} {\color{gray} 1} {\color{gray} 1} {\color{gray} 1} {\color{gray} 1} {\color{gray} 1} \underline{3 4} \underline{2 3 4} {\color{gray} 1} {\color{gray} 1} {\color{gray} 1} {\color{gray} 1} {\color{gray} 1} {\color{gray} 1} {\color{gray} 1} {\color{gray} 1} {\color{gray} 1} {\color{gray} 1} \\
t = 4 & \cdots {\color{gray} 1} {\color{gray} 1} {\color{gray} 1} {\color{gray} 1} {\color{gray} 1} {\color{gray} 1} {\color{gray} 1} {\color{gray} 1} {\color{gray} 1} {\color{gray} 1} {\color{gray} 1} {\color{gray} 1} {\color{gray} 1} {\color{gray} 1} {\color{gray} 1} {\color{gray} 1} {\color{gray} 1} {\color{gray} 1} {\color{gray} 1} {\color{gray} 1} {\color{gray} 1} {\color{gray} 1} {\color{gray} 1} 3 \underline{4} \underline{\underline{4} 2 3} {\color{gray} 1} {\color{gray} 1} {\color{gray} 1} {\color{gray} 1} {\color{gray} 1} {\color{gray} 1} {\color{gray} 1} {\color{gray} 1} {\color{gray} 1} {\color{gray} 1} {\color{gray} 1} {\color{gray} 1} {\color{gray} 1} \\
t = 5 & \cdots {\color{gray} 1} {\color{gray} 1} {\color{gray} 1} {\color{gray} 1} {\color{gray} 1} {\color{gray} 1} {\color{gray} 1} {\color{gray} 1} {\color{gray} 1} {\color{gray} 1} {\color{gray} 1} {\color{gray} 1} {\color{gray} 1} {\color{gray} 1} {\color{gray} 1} {\color{gray} 1} {\color{gray} 1} {\color{gray} 1} {\color{gray} 1} {\color{gray} 1} 3 4 \underline{\underline{4 {\color{gray} 1}} 2} 3 {\color{gray} 1} {\color{gray} 1} {\color{gray} 1} {\color{gray} 1} {\color{gray} 1} {\color{gray} 1} {\color{gray} 1} {\color{gray} 1} {\color{gray} 1} {\color{gray} 1} {\color{gray} 1} {\color{gray} 1} {\color{gray} 1} {\color{gray} 1} {\color{gray} 1} \\
t = 6 & \cdots {\color{gray} 1} {\color{gray} 1} {\color{gray} 1} {\color{gray} 1} {\color{gray} 1} {\color{gray} 1} {\color{gray} 1} {\color{gray} 1} {\color{gray} 1} {\color{gray} 1} {\color{gray} 1} {\color{gray} 1} {\color{gray} 1} {\color{gray} 1} {\color{gray} 1} {\color{gray} 1} {\color{gray} 1} 3 4 \underline{4 \underline{{\color{gray} 1} {\color{gray} 1}}} 2 3 {\color{gray} 1} {\color{gray} 1} {\color{gray} 1} {\color{gray} 1} {\color{gray} 1} {\color{gray} 1} {\color{gray} 1} {\color{gray} 1} {\color{gray} 1} {\color{gray} 1} {\color{gray} 1} {\color{gray} 1} {\color{gray} 1} {\color{gray} 1} {\color{gray} 1} {\color{gray} 1} {\color{gray} 1} \\
t = 7 & \cdots {\color{gray} 1} {\color{gray} 1} {\color{gray} 1} {\color{gray} 1} {\color{gray} 1} {\color{gray} 1} {\color{gray} 1} {\color{gray} 1} {\color{gray} 1} {\color{gray} 1} {\color{gray} 1} {\color{gray} 1} {\color{gray} 1} {\color{gray} 1} 3 4 \underline{4 {\color{gray} 1} \underline{{\color{gray} 1}}} \underline{{\color{gray} 1}} 2 3 {\color{gray} 1} {\color{gray} 1} {\color{gray} 1} {\color{gray} 1} {\color{gray} 1} {\color{gray} 1} {\color{gray} 1} {\color{gray} 1} {\color{gray} 1} {\color{gray} 1} {\color{gray} 1} {\color{gray} 1} {\color{gray} 1} {\color{gray} 1} {\color{gray} 1} {\color{gray} 1} {\color{gray} 1} {\color{gray} 1} {\color{gray} 1} \\
t = 8 & \cdots {\color{gray} 1} {\color{gray} 1} {\color{gray} 1} {\color{gray} 1} {\color{gray} 1} {\color{gray} 1} {\color{gray} 1} {\color{gray} 1} {\color{gray} 1} {\color{gray} 1} {\color{gray} 1} 3 4 \underline{4 {\color{gray} 1} {\color{gray} 1}} \underline{{\color{gray} 1} {\color{gray} 1}} 2 3 {\color{gray} 1} {\color{gray} 1} {\color{gray} 1} {\color{gray} 1} {\color{gray} 1} {\color{gray} 1} {\color{gray} 1} {\color{gray} 1} {\color{gray} 1} {\color{gray} 1} {\color{gray} 1} {\color{gray} 1} {\color{gray} 1} {\color{gray} 1} {\color{gray} 1} {\color{gray} 1} {\color{gray} 1} {\color{gray} 1} {\color{gray} 1} {\color{gray} 1} {\color{gray} 1} \\
t = 9 & \cdots {\color{gray} 1} {\color{gray} 1} {\color{gray} 1} {\color{gray} 1} {\color{gray} 1} {\color{gray} 1} {\color{gray} 1} {\color{gray} 1} 3 4 \underline{4 {\color{gray} 1} {\color{gray} 1}} {\color{gray} 1} \underline{{\color{gray} 1} {\color{gray} 1}} 2 3 {\color{gray} 1} {\color{gray} 1} {\color{gray} 1} {\color{gray} 1} {\color{gray} 1} {\color{gray} 1} {\color{gray} 1} {\color{gray} 1} {\color{gray} 1} {\color{gray} 1} {\color{gray} 1} {\color{gray} 1} {\color{gray} 1} {\color{gray} 1} {\color{gray} 1} {\color{gray} 1} {\color{gray} 1} {\color{gray} 1} {\color{gray} 1} {\color{gray} 1} {\color{gray} 1} {\color{gray} 1} {\color{gray} 1} \\
t = 10 & \cdots {\color{gray} 1} {\color{gray} 1} {\color{gray} 1} {\color{gray} 1} {\color{gray} 1} 3 4 \underline{4 {\color{gray} 1} {\color{gray} 1}} {\color{gray} 1} {\color{gray} 1} \underline{{\color{gray} 1} {\color{gray} 1}} 2 3 {\color{gray} 1} {\color{gray} 1} {\color{gray} 1} {\color{gray} 1} {\color{gray} 1} {\color{gray} 1} {\color{gray} 1} {\color{gray} 1} {\color{gray} 1} {\color{gray} 1} {\color{gray} 1} {\color{gray} 1} {\color{gray} 1} {\color{gray} 1} {\color{gray} 1} {\color{gray} 1} {\color{gray} 1} {\color{gray} 1} {\color{gray} 1} {\color{gray} 1} {\color{gray} 1} {\color{gray} 1} {\color{gray} 1} {\color{gray} 1} {\color{gray} 1} \\
\end{array}}
\]
We underline the positions how the solitons propagate under no interaction.
We note that the phase shift is $\pm 2$.
Under Conjecture~\ref{conj:SCA}(\ref{conj:SCA_bijection}), the two solitons correspond to $23 \otimes 123 \in B^{1,2} \otimes B^{1,3}$ in type $A_2^{(1)}$.
We have $R(23 \otimes 123) = 233 \otimes 12$ and $H(23 \otimes 123) = 2$, which agrees with Part~(\ref{conj:SCA_scattering}) and Part~(\ref{conj:SCA_phase_shift}).
\end{ex}

\subsection{Rigged configurations}
%
Denote $\HH_0 := I_0 \times \ZZ_{>0}$. Fix a tensor product of KR crystals $B = \bigotimes_{k=1}^N B^{r_k, s_k}$.
A \defn{configuration} $\nu = \bigl(\nu^{(i)}\bigr)_{i \in I_0}$ is a sequence of partitions.
Let $m_{\ell}^{(i)}$ denote the multiplicity of $\ell$ in $\nu^{(i)}$.
Define the \defn{vacancy numbers}
\begin{equation}
\label{eq:vacancy}
p_{\ell}^{(i)}(\nu; B) := \sum_{k \in \ZZ_{>0}} L_k^{(i)} \min(\ell, k) - \sum_{(j,k) \in \HH_0} \frac{A_{ij}}{\gamma_j} \min(\modgamma_i \ell, \modgamma_j k) m_k^{(j)},
\end{equation}
where $L_s^{(r)}$ equals the number of factors $B^{r,s}$ that occur in $B$.
When there is no danger of confusion, we will simply write $p_{\ell}^{(i)}$.

\begin{table}
\[
\begin{array}{|c|c|}
\hline
\g & (\gamma_i)_{i \in I} \\\hline
\text{dual untwisted}  & (1,\dotsc,1) \\
B_n^{(1)} & (1, 1, 2, \dotsc, 2) \\
C_n^{(1)} & (2, 1, \dotsc,1, 2) \\
A_{2n}^{(2)} & (1, \dotsc, 1, 2) \\
A_{2n}^{(2)\dagger} & (2, 1, \dotsc, 1) \\
F_4^{(1)} & (2, 2, 2, 1, 1) \\
G_2^{(1)} & (3, 1, 3) \\\hline
\end{array}
\]
\caption{The factors $(\gamma_i)_{i \in I}$.}
\label{table:scaling_factors}
\end{table}

A \defn{$B$-rigged configuration} is the pair $(\nu, J)$, where $\nu$ is a configuration and $J = (J_{\ell}^{(i)})_{(i,\ell) \in \HH_0}$ be such that $J_{\ell}^{(i)}$ is a multiset $\{x \in \ZZ \mid x \leq p_{\ell}^{(i)}(\nu; B) \}$\footnote{If $\g = A_{2n}^{(2)\dagger}$ and $i = n$ and $i$ odd, then we take $x \in \ZZ + \frac{1}{2}$.} with $\lvert J_{\ell}^{(i)} \rvert = m_{\ell}^{(i)}$ for all $(i, \ell) \in \HH_0$.
The integers in $J_{\ell}^{(i)}$ are called \defn{riggings} or \defn{labels}, and we can associate each rigging in $J_{\ell}^{(i)}$ to a row of length $\ell$ in $\nu^{(i)}$.
The \defn{corigging} or \defn{colabel} of a rigging $x \in J_{\ell}^{(i)}$ is defined as $p_{\ell}^{(i)} - x$.
For any $K \subseteq I$, a \defn{$K$-highest weight} $B$-rigged configuration is a rigged configuration $(\nu, J)$ such that $\min J_{\ell}^{(i)} \geq 0$ (we define the minimum to be $0$ when $\lvert J_{\ell}^{(i)} \rvert = 0$) for all $(i, \ell) \in K \times \ZZ_{>0}$.
When $K = I$, we say the $B$-rigged configuration is \defn{highest weight}.
Let $\hwRC(B)$ denote the set of all highest weight $B$-rigged configurations.
When $B$ is clear, we call a $B$-rigged configuration simply a rigged configuration.

Next, let $\RC(B)$ denote the closure of $\hwRC(B)$ under the following crystal operators.
For simplicity of the exposition, we consider $\g$ to not be of type $A_{2n}^{(2)}$ nor $A_{2n}^{(2)\dagger}$ and refer the reader to~\cite[Def.~3.1]{SchillingS15}.
Fix a $B$-rigged configuration $(\nu, J)$ and $i \in I_0$.
For simplicity, we assume there exists a row of length $0$ in $\nu^{(a)}$ with a rigging of $0$.
Let $x = \min \{ \min J_{\ell}^{(i)} \mid \ell \in \ZZ_{>0} \}$; i.e., the smallest rigging in $\nu^{(i)}$.
\begin{description}
\item[\defn{$e_i$}] If $x = 0$, then define $e_i(\nu, J) = 0$.
Otherwise, remove a box from the smallest row with rigging $x$, replace that label with $x + 1$, and change all other riggings so that the coriggings remain fixed.
The result is $e_i(\nu, J)$.

\item[\defn{$f_i$}] If the smallest corigging of $\nu^{(i)}$ is $0$, then define $f_i(\nu, J) = 0$.
Otherwise, add a box from the largest row with rigging $x$, replace that label with $x - 1$, and change all other riggings so that the coriggings remain fixed.
The result is $f_i(\nu, J)$.
\end{description}
We finish the $U_q(\g_0)$-crystal structure on $\RC(B)$ by defining
\[
\overline{\wt}(\nu, J) = \sum_{(i, \ell) \in \HH_0} \ell \left( L_{\ell}^{(i)} \clfw_i - m_{\ell}^{(i)} \clsr_i \right).
\]

\begin{thm}[{\cite{S06,SchillingS15}}]
\label{thm:rc_crystal}
Let $B$ be a tensor product of KR crystals. Fix some $(\nu, J) \in \hwRC(B)$. Let $X_{(\nu, J)}$ denote the closure of $(\nu, J)$ under $e_i$ and $f_i$ for all $a \in I_0$.
Then, we have $X_{(\nu, J)} \iso B(\lambda)$, where $\lambda = \overline{\wt}(\nu, J)$.
\end{thm}

%
%

Rigged configurations also come with a natural statistics from mathematical physics called \defn{cocharge} given by
\[
\cc(\nu, J) = \frac{1}{2} \sum_{\substack{(i, \ell) \in \HH_0 \\ (j, k) \in \HH_0}} \frac{t_i^{\vee} A_{ij}}{\gamma_k} \min(\modgamma_i \ell, \modgamma_j k) m_{\ell}^{(i)} m_k^{(j)} + \sum_{(i, \ell) \in \HH_0} t_i^{\vee} \sum_{x \in J_{\ell}^{(i)}} x.
\]
Cocharge is invariant under the classical crystal operators.
\begin{prop}[{\cite{S06, SchillingS15}}]
\label{prop:cocharge_classical_invar}
Fix a classical component $X_{(\nu, J)}$ as given in Theorem~\ref{thm:rc_crystal}.
Cocharge is constant on $X_{(\nu, J)}$.
\end{prop}

Next, fix $B = (B^{r,1})^{\otimes N}$, where $r$ is a minuscule node.
We define a bijection $\Phi \colon \RC(B) \to B$ by repeating the process below for each factor from left to right.
Start at $b = u_{\clfw_r}$, and set $\ell_0 = 1$.
Consider step $j$.
Let $\ell_j$ denote the minimal $k_i \geq \ell_{j-1}$ over all $i \in I_0$ such that $f_i b \neq 0$ and $\nu^{(i)}$ has a singular row of length $k_i$ that has not been previously selected.
If no such row exists, terminate, set all $\ell_{j'} = \infty$ for $j' \geq j$, and return $b$.
Otherwise, select such a row and repeat the above with $f_{i'}(b)$, where $\nu^{(i')}$ is the partition that the row was selected from.
We then construct $\delta(\nu, J)$ by removing a box from all selected rows and keeping those rows singular.

We sketch the inverse bijection $\Phi^{-1}$ when $r$ is a minuscule node.
This is given by adding factors right to left by starting at $b$ and finding a path to $u_{\clfw_r}$ by selecting the largest singular rows that were at most as large as the previously selected row.
We terminate when we reach $u_{\clfw_r}$.

When $r \sim 0$, the bijection $\Phi$ is similar to when $r$ is minuscule except we are allowed to select a row twice when at a negative root and we can select a quasisingular row, a row such that the rigging equals $p_{\ell}^{(i)} - 1$ and $\max J_{\ell}^{(i)} = p_{\ell}^{(i)} - 1$, when going into $y_i$ for some $i \in I_0$.
For the remaining nodes, we need the box-splitting map, which we do not describe here.
Instead, for a precise description of $\Phi$, we refer the reader to~\cite{KSS02, OSSS16, OSS17, Scrimshaw17}.

We recall some conjectural properties of the bijection $\Phi$ (see, \textit{e.g.},~\cite{SchillingS15}).
We will need the map $\theta \colon \RC(B) \to \RC(B)$ on highest weight rigged configurations that sends every rigging to its corresponding corigging and extended as a classical crystal isomorphism.

\begin{conj}
\label{conj:bijection}
Let $B = \bigotimes_{k=1}^N B^{r_k,s_k}$.
The map $\Phi \colon \RC(B) \to B$ is a classical crystal isomorphism such that $\cc = D \circ \Phi \circ \theta$.
\end{conj}

\begin{conj}
\label{conj:R_matrix_identity}
Let $B = \bigotimes_{k=1}^N B^{r_k,s_k}$.
The diagram
\[
\xymatrixcolsep{4em}
\xymatrixrowsep{3em}
\xymatrix{
\RC(B) \ar[r]^{\id} \ar[d]_{\Phi} & \RC(B') \ar[d]^{\Phi}
\\ B \ar[r]_{R} & B'
}
\]
commutes for any reordering $B' = \bigotimes_{k=1}^N B^{r'_k,s'_k}$.
\end{conj}

It is known that Conjecture~\ref{conj:bijection} and Conjecture~\ref{conj:R_matrix_identity} hold in general for nonexceptional types $A_n^{(1)}$~\cite{KSS02,OSSS16,OSS17}.
These conjectures have been proven in a number of special cases for exceptional types~\cite{OS12,SchillingS15,SS2006,Scrimshaw15,Scrimshaw17}.

\section{Solitons}
\label{sec:solitons}

In this section, we describe the solitons of length $\ell$ by using tensor products of KR crystals in a number of special cases.
Two particular cases we cover in general are when $r$ is minuscule and $r \sim 0$.

We adopt the following notation.
Let $(\varpi_r)_{r \in I_{0,r}}$ denote the fundamental weights of type $\g_{0,r}$.
Let $\varpi = \sum_{r' \sim r} \varpi_{r'}$ and $u_{\varpi} = f_r u_{\clfw_r}$.
Note that this is a slight abuse of notation, but $u_{\varpi}$ is the unique $I_{0,r}$-highest weight element of the unique $U_q(\g_{0,r})$-crystal $B(\varpi) \subseteq B(\clfw_r)$.
Let $v_{\varpi}$ denote the $I_{0,r}$-lowest weight element of $B(\varpi) \subseteq B(\clfw_r)$.
Let $u = u(B^{r,1})$ and $v  = v(B^{r,1})$.

\begin{prop}
\label{prop:soliton_H1}
Let $r$ be a special node or $r = 1$ for $\g$ of type $B_n^{(1)}$, $C_n^{(1)}$, $A_{2n}^{(2)\dagger}$, or $D_4^{(3)}$.
For a soliton $p = b_{-L} \otimes \cdots \otimes b_0 \in \Sol^{(r)}(\ell)$, we have
\[
u < b_{-L} \leq \cdots \leq b_0 \leq w = \begin{cases}
v & \text{if $\g = C_n^{(1)}, A_{2n}^{(2)\dagger}$ and $r = 1$},
\\ v_{\varpi} & \text{otherwise},
\end{cases}
\]
where $y_r$ appears at most once.
Moreover, we have $\ell = L + m_v + 1$, where $m_v$ is the number of factors equal to $v$ in $p$.
\end{prop}

\begin{proof}
We claim $H(b_0 \otimes u) = 1$ if and only if $b_0 \neq u$ in type $C_n^{(1)}$, $A_{2n}^{(2)\dagger}$ or $b_0 \in B(\varpi)$ otherwise.
First, note that $f_r(b_0 \otimes u) = b_0 \otimes (f_r u)$ and $f_i(b_0 \otimes u) = (f_i b_0) \otimes u$ for all $i \in I_{0,r}$.
Recall that $H$ is constant on classical components.
Thus, it remains to show the claim on all $I_{0,r}$-highest weight elements, and we show this case by case.
For special nodes, the claim follows from Theorem~\ref{thm:minuscule_local_energy}.
For type $B_n^{(1)}$, we have $f_0 f_2 f_3 \dotsm f_n f_n  \dotsm f_2 (u_{\varpi} \otimes u) = u \otimes u$ and $f_0(v \otimes u) = u_{\varpi} \otimes u$.
For type $C_n^{(1)}$, we have $f_0 f_1^2 f_2 \dotsm f_n \dotsm f_2 (u_{\varpi} \otimes u) = u \otimes u_{\varpi}$ and $f_0(v \otimes u) = u \otimes u$.
For type $A_{2n}^{(2)\dagger}$, we have $f_0 f_1^2 f_2 \dotsm f_n f_n \dotsm f_2 (u_{\varpi} \otimes u) = u \otimes u_{\varpi}$ and $f_0(v \otimes u) = u \otimes u$.
For types $D_4^{(3)}$, this is a finite computation.
Note that $b_0 \neq u$ by the definition of soliton.

We have $H(u_{\varpi} \otimes u) = 1$ in all cases, and thus we require $E_1(p) = 1$.
Since the local energy is nonnegative and $E_1(p)$ is the sum of local energy, if $H(b_0 \otimes u) \geq 2$, we will have  $E_1(p) \geq 2$, so $p$ is not a soliton.
Therefore, in order to have $E_1(p) = 1$, we need $H(b_0 \otimes u) = 1$.

Next, we must have $H(b_{-1} \otimes b_0) = 0$ in order to have $E_1(p) = 1$.
Recall that this implies $b_{-1} \otimes b_0$ is in the $I_0$-connected component of $u \otimes u$, which is isomorphic to $B(2\clfw_r)$.
From Proposition~\ref{prop:minuscule_description}, we have $b_{-1} \leq b_0$.
Similarly, we must have $b_j \leq b_{j+1}$ for all $-L \leq j < 0$.

From the definition of the length of a soliton, we have $\ell = L + m_v + 1$.
\end{proof}

\begin{ex}
\label{ex:SCA_D43}
Consider type $D_4^{(3)}$ and $r = 1$, and note $v = \bon$, $v_{\omega_1} = \btw$, and $y_1 = 0$.
Consider an SCA starting with a state in $\Sol^{(1)}(1,2,3)$ of:
\[
{\begin{array}{c|c}
t = 0 & \cdots {\color{gray} 1} {\color{gray} 1} {\color{gray} 1} {\color{gray} 1} {\color{gray} 1} {\color{gray} 1} {\color{gray} 1} {\color{gray} 1} {\color{gray} 1} {\color{gray} 1} {\color{gray} 1} {\color{gray} 1} {\color{gray} 1} {\color{gray} 1} {\color{gray} 1} {\color{gray} 1} {\color{gray} 1} {\color{gray} 1} {\color{gray} 1} {\color{gray} 1} {\color{gray} 1} {\color{gray} 1} {\color{gray} 1} {\color{gray} 1} {\color{gray} 1} {\color{gray} 1} {\color{gray} 1} {\color{gray} 1} {\color{gray} 1} 3 {\color{gray} 1} {\color{gray} 1} {\color{gray} 1} 2 3 {\color{gray} 1} {\color{gray} 1} {\color{gray} 1} 2 2 2 {\color{gray} 1} \\
t = 1 & \cdots {\color{gray} 1} {\color{gray} 1} {\color{gray} 1} {\color{gray} 1} {\color{gray} 1} {\color{gray} 1} {\color{gray} 1} {\color{gray} 1} {\color{gray} 1} {\color{gray} 1} {\color{gray} 1} {\color{gray} 1} {\color{gray} 1} {\color{gray} 1} {\color{gray} 1} {\color{gray} 1} {\color{gray} 1} {\color{gray} 1} {\color{gray} 1} {\color{gray} 1} {\color{gray} 1} {\color{gray} 1} {\color{gray} 1} {\color{gray} 1} {\color{gray} 1} {\color{gray} 1} {\color{gray} 1} {\color{gray} 1} 3 {\color{gray} 1} {\color{gray} 1} 2 3 {\color{gray} 1} {\color{gray} 1} 2 2 2 {\color{gray} 1} {\color{gray} 1} {\color{gray} 1} {\color{gray} 1} \\
t = 2 & \cdots {\color{gray} 1} {\color{gray} 1} {\color{gray} 1} {\color{gray} 1} {\color{gray} 1} {\color{gray} 1} {\color{gray} 1} {\color{gray} 1} {\color{gray} 1} {\color{gray} 1} {\color{gray} 1} {\color{gray} 1} {\color{gray} 1} {\color{gray} 1} {\color{gray} 1} {\color{gray} 1} {\color{gray} 1} {\color{gray} 1} {\color{gray} 1} {\color{gray} 1} {\color{gray} 1} {\color{gray} 1} {\color{gray} 1} {\color{gray} 1} {\color{gray} 1} {\color{gray} 1} {\color{gray} 1} 3 {\color{gray} 1} 2 3 {\color{gray} 1} 2 2 2 {\color{gray} 1} {\color{gray} 1} {\color{gray} 1} {\color{gray} 1} {\color{gray} 1} {\color{gray} 1} {\color{gray} 1} \\
t = 3 & \cdots {\color{gray} 1} {\color{gray} 1} {\color{gray} 1} {\color{gray} 1} {\color{gray} 1} {\color{gray} 1} {\color{gray} 1} {\color{gray} 1} {\color{gray} 1} {\color{gray} 1} {\color{gray} 1} {\color{gray} 1} {\color{gray} 1} {\color{gray} 1} {\color{gray} 1} {\color{gray} 1} {\color{gray} 1} {\color{gray} 1} {\color{gray} 1} {\color{gray} 1} {\color{gray} 1} {\color{gray} 1} {\color{gray} 1} {\color{gray} 1} {\color{gray} 1} {\color{gray} 1} 3 2 3 2 2 2 {\color{gray} 1} {\color{gray} 1} {\color{gray} 1} {\color{gray} 1} {\color{gray} 1} {\color{gray} 1} {\color{gray} 1} {\color{gray} 1} {\color{gray} 1} {\color{gray} 1} \\
t = 4 & \cdots {\color{gray} 1} {\color{gray} 1} {\color{gray} 1} {\color{gray} 1} {\color{gray} 1} {\color{gray} 1} {\color{gray} 1} {\color{gray} 1} {\color{gray} 1} {\color{gray} 1} {\color{gray} 1} {\color{gray} 1} {\color{gray} 1} {\color{gray} 1} {\color{gray} 1} {\color{gray} 1} {\color{gray} 1} {\color{gray} 1} {\color{gray} 1} {\color{gray} 1} {\color{gray} 1} {\color{gray} 1} {\color{gray} 1} {\color{gray} 1} {\color{gray} 1} 0 0 2 2 {\color{gray} 1} {\color{gray} 1} {\color{gray} 1} {\color{gray} 1} {\color{gray} 1} {\color{gray} 1} {\color{gray} 1} {\color{gray} 1} {\color{gray} 1} {\color{gray} 1} {\color{gray} 1} {\color{gray} 1} {\color{gray} 1} \\
t = 5 & \cdots {\color{gray} 1} {\color{gray} 1} {\color{gray} 1} {\color{gray} 1} {\color{gray} 1} {\color{gray} 1} {\color{gray} 1} {\color{gray} 1} {\color{gray} 1} {\color{gray} 1} {\color{gray} 1} {\color{gray} 1} {\color{gray} 1} {\color{gray} 1} {\color{gray} 1} {\color{gray} 1} {\color{gray} 1} {\color{gray} 1} {\color{gray} 1} {\color{gray} 1} {\color{gray} 1} {\color{gray} 1} {\color{gray} 1} 2 \overline{1} 2 {\color{gray} 1} {\color{gray} 1} {\color{gray} 1} {\color{gray} 1} {\color{gray} 1} {\color{gray} 1} {\color{gray} 1} {\color{gray} 1} {\color{gray} 1} {\color{gray} 1} {\color{gray} 1} {\color{gray} 1} {\color{gray} 1} {\color{gray} 1} {\color{gray} 1} {\color{gray} 1} \\
t = 6 & \cdots {\color{gray} 1} {\color{gray} 1} {\color{gray} 1} {\color{gray} 1} {\color{gray} 1} {\color{gray} 1} {\color{gray} 1} {\color{gray} 1} {\color{gray} 1} {\color{gray} 1} {\color{gray} 1} {\color{gray} 1} {\color{gray} 1} {\color{gray} 1} {\color{gray} 1} {\color{gray} 1} {\color{gray} 1} {\color{gray} 1} {\color{gray} 1} {\color{gray} 1} 2 0 \emptyset {\color{gray} 1} 2 {\color{gray} 1} {\color{gray} 1} {\color{gray} 1} {\color{gray} 1} {\color{gray} 1} {\color{gray} 1} {\color{gray} 1} {\color{gray} 1} {\color{gray} 1} {\color{gray} 1} {\color{gray} 1} {\color{gray} 1} {\color{gray} 1} {\color{gray} 1} {\color{gray} 1} {\color{gray} 1} {\color{gray} 1} \\
t = 7 & \cdots {\color{gray} 1} {\color{gray} 1} {\color{gray} 1} {\color{gray} 1} {\color{gray} 1} {\color{gray} 1} {\color{gray} 1} {\color{gray} 1} {\color{gray} 1} {\color{gray} 1} {\color{gray} 1} {\color{gray} 1} {\color{gray} 1} {\color{gray} 1} {\color{gray} 1} {\color{gray} 1} {\color{gray} 1} 2 3 \overline{3} {\color{gray} 1} {\color{gray} 1} {\color{gray} 1} 2 {\color{gray} 1} {\color{gray} 1} {\color{gray} 1} {\color{gray} 1} {\color{gray} 1} {\color{gray} 1} {\color{gray} 1} {\color{gray} 1} {\color{gray} 1} {\color{gray} 1} {\color{gray} 1} {\color{gray} 1} {\color{gray} 1} {\color{gray} 1} {\color{gray} 1} {\color{gray} 1} {\color{gray} 1} {\color{gray} 1} \\
t = 8 & \cdots {\color{gray} 1} {\color{gray} 1} {\color{gray} 1} {\color{gray} 1} {\color{gray} 1} {\color{gray} 1} {\color{gray} 1} {\color{gray} 1} {\color{gray} 1} {\color{gray} 1} {\color{gray} 1} {\color{gray} 1} {\color{gray} 1} {\color{gray} 1} 2 3 0 2 {\color{gray} 1} {\color{gray} 1} {\color{gray} 1} {\color{gray} 1} 2 {\color{gray} 1} {\color{gray} 1} {\color{gray} 1} {\color{gray} 1} {\color{gray} 1} {\color{gray} 1} {\color{gray} 1} {\color{gray} 1} {\color{gray} 1} {\color{gray} 1} {\color{gray} 1} {\color{gray} 1} {\color{gray} 1} {\color{gray} 1} {\color{gray} 1} {\color{gray} 1} {\color{gray} 1} {\color{gray} 1} {\color{gray} 1} \\
t = 9 & \cdots {\color{gray} 1} {\color{gray} 1} {\color{gray} 1} {\color{gray} 1} {\color{gray} 1} {\color{gray} 1} {\color{gray} 1} {\color{gray} 1} {\color{gray} 1} {\color{gray} 1} {\color{gray} 1} 2 3 3 2 2 {\color{gray} 1} {\color{gray} 1} {\color{gray} 1} {\color{gray} 1} {\color{gray} 1} 2 {\color{gray} 1} {\color{gray} 1} {\color{gray} 1} {\color{gray} 1} {\color{gray} 1} {\color{gray} 1} {\color{gray} 1} {\color{gray} 1} {\color{gray} 1} {\color{gray} 1} {\color{gray} 1} {\color{gray} 1} {\color{gray} 1} {\color{gray} 1} {\color{gray} 1} {\color{gray} 1} {\color{gray} 1} {\color{gray} 1} {\color{gray} 1} {\color{gray} 1} \\
t = 10 & \cdots {\color{gray} 1} {\color{gray} 1} {\color{gray} 1} {\color{gray} 1} {\color{gray} 1} {\color{gray} 1} {\color{gray} 1} {\color{gray} 1} 2 3 3 {\color{gray} 1} 2 2 {\color{gray} 1} {\color{gray} 1} {\color{gray} 1} {\color{gray} 1} {\color{gray} 1} {\color{gray} 1} 2 {\color{gray} 1} {\color{gray} 1} {\color{gray} 1} {\color{gray} 1} {\color{gray} 1} {\color{gray} 1} {\color{gray} 1} {\color{gray} 1} {\color{gray} 1} {\color{gray} 1} {\color{gray} 1} {\color{gray} 1} {\color{gray} 1} {\color{gray} 1} {\color{gray} 1} {\color{gray} 1} {\color{gray} 1} {\color{gray} 1} {\color{gray} 1} {\color{gray} 1} {\color{gray} 1} \\
t = 11 & \cdots {\color{gray} 1} {\color{gray} 1} {\color{gray} 1} {\color{gray} 1} {\color{gray} 1} 2 3 3 {\color{gray} 1} {\color{gray} 1} 2 2 {\color{gray} 1} {\color{gray} 1} {\color{gray} 1} {\color{gray} 1} {\color{gray} 1} {\color{gray} 1} {\color{gray} 1} 2 {\color{gray} 1} {\color{gray} 1} {\color{gray} 1} {\color{gray} 1} {\color{gray} 1} {\color{gray} 1} {\color{gray} 1} {\color{gray} 1} {\color{gray} 1} {\color{gray} 1} {\color{gray} 1} {\color{gray} 1} {\color{gray} 1} {\color{gray} 1} {\color{gray} 1} {\color{gray} 1} {\color{gray} 1} {\color{gray} 1} {\color{gray} 1} {\color{gray} 1} {\color{gray} 1} {\color{gray} 1} \\
t = 12 & \cdots  {\color{gray} 1} {\color{gray} 1} 2 3 3 {\color{gray} 1} {\color{gray} 1} {\color{gray} 1} 2 2 {\color{gray} 1} {\color{gray} 1} {\color{gray} 1} {\color{gray} 1} {\color{gray} 1} {\color{gray} 1} {\color{gray} 1} {\color{gray} 1} 2 {\color{gray} 1} {\color{gray} 1} {\color{gray} 1} {\color{gray} 1} {\color{gray} 1} {\color{gray} 1} {\color{gray} 1} {\color{gray} 1} {\color{gray} 1} {\color{gray} 1} {\color{gray} 1} {\color{gray} 1} {\color{gray} 1} {\color{gray} 1} {\color{gray} 1} {\color{gray} 1} {\color{gray} 1} {\color{gray} 1} {\color{gray} 1} {\color{gray} 1} {\color{gray} 1} {\color{gray} 1} {\color{gray} 1} \\
\end{array}}
\]
\end{ex}

\begin{prop}
\label{prop:soliton_H2}
Let $r = 1$ and $\g$ of type $D_{n+1}^{(2)}$ or $A_{2n}^{(2)}$.
For a soliton $p = b_{-L} \otimes \cdots \otimes b_0 \in \Sol^{(r)}(\ell)$, we have
\begin{align*}
u < b_{-L} \leq \cdots \leq b_0
\end{align*}
where $y_1$ (and $\emptyset$) appears at most once and we consider $v < \emptyset$.
Moreover, we have $\ell = L + m_v + 1$.
\end{prop}

\begin{proof}
In both of these cases, we have $1 \sim 0$, and so the local energy is given by Theorem~\ref{thm:adjoint_local_energy}.
In particular, we have $H(u_{\varpi} \otimes u) = 2$.
Note that $H(\emptyset \otimes \emptyset) = 2$ and $H(b \otimes \emptyset) = H(\emptyset \otimes b) = 1$ for all $b \in B(\clfw_r)$.
The remainder of the proof is similar to the proof of Proposition~\ref{prop:soliton_H1}.
\end{proof}

We note that the description of a soliton for $r = 1$ with $\g$ of nonexceptional type in Proposition~\ref{prop:soliton_H1} and Proposition~\ref{prop:soliton_H2} agrees with the description from~\cite[Sec.~2.3]{HKOTY02} by~\cite[Lemma~2.5]{HKOTY02}.
Proposition~\ref{prop:soliton_H1} also covers the cases considered by~\cite{binMohamad12,Yamada04,Yamada07}.
The proofs are similar to ours except we use the uniform statements Theorem~\ref{thm:minuscule_local_energy} and Theorem~\ref{thm:adjoint_local_energy} to compute the local energy instead of type specific computations.

\begin{ex}
\label{ex:SCA_Dtwisted}
Consider an SCA starting with a state in $\Sol^{(1)}(1, 2, 4)$ of type $D_4^{(2)}$:
\[
{\begin{array}{c|c}
t = 0 & \cdots {\color{gray} 1} {\color{gray} 1} {\color{gray} 1} {\color{gray} 1} {\color{gray} 1} {\color{gray} 1} {\color{gray} 1} {\color{gray} 1} {\color{gray} 1} {\color{gray} 1} {\color{gray} 1} {\color{gray} 1} {\color{gray} 1} {\color{gray} 1} {\color{gray} 1} {\color{gray} 1} {\color{gray} 1} {\color{gray} 1} {\color{gray} 1} {\color{gray} 1} {\color{gray} 1} {\color{gray} 1} {\color{gray} 1} {\color{gray} 1} {\color{gray} 1} {\color{gray} 1} {\color{gray} 1} {\color{gray} 1} {\color{gray} 1} {\color{gray} 1} {\color{gray} 1} \emptyset {\color{gray} 1} {\color{gray} 1} {\color{gray} 1} 3 0 {\color{gray} 1} {\color{gray} 1} {\color{gray} 1} {\color{gray} 1} 2 \overline{3} \overline{1} {\color{gray} 1} \\
t = 1 & \cdots {\color{gray} 1} {\color{gray} 1} {\color{gray} 1} {\color{gray} 1} {\color{gray} 1} {\color{gray} 1} {\color{gray} 1} {\color{gray} 1} {\color{gray} 1} {\color{gray} 1} {\color{gray} 1} {\color{gray} 1} {\color{gray} 1} {\color{gray} 1} {\color{gray} 1} {\color{gray} 1} {\color{gray} 1} {\color{gray} 1} {\color{gray} 1} {\color{gray} 1} {\color{gray} 1} {\color{gray} 1} {\color{gray} 1} {\color{gray} 1} {\color{gray} 1} {\color{gray} 1} {\color{gray} 1} {\color{gray} 1} {\color{gray} 1} {\color{gray} 1} \emptyset {\color{gray} 1} {\color{gray} 1} 3 0 {\color{gray} 1} {\color{gray} 1} 2 \overline{3} \overline{1} {\color{gray} 1} {\color{gray} 1} {\color{gray} 1} {\color{gray} 1} {\color{gray} 1} \\
t = 2 & \cdots {\color{gray} 1} {\color{gray} 1} {\color{gray} 1} {\color{gray} 1} {\color{gray} 1} {\color{gray} 1} {\color{gray} 1} {\color{gray} 1} {\color{gray} 1} {\color{gray} 1} {\color{gray} 1} {\color{gray} 1} {\color{gray} 1} {\color{gray} 1} {\color{gray} 1} {\color{gray} 1} {\color{gray} 1} {\color{gray} 1} {\color{gray} 1} {\color{gray} 1} {\color{gray} 1} {\color{gray} 1} {\color{gray} 1} {\color{gray} 1} {\color{gray} 1} {\color{gray} 1} {\color{gray} 1} {\color{gray} 1} {\color{gray} 1} \emptyset 3 0 \overline{3} {\color{gray} 1} 2 \overline{1} {\color{gray} 1} {\color{gray} 1} {\color{gray} 1} {\color{gray} 1} {\color{gray} 1} {\color{gray} 1} {\color{gray} 1} {\color{gray} 1} {\color{gray} 1} \\
t = 3 & \cdots {\color{gray} 1} {\color{gray} 1} {\color{gray} 1} {\color{gray} 1} {\color{gray} 1} {\color{gray} 1} {\color{gray} 1} {\color{gray} 1} {\color{gray} 1} {\color{gray} 1} {\color{gray} 1} {\color{gray} 1} {\color{gray} 1} {\color{gray} 1} {\color{gray} 1} {\color{gray} 1} {\color{gray} 1} {\color{gray} 1} {\color{gray} 1} {\color{gray} 1} {\color{gray} 1} {\color{gray} 1} {\color{gray} 1} {\color{gray} 1} {\color{gray} 1} 3 0 \overline{3} \emptyset \overline{2} {\color{gray} 1} 2 2 {\color{gray} 1} {\color{gray} 1} {\color{gray} 1} {\color{gray} 1} {\color{gray} 1} {\color{gray} 1} {\color{gray} 1} {\color{gray} 1} {\color{gray} 1} {\color{gray} 1} {\color{gray} 1} {\color{gray} 1} \\
t = 4 & \cdots {\color{gray} 1} {\color{gray} 1} {\color{gray} 1} {\color{gray} 1} {\color{gray} 1} {\color{gray} 1} {\color{gray} 1} {\color{gray} 1} {\color{gray} 1} {\color{gray} 1} {\color{gray} 1} {\color{gray} 1} {\color{gray} 1} {\color{gray} 1} {\color{gray} 1} {\color{gray} 1} {\color{gray} 1} {\color{gray} 1} {\color{gray} 1} {\color{gray} 1} {\color{gray} 1} 3 0 \overline{3} \emptyset {\color{gray} 1} {\color{gray} 1} {\color{gray} 1} \overline{2} 2 2 {\color{gray} 1} {\color{gray} 1} {\color{gray} 1} {\color{gray} 1} {\color{gray} 1} {\color{gray} 1} {\color{gray} 1} {\color{gray} 1} {\color{gray} 1} {\color{gray} 1} {\color{gray} 1} {\color{gray} 1} {\color{gray} 1} {\color{gray} 1} \\
t = 5 & \cdots {\color{gray} 1} {\color{gray} 1} {\color{gray} 1} {\color{gray} 1} {\color{gray} 1} {\color{gray} 1} {\color{gray} 1} {\color{gray} 1} {\color{gray} 1} {\color{gray} 1} {\color{gray} 1} {\color{gray} 1} {\color{gray} 1} {\color{gray} 1} {\color{gray} 1} {\color{gray} 1} {\color{gray} 1} 3 0 \overline{3} \emptyset {\color{gray} 1} {\color{gray} 1} {\color{gray} 1} {\color{gray} 1} {\color{gray} 1} {\color{gray} 1} \overline{1} 2 {\color{gray} 1} {\color{gray} 1} {\color{gray} 1} {\color{gray} 1} {\color{gray} 1} {\color{gray} 1} {\color{gray} 1} {\color{gray} 1} {\color{gray} 1} {\color{gray} 1} {\color{gray} 1} {\color{gray} 1} {\color{gray} 1} {\color{gray} 1} {\color{gray} 1} {\color{gray} 1} \\
t = 6 & \cdots {\color{gray} 1} {\color{gray} 1} {\color{gray} 1} {\color{gray} 1} {\color{gray} 1} {\color{gray} 1} {\color{gray} 1} {\color{gray} 1} {\color{gray} 1} {\color{gray} 1} {\color{gray} 1} {\color{gray} 1} {\color{gray} 1} 3 0 \overline{3} \emptyset {\color{gray} 1} {\color{gray} 1} {\color{gray} 1} {\color{gray} 1} {\color{gray} 1} {\color{gray} 1} {\color{gray} 1} {\color{gray} 1} \overline{1} {\color{gray} 1} 2 {\color{gray} 1} {\color{gray} 1} {\color{gray} 1} {\color{gray} 1} {\color{gray} 1} {\color{gray} 1} {\color{gray} 1} {\color{gray} 1} {\color{gray} 1} {\color{gray} 1} {\color{gray} 1} {\color{gray} 1} {\color{gray} 1} {\color{gray} 1} {\color{gray} 1} {\color{gray} 1} {\color{gray} 1} \\
t = 7 & \cdots {\color{gray} 1} {\color{gray} 1} {\color{gray} 1} {\color{gray} 1} {\color{gray} 1} {\color{gray} 1} {\color{gray} 1} {\color{gray} 1} {\color{gray} 1} 3 0 \overline{3} \emptyset {\color{gray} 1} {\color{gray} 1} {\color{gray} 1} {\color{gray} 1} {\color{gray} 1} {\color{gray} 1} {\color{gray} 1} {\color{gray} 1} {\color{gray} 1} {\color{gray} 1} \overline{1} {\color{gray} 1} {\color{gray} 1} 2 {\color{gray} 1} {\color{gray} 1} {\color{gray} 1} {\color{gray} 1} {\color{gray} 1} {\color{gray} 1} {\color{gray} 1} {\color{gray} 1} {\color{gray} 1} {\color{gray} 1} {\color{gray} 1} {\color{gray} 1} {\color{gray} 1} {\color{gray} 1} {\color{gray} 1} {\color{gray} 1} {\color{gray} 1} {\color{gray} 1} \\
t = 8 & \cdots {\color{gray} 1} {\color{gray} 1} {\color{gray} 1} {\color{gray} 1} {\color{gray} 1} 3 0 \overline{3} \emptyset {\color{gray} 1} {\color{gray} 1} {\color{gray} 1} {\color{gray} 1} {\color{gray} 1} {\color{gray} 1} {\color{gray} 1} {\color{gray} 1} {\color{gray} 1} {\color{gray} 1} {\color{gray} 1} {\color{gray} 1} \overline{1} {\color{gray} 1} {\color{gray} 1} {\color{gray} 1} 2 {\color{gray} 1} {\color{gray} 1} {\color{gray} 1} {\color{gray} 1} {\color{gray} 1} {\color{gray} 1} {\color{gray} 1} {\color{gray} 1} {\color{gray} 1} {\color{gray} 1} {\color{gray} 1} {\color{gray} 1} {\color{gray} 1} {\color{gray} 1} {\color{gray} 1} {\color{gray} 1} {\color{gray} 1} {\color{gray} 1} {\color{gray} 1} \\
t = 9 & \cdots {\color{gray} 1} 3 0 \overline{3} \emptyset {\color{gray} 1} {\color{gray} 1} {\color{gray} 1} {\color{gray} 1} {\color{gray} 1} {\color{gray} 1} {\color{gray} 1} {\color{gray} 1} {\color{gray} 1} {\color{gray} 1} {\color{gray} 1} {\color{gray} 1} {\color{gray} 1} {\color{gray} 1} \overline{1} {\color{gray} 1} {\color{gray} 1} {\color{gray} 1} {\color{gray} 1} 2 {\color{gray} 1} {\color{gray} 1} {\color{gray} 1} {\color{gray} 1} {\color{gray} 1} {\color{gray} 1} {\color{gray} 1} {\color{gray} 1} {\color{gray} 1} {\color{gray} 1} {\color{gray} 1} {\color{gray} 1} {\color{gray} 1} {\color{gray} 1} {\color{gray} 1} {\color{gray} 1} {\color{gray} 1} {\color{gray} 1} {\color{gray} 1} {\color{gray} 1} \\
\end{array}}
\]
\end{ex}

\begin{prop}
\label{prop:soliton_adjoint}
Let $\g$ be of affine type except type $A_n^{(1)}$, $C_n^{(1)}$, $D_{n+1}^{(2)}$, $A_{2n}^{(2)}$ or $A_{2n}^{(2)\dagger}$.
Suppose $r \sim 0$.
For a soliton $p = b_{-L} \otimes \cdots \otimes b_0 \in \Sol^{(r)}(\ell)$, we have
\[
u < b_{-L} \leq \cdots \leq b_0 \leq v_{\varpi}.
\]
where $\ell = L + 1$ and
\begin{itemize}
\item in type $B_n^{(1)}$, there is at most one $x_{\clfw_1}$ or $x_{\clfw_1 - \alpha_1}$,
\item in type $G_2^{(1)}$, there is at most one $x_{\alpha_1+\alpha_2}$ or $x_{2\alpha_1+\alpha_2}$.
\end{itemize}
\end{prop}

\begin{proof}
We have $H(u_{\varpi} \otimes u) = 1$.

We note that any $r' \sim r$ is a special node in $\widehat{\g}_{0,r}$.
Therefore, we have $\mathcal{B}^{(r)}(s) \iso B(s\varpi)$ as $U_q(\g_{0,r})$-crystals.
Next, we note that if $H(b_0 \otimes u) = 1$ then either $b_0 \in B(\varpi)$ or $b_0 = \emptyset$.
The proof of this is similar to the proof given in Proposition~\ref{prop:soliton_H1}, and the lowest weight element is $v_{\varpi} \otimes u$.
We note that we cannot have $\emptyset \otimes u$ as $H(x \otimes \emptyset) = 1$ for all $x \in B(\clfw_r)$, which would give a state energy of at least $2$.

Next, consider $b_{-1} \otimes b_0$ with $b_1 \neq u$ and for a fixed $u_{\varpi} \leq b_0 \leq v_{\varpi}$.
We claim $H(b_{-1} \otimes b_0) = 0$ if and only if $b_{-1} \in B(\varpi)$ and $b_{-1} \leq b_0$.
For the exceptional types, this is a finite computation.
Thus, we assume $\g$ is of type $B_n^{(1)}$, $D_n^{(1)}$, $A_{2n-1}^{(2)}$, or $A_{2n}^{(2)\dagger}$, and we note that Proposition~\ref{prop:minuscule_description} holds for $B(s\varpi)$ in all of these cases (even if $r'$ is not necessarily minuscule).
If $b_{-1} \in B(\varpi)$ and $b_{-1} \leq b_0$, then $H(b_{-1} \otimes b_0) = 0$ by Proposition~\ref{prop:minuscule_description} and that $e_r^2(u_{\varpi} \otimes u_{\varpi}) = u \otimes u$.
To show the converse, we proceed by induction.
The base case is $u_{\varpi} \otimes b_0$, for which the claim immediately follows.
Assume the claim holds for some $b_{-1} \otimes b_0$, and consider $f_i(b_{-1} \otimes b_0) = (f_i b_{-1}) \otimes b_0 \neq 0$.
If $i \neq r$, then the claim follows from Proposition~\ref{prop:minuscule_description}.
Next, for $i = r$, we must have $f_r b_0 \neq 0$ since $b_{-1} \leq b_0$.
However, this contradicts the tensor product rule, and hence we must have $b_{-1} \in X$ and $b_{-1} \leq b_0$.
Similarly, we must have $b_i \leq b_{i+1}$ for all $-L \leq j < 0$.
\end{proof}

We give an example of the computation of Proposition~\ref{prop:soliton_adjoint}.
We note that $B(\clfw_3)$ in type $C_3$, which comes from $F_4^{(1)}$, is also characterized by Proposition~\ref{prop:minuscule_description}, but not $B(\clfw_4)$ in type $C_4$.

\begin{lstlisting}
sage: C3 = crystals.Tableaux(['C',3], shape=[1]*3)
sage: P = Poset(C3.digraph())
sage: C32 = crystals.Tableaux(['C',3], shape=[2]*3)
sage: all(P.le(C3(*(list(x)[:3])), C3(*(list(x)[3:])))
....:     for x in C32)
True
sage: C4 = crystals.Tableaux(['C',4], shape=[1]*4)
sage: P = Poset(C4.digraph())
sage: C42 = crystals.Tableaux(['C',4], shape=[2]*4)
sage: all(P.le(C4(*(list(x)[:4])), C4(*(list(x)[4:])))
....:     for x in C42)
False
\end{lstlisting}

\begin{ex}
\label{ex:SCA_B_2}
Consider an SCA starting with a state in $\Sol^{(2)}(2,4)$ in type $B_4^{(1)}$:
\[
{\begin{array}{c|c}
t = 0 & \cdots {\color{gray} \makebox[6pt]{$\begin{array}{c}1\\2\end{array}$}} {\color{gray} \makebox[6pt]{$\begin{array}{c}1\\2\end{array}$}} {\color{gray} \makebox[6pt]{$\begin{array}{c}1\\2\end{array}$}} {\color{gray} \makebox[6pt]{$\begin{array}{c}1\\2\end{array}$}} {\color{gray} \makebox[6pt]{$\begin{array}{c}1\\2\end{array}$}} {\color{gray} \makebox[6pt]{$\begin{array}{c}1\\2\end{array}$}} {\color{gray} \makebox[6pt]{$\begin{array}{c}1\\2\end{array}$}} {\color{gray} \makebox[6pt]{$\begin{array}{c}1\\2\end{array}$}} {\color{gray} \makebox[6pt]{$\begin{array}{c}1\\2\end{array}$}} {\color{gray} \makebox[6pt]{$\begin{array}{c}1\\2\end{array}$}} {\color{gray} \makebox[6pt]{$\begin{array}{c}1\\2\end{array}$}} {\color{gray} \makebox[6pt]{$\begin{array}{c}1\\2\end{array}$}} {\color{gray} \makebox[6pt]{$\begin{array}{c}1\\2\end{array}$}} {\color{gray} \makebox[6pt]{$\begin{array}{c}1\\2\end{array}$}} {\color{gray} \makebox[6pt]{$\begin{array}{c}1\\2\end{array}$}} {\color{gray} \makebox[6pt]{$\begin{array}{c}1\\2\end{array}$}} {\color{gray} \makebox[6pt]{$\begin{array}{c}1\\2\end{array}$}} {\color{gray} \makebox[6pt]{$\begin{array}{c}1\\2\end{array}$}} {\color{gray} \makebox[6pt]{$\begin{array}{c}1\\2\end{array}$}} {\color{gray} \makebox[6pt]{$\begin{array}{c}1\\2\end{array}$}} {\color{gray} \makebox[6pt]{$\begin{array}{c}1\\2\end{array}$}} \makebox[6pt]{$\begin{array}{c}1\\4\end{array}$} \makebox[6pt]{$\begin{array}{c}2\\\overline{3}\end{array}$} {\color{gray} \makebox[6pt]{$\begin{array}{c}1\\2\end{array}$}} {\color{gray} \makebox[6pt]{$\begin{array}{c}1\\2\end{array}$}} {\color{gray} \makebox[6pt]{$\begin{array}{c}1\\2\end{array}$}} \makebox[6pt]{$\begin{array}{c}1\\3\end{array}$} \makebox[6pt]{$\begin{array}{c}1\\3\end{array}$} \makebox[6pt]{$\begin{array}{c}2\\0\end{array}$} \makebox[6pt]{$\begin{array}{c}2\\\overline{3}\end{array}$} {\color{gray} \makebox[6pt]{$\begin{array}{c}1\\2\end{array}$}}  \\
t = 1 & \cdots {\color{gray} \makebox[6pt]{$\begin{array}{c}1\\2\end{array}$}} {\color{gray} \makebox[6pt]{$\begin{array}{c}1\\2\end{array}$}} {\color{gray} \makebox[6pt]{$\begin{array}{c}1\\2\end{array}$}} {\color{gray} \makebox[6pt]{$\begin{array}{c}1\\2\end{array}$}} {\color{gray} \makebox[6pt]{$\begin{array}{c}1\\2\end{array}$}} {\color{gray} \makebox[6pt]{$\begin{array}{c}1\\2\end{array}$}} {\color{gray} \makebox[6pt]{$\begin{array}{c}1\\2\end{array}$}} {\color{gray} \makebox[6pt]{$\begin{array}{c}1\\2\end{array}$}} {\color{gray} \makebox[6pt]{$\begin{array}{c}1\\2\end{array}$}} {\color{gray} \makebox[6pt]{$\begin{array}{c}1\\2\end{array}$}} {\color{gray} \makebox[6pt]{$\begin{array}{c}1\\2\end{array}$}} {\color{gray} \makebox[6pt]{$\begin{array}{c}1\\2\end{array}$}} {\color{gray} \makebox[6pt]{$\begin{array}{c}1\\2\end{array}$}} {\color{gray} \makebox[6pt]{$\begin{array}{c}1\\2\end{array}$}} {\color{gray} \makebox[6pt]{$\begin{array}{c}1\\2\end{array}$}} {\color{gray} \makebox[6pt]{$\begin{array}{c}1\\2\end{array}$}} {\color{gray} \makebox[6pt]{$\begin{array}{c}1\\2\end{array}$}} {\color{gray} \makebox[6pt]{$\begin{array}{c}1\\2\end{array}$}} {\color{gray} \makebox[6pt]{$\begin{array}{c}1\\2\end{array}$}} \makebox[6pt]{$\begin{array}{c}1\\4\end{array}$} \makebox[6pt]{$\begin{array}{c}2\\\overline{3}\end{array}$} {\color{gray} \makebox[6pt]{$\begin{array}{c}1\\2\end{array}$}} \makebox[6pt]{$\begin{array}{c}1\\3\end{array}$} \makebox[6pt]{$\begin{array}{c}1\\3\end{array}$} \makebox[6pt]{$\begin{array}{c}2\\0\end{array}$} \makebox[6pt]{$\begin{array}{c}2\\\overline{3}\end{array}$} {\color{gray} \makebox[6pt]{$\begin{array}{c}1\\2\end{array}$}} {\color{gray} \makebox[6pt]{$\begin{array}{c}1\\2\end{array}$}} {\color{gray} \makebox[6pt]{$\begin{array}{c}1\\2\end{array}$}} {\color{gray} \makebox[6pt]{$\begin{array}{c}1\\2\end{array}$}} {\color{gray} \makebox[6pt]{$\begin{array}{c}1\\2\end{array}$}} \\
t = 2 & \cdots {\color{gray} \makebox[6pt]{$\begin{array}{c}1\\2\end{array}$}} {\color{gray} \makebox[6pt]{$\begin{array}{c}1\\2\end{array}$}} {\color{gray} \makebox[6pt]{$\begin{array}{c}1\\2\end{array}$}} {\color{gray} \makebox[6pt]{$\begin{array}{c}1\\2\end{array}$}} {\color{gray} \makebox[6pt]{$\begin{array}{c}1\\2\end{array}$}} {\color{gray} \makebox[6pt]{$\begin{array}{c}1\\2\end{array}$}} {\color{gray} \makebox[6pt]{$\begin{array}{c}1\\2\end{array}$}} {\color{gray} \makebox[6pt]{$\begin{array}{c}1\\2\end{array}$}} {\color{gray} \makebox[6pt]{$\begin{array}{c}1\\2\end{array}$}} {\color{gray} \makebox[6pt]{$\begin{array}{c}1\\2\end{array}$}} {\color{gray} \makebox[6pt]{$\begin{array}{c}1\\2\end{array}$}} {\color{gray} \makebox[6pt]{$\begin{array}{c}1\\2\end{array}$}} {\color{gray} \makebox[6pt]{$\begin{array}{c}1\\2\end{array}$}} {\color{gray} \makebox[6pt]{$\begin{array}{c}1\\2\end{array}$}} {\color{gray} \makebox[6pt]{$\begin{array}{c}1\\2\end{array}$}} {\color{gray} \makebox[6pt]{$\begin{array}{c}1\\2\end{array}$}} {\color{gray} \makebox[6pt]{$\begin{array}{c}1\\2\end{array}$}} \makebox[6pt]{$\begin{array}{c}1\\4\end{array}$} \makebox[6pt]{$\begin{array}{c}3\\\overline{3}\end{array}$} \makebox[6pt]{$\begin{array}{c}1\\3\end{array}$} \makebox[6pt]{$\begin{array}{c}2\\0\end{array}$} \makebox[6pt]{$\begin{array}{c}2\\\overline{3}\end{array}$} {\color{gray} \makebox[6pt]{$\begin{array}{c}1\\2\end{array}$}} {\color{gray} \makebox[6pt]{$\begin{array}{c}1\\2\end{array}$}} {\color{gray} \makebox[6pt]{$\begin{array}{c}1\\2\end{array}$}} {\color{gray} \makebox[6pt]{$\begin{array}{c}1\\2\end{array}$}} {\color{gray} \makebox[6pt]{$\begin{array}{c}1\\2\end{array}$}} {\color{gray} \makebox[6pt]{$\begin{array}{c}1\\2\end{array}$}} {\color{gray} \makebox[6pt]{$\begin{array}{c}1\\2\end{array}$}} {\color{gray} \makebox[6pt]{$\begin{array}{c}1\\2\end{array}$}} {\color{gray} \makebox[6pt]{$\begin{array}{c}1\\2\end{array}$}} \\
t = 3 & \cdots {\color{gray} \makebox[6pt]{$\begin{array}{c}1\\2\end{array}$}} {\color{gray} \makebox[6pt]{$\begin{array}{c}1\\2\end{array}$}} {\color{gray} \makebox[6pt]{$\begin{array}{c}1\\2\end{array}$}} {\color{gray} \makebox[6pt]{$\begin{array}{c}1\\2\end{array}$}} {\color{gray} \makebox[6pt]{$\begin{array}{c}1\\2\end{array}$}} {\color{gray} \makebox[6pt]{$\begin{array}{c}1\\2\end{array}$}} {\color{gray} \makebox[6pt]{$\begin{array}{c}1\\2\end{array}$}} {\color{gray} \makebox[6pt]{$\begin{array}{c}1\\2\end{array}$}} {\color{gray} \makebox[6pt]{$\begin{array}{c}1\\2\end{array}$}} {\color{gray} \makebox[6pt]{$\begin{array}{c}1\\2\end{array}$}} {\color{gray} \makebox[6pt]{$\begin{array}{c}1\\2\end{array}$}} {\color{gray} \makebox[6pt]{$\begin{array}{c}1\\2\end{array}$}} {\color{gray} \makebox[6pt]{$\begin{array}{c}1\\2\end{array}$}} {\color{gray} \makebox[6pt]{$\begin{array}{c}1\\2\end{array}$}} \makebox[6pt]{$\begin{array}{c}1\\4\end{array}$} \makebox[6pt]{$\begin{array}{c}1\\0\end{array}$} \makebox[6pt]{$\begin{array}{c}3\\\overline{1}\end{array}$} \makebox[6pt]{$\begin{array}{c}2\\\overline{3}\end{array}$} {\color{gray} \makebox[6pt]{$\begin{array}{c}1\\2\end{array}$}} {\color{gray} \makebox[6pt]{$\begin{array}{c}1\\2\end{array}$}} {\color{gray} \makebox[6pt]{$\begin{array}{c}1\\2\end{array}$}} {\color{gray} \makebox[6pt]{$\begin{array}{c}1\\2\end{array}$}} {\color{gray} \makebox[6pt]{$\begin{array}{c}1\\2\end{array}$}} {\color{gray} \makebox[6pt]{$\begin{array}{c}1\\2\end{array}$}} {\color{gray} \makebox[6pt]{$\begin{array}{c}1\\2\end{array}$}} {\color{gray} \makebox[6pt]{$\begin{array}{c}1\\2\end{array}$}} {\color{gray} \makebox[6pt]{$\begin{array}{c}1\\2\end{array}$}} {\color{gray} \makebox[6pt]{$\begin{array}{c}1\\2\end{array}$}} {\color{gray} \makebox[6pt]{$\begin{array}{c}1\\2\end{array}$}} {\color{gray} \makebox[6pt]{$\begin{array}{c}1\\2\end{array}$}} {\color{gray} \makebox[6pt]{$\begin{array}{c}1\\2\end{array}$}} \\
t = 4 & \cdots {\color{gray} \makebox[6pt]{$\begin{array}{c}1\\2\end{array}$}} {\color{gray} \makebox[6pt]{$\begin{array}{c}1\\2\end{array}$}} {\color{gray} \makebox[6pt]{$\begin{array}{c}1\\2\end{array}$}} {\color{gray} \makebox[6pt]{$\begin{array}{c}1\\2\end{array}$}} {\color{gray} \makebox[6pt]{$\begin{array}{c}1\\2\end{array}$}} {\color{gray} \makebox[6pt]{$\begin{array}{c}1\\2\end{array}$}} {\color{gray} \makebox[6pt]{$\begin{array}{c}1\\2\end{array}$}} {\color{gray} \makebox[6pt]{$\begin{array}{c}1\\2\end{array}$}} {\color{gray} \makebox[6pt]{$\begin{array}{c}1\\2\end{array}$}} {\color{gray} \makebox[6pt]{$\begin{array}{c}1\\2\end{array}$}} \makebox[6pt]{$\begin{array}{c}1\\4\end{array}$} \makebox[6pt]{$\begin{array}{c}1\\0\end{array}$} \makebox[6pt]{$\begin{array}{c}2\\\overline{4}\end{array}$} \makebox[6pt]{$\begin{array}{c}3\\\overline{3}\end{array}$} \makebox[6pt]{$\begin{array}{c}2\\4\end{array}$} {\color{gray} \makebox[6pt]{$\begin{array}{c}1\\2\end{array}$}} {\color{gray} \makebox[6pt]{$\begin{array}{c}1\\2\end{array}$}} {\color{gray} \makebox[6pt]{$\begin{array}{c}1\\2\end{array}$}} {\color{gray} \makebox[6pt]{$\begin{array}{c}1\\2\end{array}$}} {\color{gray} \makebox[6pt]{$\begin{array}{c}1\\2\end{array}$}} {\color{gray} \makebox[6pt]{$\begin{array}{c}1\\2\end{array}$}} {\color{gray} \makebox[6pt]{$\begin{array}{c}1\\2\end{array}$}} {\color{gray} \makebox[6pt]{$\begin{array}{c}1\\2\end{array}$}} {\color{gray} \makebox[6pt]{$\begin{array}{c}1\\2\end{array}$}} {\color{gray} \makebox[6pt]{$\begin{array}{c}1\\2\end{array}$}} {\color{gray} \makebox[6pt]{$\begin{array}{c}1\\2\end{array}$}} {\color{gray} \makebox[6pt]{$\begin{array}{c}1\\2\end{array}$}} {\color{gray} \makebox[6pt]{$\begin{array}{c}1\\2\end{array}$}} {\color{gray} \makebox[6pt]{$\begin{array}{c}1\\2\end{array}$}} {\color{gray} \makebox[6pt]{$\begin{array}{c}1\\2\end{array}$}} {\color{gray} \makebox[6pt]{$\begin{array}{c}1\\2\end{array}$}} \\
t = 5 & \cdots {\color{gray} \makebox[6pt]{$\begin{array}{c}1\\2\end{array}$}} {\color{gray} \makebox[6pt]{$\begin{array}{c}1\\2\end{array}$}} {\color{gray} \makebox[6pt]{$\begin{array}{c}1\\2\end{array}$}} {\color{gray} \makebox[6pt]{$\begin{array}{c}1\\2\end{array}$}} {\color{gray} \makebox[6pt]{$\begin{array}{c}1\\2\end{array}$}} {\color{gray} \makebox[6pt]{$\begin{array}{c}1\\2\end{array}$}} \makebox[6pt]{$\begin{array}{c}1\\4\end{array}$} \makebox[6pt]{$\begin{array}{c}1\\0\end{array}$} \makebox[6pt]{$\begin{array}{c}2\\\overline{4}\end{array}$} \makebox[6pt]{$\begin{array}{c}2\\\overline{3}\end{array}$} {\color{gray} \makebox[6pt]{$\begin{array}{c}1\\2\end{array}$}} \makebox[6pt]{$\begin{array}{c}1\\3\end{array}$} \makebox[6pt]{$\begin{array}{c}2\\4\end{array}$} {\color{gray} \makebox[6pt]{$\begin{array}{c}1\\2\end{array}$}} {\color{gray} \makebox[6pt]{$\begin{array}{c}1\\2\end{array}$}} {\color{gray} \makebox[6pt]{$\begin{array}{c}1\\2\end{array}$}} {\color{gray} \makebox[6pt]{$\begin{array}{c}1\\2\end{array}$}} {\color{gray} \makebox[6pt]{$\begin{array}{c}1\\2\end{array}$}} {\color{gray} \makebox[6pt]{$\begin{array}{c}1\\2\end{array}$}} {\color{gray} \makebox[6pt]{$\begin{array}{c}1\\2\end{array}$}} {\color{gray} \makebox[6pt]{$\begin{array}{c}1\\2\end{array}$}} {\color{gray} \makebox[6pt]{$\begin{array}{c}1\\2\end{array}$}} {\color{gray} \makebox[6pt]{$\begin{array}{c}1\\2\end{array}$}} {\color{gray} \makebox[6pt]{$\begin{array}{c}1\\2\end{array}$}} {\color{gray} \makebox[6pt]{$\begin{array}{c}1\\2\end{array}$}} {\color{gray} \makebox[6pt]{$\begin{array}{c}1\\2\end{array}$}} {\color{gray} \makebox[6pt]{$\begin{array}{c}1\\2\end{array}$}} {\color{gray} \makebox[6pt]{$\begin{array}{c}1\\2\end{array}$}} {\color{gray} \makebox[6pt]{$\begin{array}{c}1\\2\end{array}$}} {\color{gray} \makebox[6pt]{$\begin{array}{c}1\\2\end{array}$}} {\color{gray} \makebox[6pt]{$\begin{array}{c}1\\2\end{array}$}} \\
t = 6 & \cdots {\color{gray} \makebox[6pt]{$\begin{array}{c}1\\2\end{array}$}} {\color{gray} \makebox[6pt]{$\begin{array}{c}1\\2\end{array}$}} \makebox[6pt]{$\begin{array}{c}1\\4\end{array}$} \makebox[6pt]{$\begin{array}{c}1\\0\end{array}$} \makebox[6pt]{$\begin{array}{c}2\\\overline{4}\end{array}$} \makebox[6pt]{$\begin{array}{c}2\\\overline{3}\end{array}$} {\color{gray} \makebox[6pt]{$\begin{array}{c}1\\2\end{array}$}} {\color{gray} \makebox[6pt]{$\begin{array}{c}1\\2\end{array}$}} {\color{gray} \makebox[6pt]{$\begin{array}{c}1\\2\end{array}$}} \makebox[6pt]{$\begin{array}{c}1\\3\end{array}$} \makebox[6pt]{$\begin{array}{c}2\\4\end{array}$} {\color{gray} \makebox[6pt]{$\begin{array}{c}1\\2\end{array}$}} {\color{gray} \makebox[6pt]{$\begin{array}{c}1\\2\end{array}$}} {\color{gray} \makebox[6pt]{$\begin{array}{c}1\\2\end{array}$}} {\color{gray} \makebox[6pt]{$\begin{array}{c}1\\2\end{array}$}} {\color{gray} \makebox[6pt]{$\begin{array}{c}1\\2\end{array}$}} {\color{gray} \makebox[6pt]{$\begin{array}{c}1\\2\end{array}$}} {\color{gray} \makebox[6pt]{$\begin{array}{c}1\\2\end{array}$}} {\color{gray} \makebox[6pt]{$\begin{array}{c}1\\2\end{array}$}} {\color{gray} \makebox[6pt]{$\begin{array}{c}1\\2\end{array}$}} {\color{gray} \makebox[6pt]{$\begin{array}{c}1\\2\end{array}$}} {\color{gray} \makebox[6pt]{$\begin{array}{c}1\\2\end{array}$}} {\color{gray} \makebox[6pt]{$\begin{array}{c}1\\2\end{array}$}} {\color{gray} \makebox[6pt]{$\begin{array}{c}1\\2\end{array}$}} {\color{gray} \makebox[6pt]{$\begin{array}{c}1\\2\end{array}$}} {\color{gray} \makebox[6pt]{$\begin{array}{c}1\\2\end{array}$}} {\color{gray} \makebox[6pt]{$\begin{array}{c}1\\2\end{array}$}} {\color{gray} \makebox[6pt]{$\begin{array}{c}1\\2\end{array}$}} {\color{gray} \makebox[6pt]{$\begin{array}{c}1\\2\end{array}$}} {\color{gray} \makebox[6pt]{$\begin{array}{c}1\\2\end{array}$}} {\color{gray} \makebox[6pt]{$\begin{array}{c}1\\2\end{array}$}} \\
\end{array}}
\]
Note that in this example, the phase shift is $0$.
\end{ex}

Proposition~\ref{prop:soliton_adjoint} yields the solitons as given in~\cite[Prop.~8]{MOW12}, and similar to the above, our proof is essentially the same as~\cite[Prop.~8]{MOW12}.

Our next result shows that Conjecture~\ref{conj:SCA}(\ref{conj:SCA_bijection}) holds for some special cases, which we show by direct computation.

\begin{prop}
\label{prop:soliton_crystal_bijection}
Let $r$ be a special node or $\g$ is not of type $A_n^{(1)}$ with $r \sim 0$.
There exists a $U_q(\g_{0,r})$-crystal isomorphism
\[
\Psi \colon \Sol^{(r)}(\ell) \to \mathcal{B}^{(r)}(\ell).
\]
\end{prop}

\begin{proof}
We first need to take care of the special case of type $C_2^{(1)}$, where we define $\Psi$ by
\[
u_{\varpi}^{\otimes m} \otimes v^{\otimes \overline{m}} \mapsto (f_1^{\overline{m}} u_{\lfloor \ell/2 \rfloor}) \otimes u_{\lceil \ell/2 \rceil},
\]
where $\ell = m + 2\overline{m}$, and extended as a $U_q(\g_{0,r})$-crystal morphism.
It is straightforward to see that this is a $U_q(\g_{0,r})$-crystal isomorphism as $I_{0,r} = \{1\}$.

We note that $v$ and $\emptyset$ both have weight $0$ as $U_q(\g_{0,r})$-crystals and contribute $2$ and $1$, respectively, to the length of the soliton.
These are also the only elements that appear in the solitons that are not in $B(\varpi)$.
Therefore, we define $\Psi$ by
\[
u_{\varpi}^{\otimes m} \otimes v^{\otimes \overline{m}} \otimes \emptyset^{\otimes m_*} \mapsto u_{m\varpi},
\]
where $m + 2\overline{m} + m_* = \ell$ (and $m_* \in \{0, 1\}$), and extend as a $U_q(\g_{0,r})$-crystal morphism.

Note that the defining condition of the solitons from Proposition~\ref{prop:soliton_H1}, Proposition~\ref{prop:soliton_H2}, or Proposition~\ref{prop:soliton_adjoint} and elements in $B(k\varpi)$ from Proposition~\ref{prop:minuscule_description} agree with the description of single row tableaux, which are given pairwise by $x \leq y$ for $x \otimes y$ with $y_r$ appearing at most once.
Therefore $\Psi$ is a bijection by considering the classical decompositions of $\mathcal{B}^{(r)}(\ell)$.
It is clear that the $U_q(\g_{0,r})$-weights agree and $\Psi(f_i p) = f_i \Psi(p)$ for all $i \in I_{0,r}$.
Hence, the map $\Psi$ is a $U_q(\g_{0,r})$-crystal isomorphism as desired.
\end{proof}

We can describe the map $\Psi$ given by Proposition~\ref{prop:soliton_crystal_bijection} explicitly using Kashiwara--Nakashima (KN) tableaux~\cite{FOS09, KN94} for $r = 2$, which is adjacent to $0$, in types $A_{2n-1}^{(2)}$, $B_n^{(1)}$, and $D_n^{(1)}$.\footnote{For type $B_3^{(1)}$, the map is slightly more complicated because the right factor should instead be $B_A^{1,\lfloor \ell/2 \rfloor} \otimes B_A^{1,\lceil \ell/2 \rceil}$. This map is similar to the $C_2^{(1)}$ for $r = 1$ case, and we leave the details to the reader.}
We consider type $D_n^{(1)}$~\cite{MW13}, but the other types are similar.
We have
\[
\vdomino{x_{-L}}{y_{-L}} \otimes \cdots \otimes \vdomino{x_0}{y_0} \mapsto
\begin{array}{|c|c|c|} \hline x_{-L} & \cdots & x_0 \\\hline \end{array} \otimes \mathbf{y}_{\downarrow} \in B_A^{1,\ell} \otimes B_D^{1,\ell},
\]
where $B_A^{1,\ell}$ is a KR crystal of type $A_1^{(1)}$ and $B_D^{1,\ell}$ is a KR crystal of type $D_{n-2}^{(1)}$ and $\mathbf{y}_{\downarrow}$ decreases all (barred) entries by $2$ of
\[
\mathbf{y} = \begin{array}{|c|c|c|} \hline y_{-L} & \cdots & y_0 \\\hline \end{array}\ .
\]

\begin{ex}
\label{ex:decoupling_r=2}
Let $\g$ be of type $B_4^{(1)}$, then we have
\[
\Psi\left( \young(1,3) \otimes \young(1,0) \otimes \young(1,\bfo) \otimes \young(2,\bfo) \otimes \young(2,\bth) \right)
= \young(11122) \otimes \young(10\btw\btw\bon)\ ,
\]
where the image is in $B^{1,5}_A \otimes B^{1,5}_B$ of type $A_1^{(1)} \times B_2^{(1)}$.
\end{ex}

A similar description of $\Psi$ exists for $r = 1$ in types $C_n^{(1)}$, $D_{n+1}^{(2)}$, $A_{2n}^{(2)}$, and $A_{2n}^{(2)\dagger}$ using KN tableaux.
In this case, it is a single row tableaux mapping to a single row tableaux, removing all $\bon$ and $\emptyset$ entries, and decreasing all (barred) entries by $1$.
In~\cite[Prop.~8]{MOW12}, the map $\Psi$ was described explicitly for $r = 2$ in type $G_2^{(1)}$.

\begin{prop}
\label{prop:conserved_quantities}
Fix $r$ such that $\mathcal{B}^{(r)}(\ell) \iso B(\ell \varpi)$ as $U_q(\g_{0,r})$-crystals.
Let $p \in \Sol^{(r)}(\ell)$.
Then we have $E_k(p) = t_r^{\vee} \min(k, \ell)$ and $T_k(p)$ moves the soliton $\min(k, \ell)$ steps to the left.
\end{prop}

\begin{proof}
Since $\mathcal{B}^{(r)}(\ell) \iso B(\ell \varpi)$, there exist a unique $I_{0,r}$-highest weight element.
Therefore, we can apply $e_i$, where $i \in I_{0,r}$, as many times as possible to obtain the $I_{0,r}$-highest weight element
\[
p' = u_{\varpi} \otimes \cdots \otimes u_{\varpi}.
\]
Note that $H$ does not change and $e_i$ commutes with the combinatorial $R$-matrix.
Therefore, we have $E_k(p') = E_k(p)$.

Next, we note that our time evolution and local energy correspond to those for the box-ball system under identifying $u_{\clfw_r}$ and $u_{\varpi}$ with an empty box and a ball, respectively.
Hence, the claim follows from~\cite[Lemma~4.1]{FOY00}.
\end{proof}

Now we consider the case when $\g$ is of type $B_n^{(1)}$ and $r = n$.
We recall that $B^{n,1} \iso B(\clfw_n)$ as $U_q(\g_0)$-crystals and $B(\clfw_n)$ is a minuscule representation.

\begin{prop}
\label{prop:soliton_Bn}
Let $r = n$ and $\g$ of type $B_n^{(1)}$.
Let $v_*$ denote the lowest weight vector in the $U_q(\g_{0,r})$-subcrystal $B_*$ generated by $u_* := f_n f_{n-1} f_n u$.
For a soliton $p = b_{-L} \otimes \cdots \otimes b_0 \in \Sol^{(r)}(\ell)$, we have
\begin{align*}
u < b_{-L} \leq \cdots \leq b_0 \leq v_*.
\end{align*}
Moreover, we have $\ell = L + m_* + 1$, where $m_*$ are the number of elements in $B_*$ in the soliton. 
\end{prop}

\begin{proof}
All classically highest weight elements in $B^{n,1} \otimes B^{n,1}$ are of the form
\[
u^{(k)} := k \overline{k+1} \otimes u,
\]
where the element on the left is written as a minuscule word.
A straightforward computation shows that $H(u^{(k)}) = \lceil k / 2 \rceil$.
Note that $u^{(1)} = u_{\varpi} \otimes u$ and $u^{(2)} = u_* \otimes u$.
Since $B(\clfw_n)$ is a minuscule representation, the subset $B(2\clfw_n) \subseteq B(\clfw_n) \otimes B(\clfw_n)$ is given by the elements $x \otimes y$ such that $x \leq y$.
Thus, the remainder of the proof is similar to the proof of Proposition~\ref{prop:soliton_H1}.
\end{proof}

\begin{prop}
\label{prop:soliton_crystal_bijection_B}
Let $\g$ be of type $B_n^{(1)}$ and $r = n$.
Define $u_* := f_n f_{n-1} f_n u$.
There exists a $U_q(\g_{0,r})$-crystal isomorphism
\[
\Psi \colon \Sol^{(r)}(\ell) \to \mathcal{B}^{(r)}(\ell), 
\]
where we remove the left factor if $\ell = 1$,
defined by
\[
u^{\otimes m} \otimes u_*^{\otimes m_*} \mapsto (f_{n-1}^{m_*} u_{\lfloor \ell / 2 \rfloor \varpi_{n-1}}) \otimes u_{\lceil \ell / 2 \rceil \varpi_{n-1}}.
\]
\end{prop}

\begin{proof}
Since each $u_*$ contributes 2 to the length of the soliton, we have $m_* \leq \ell / 2$.
It is straightforward to see that $\Psi$ is an $U_q(\g_{0,r})$-crystal morphism.
From~\cite{Stembridge01}, the highest weight elements of $B(\lfloor \ell/2 \rfloor \varpi_{n-1}) \otimes B(\lceil \ell/2 \rceil \varpi_{n-1})$ are given by
\[
(f_{n-1}^{m_*} u_{\lfloor \ell / 2 \rfloor \varpi_{n-1}}) \otimes u_{\lceil \ell / 2 \rceil \varpi_{n-1}}.
\]
Hence, the map $\Psi$ is a $U_q(\g_{0,r})$-crystal isomorphism.
\end{proof}

\begin{ex}
\label{ex:SCA_B_spin}
We begin with the spin representation of~\cite{KN94} (see also~\cite{BS17,HK02}) to represent elements of $B(\clfw_n)$, where the elements are given by a sequence $(s_1, \dotsc, s_n)$ with $s_i = \pm$.
Next, we consider this to be the binary representation of an integer written in reverse order with $+ \mapsto 0$ and $- \mapsto 1$.
For example, we have
\begin{align*} 0 = 000 \dotsm 0 & = u_{\clfw_n} = (+, \dotsc, +, +, +),
\\ 2^{n-1} = 100 \dotsm 0 & = f_n u_{\clfw_n} = u_{\varpi} = (+, \dotsc, +, +, -),
\\ 2^{n-1} + 2^{n-2} = 110 \dotsm 0 & = u_* = (+, \dotsc, +, -, -).
\end{align*}
Thus, consider an SCA starting with an initial state in $\Sol^{(3)}(1,3)$ in type $B_3^{(1)}$:
\[
\begin{array}{c|c}
t = 0 & \cdots {\color{gray} 0000000000000} 1 {\color{gray} 00} 46 {\color{gray} 0} \\
t = 1 & \cdots {\color{gray} 000000000000} 146 {\color{gray} 0000} \\
t = 2 & \cdots {\color{gray} 0000000000} 47 {\color{gray} 0000000} \\
t = 3 & \cdots {\color{gray} 0000000} 43 {\color{gray} 0} 4 {\color{gray} 00000000} \\
t = 4 & \cdots {\color{gray} 0000} 43 {\color{gray} 000} 4{\color{gray} 000000000} \\
t = 5 & \cdots {\color{gray} 0} 43 {\color{gray} 00000} 4{\color{gray} 0000000000}
\end{array}
\]
Note that the right soliton after scattering $u_{\varpi} \otimes u_*$ is not connected to $u_{\varpi} \otimes u_{\varpi}$ by $e_1$ and $e_2$ operators.
Furthermore, we have
\begin{align*}
\Psi(1) & = \young(2,3)\ ,
&
\Psi(4) & = \young(1,2)\ ,
\\
\Psi(46) & = \young(1,3) \otimes \young(11,22)\ ,
& 
\Psi(43) & = \young(2,3) \otimes \young(11,23)\ .
\end{align*}
Note that $46$ and $\Phi(46)$ are $I_{0,3}$-highest weight elements, but $66$ another $I_{0,r}$-highest weight element. This demonstrates the necessity of the two tensor factors as there is only a unique $I_{0,3}$ highest weight element for any KR crystal $B^{r,s}$ of type $A_2^{(1)}$.
\end{ex}

\begin{ex}
\label{ex:SCA_B_spin2}
Keeping the same conventions as Example~\ref{ex:SCA_B_spin}, we give an SCA with an initial state in $\Sol^{(3)}(1,3)$ in type $B_3^{(1)}$:
\[
\begin{array}{c|c}
t = 0 & \cdots {\color{gray} 0000000000000} 1 {\color{gray} 00} 421 {\color{gray} 0} \\
t = 1 & \cdots {\color{gray} 000000000000} 1421 {\color{gray} 0000} \\
t = 2 & \cdots {\color{gray} 0000000000} 431 {\color{gray} 0000000} \\
t = 3 & \cdots {\color{gray} 0000000} 43 {\color{gray} 00} 1 {\color{gray} 00000000} \\
t = 4 & \cdots {\color{gray} 0000} 43 {\color{gray} 0000} 1 {\color{gray} 000000000} \\
t = 5 & \cdots {\color{gray} 0} 43 {\color{gray} 000000} 1 {\color{gray} 0000000000}
\end{array}
\]
Also, we have $\Psi(421) = \young(1,2) \otimes \young(12,33)\ .$
\end{ex}

Next, we consider $r = 1$ for type $G_2^{(1)}$.
The local energy on highest weight elements is given by
\[
H(1 \otimes 1) = 0,
\qquad
H(2 \otimes 1) = 1,
\qquad
H(0 \otimes 1) = 1,
\qquad
H(\bon \otimes 1) = 2.
\]
We have $H(b \otimes b') = 0$ if and only $b \leq b'$ and we do not have $b = b' = 0$.

\begin{prop}
\label{prop:soliton_G2}
Let $r = 1$ and $\g$ of type $G_2^{(1)}$.
For a soliton $p = b_{-L} \otimes \cdots \otimes b_0 \in \Sol^{(r)}(\ell)$, we have
\begin{align*}
1 < b_{-L} \leq \cdots \leq b_0 \leq \btw,
\end{align*}
where $0$ appears at most once.
Moreover, we have $\ell = L + m_0 + 2(m_{\bth} + m_{\btw})$, where $m_b$ denotes the number of occurrences of $b$ in the soliton.
\end{prop}

\begin{proof}
This reduces to a computation of local energy on $B^{1,1} \otimes B^{1,1}$, which is a finite computation.
\end{proof}

\begin{prop}
\label{prop:soliton_crystal_bijection_G}
Let $\g$ be of type $G_2^{(1)}$ and $r = 1$.
Define
\[
\Psi \colon \Sol^{(r)}(\ell) \to \mathcal{B}^{(r)}(\ell) := B^{r', \lfloor \ell / 3 \rfloor} \otimes B^{r', \lfloor \ell / 3 \rfloor + \sigma} \otimes B^{r', \lfloor \ell / 3 \rfloor + \tau},
\]
where $\sigma = 1$ (resp.~$\tau = 1$) if the remainder of $\ell / 3$ is at least $1$ (resp.\ equals $2$) and $0$ otherwise (we remove factors of $B^{r',0}$), by
\[
\Psi(p) = f_1^d \bigl( (f_1^{m_{\bth}-M} u_{\lfloor \ell / 3 \rfloor \varpi_1}) \otimes (f_1^{m_0+2M} u_{(\lfloor \ell / 3 \rfloor + \sigma) \varpi_1}) \otimes u_{(\lfloor \ell / 3 \rfloor) \varpi_1 + \tau} \bigr),
\]
for
\[
p = 2^{\otimes m_2} \otimes 3^{\otimes m_3} \otimes 0^{\otimes m_0} \otimes \bth^{\otimes m_{\bth}} \otimes \btw^{\otimes m_{\btw}},
\]
where $M = \min\{m_3, m_{\bth}\}$ and $d = m_3 - M + m_{\btw}$.
Then $\Psi$ is a $U_q(\g_{0,r})$-crystal isomorphism.
\end{prop}

\begin{proof}
We first consider $I_{0,r}$-highest weight elements, where $m_{\bth} - m_3 \geq 0$ and $m_{\btw} = 0$.
So we have $M = m_3$ and $d = 0$.
It is clear that $\Psi$ is a weight preserving bijection and $\Psi(p)$ is $I_{0,r}$-highest weight.
For the general case, it is straightforward to see that $\Psi$ commutes with the crystal operators.
\end{proof}


\begin{ex}
Consider $p = 2330\bth\btw$, and so $\ell = 11$. Thus, we have $\sigma = \tau = 1$ and
\begin{align*}
\Psi(p) & = f_1^2\bigl( (f_1^1 u_{3\varpi_1}) \otimes (f_1^3 u_{4\varpi_1}) \otimes u_{4\varpi_1} \bigr)
\\ & = f_1^2 \bigl( \young(112) \otimes \young(1222) \otimes \young(1111) \bigr)
\\ & = \young(122) \otimes \young(1222) \otimes \young(1112)\,.
\end{align*}
\end{ex}


\section{SCA using rigged configurations}
\label{sec:SCA_RC}

In this section, we give a proof of Conjecture~\ref{conj:SCA} using rigged configurations under various assumptions.
Throughout this section, we assume Conjecture~\ref{conj:bijection} and Conjecture~\ref{conj:R_matrix_identity} hold.
We note that our results are known in general for type $A_n^{(1)}$~\cite{FOY00,HHIKTT01,KOSTY06,Sakamoto08,Sakamoto09,Yamada04}, some of which also use rigged configurations in their proofs, and generalize the other cases done in~\cite{binMohamad12,HKOTY02,HKT00,MOW12,MW13,TNS96,Yamada07}.
We start with our first main result, where the partition $\nu^{(r)}$ encodes the sizes of the solitons under $\Phi$.

\begin{thm}
\label{thm:RC_soliton_sizes}
Suppose $r$ is a special node or $r \sim 0$.
Consider a state $b$ with solitons of type $\g$ of sizes $\ell_1 > \ell_2 > \cdots > \ell_m$ that are sufficiently far apart (not necessarily in that order).
Then we have $\nu^{(r)} = (\ell_1, \ell_2, \dotsc, \ell_m)$, where $(\nu, J) = \Phi(b)$.
\end{thm}

\begin{proof}
We first consider the case $r$ is a minuscule node.
We note that when we add every box of a soliton, we add exactly one box to $\nu^{(r)}$ as there is only one $r$ arrow in the path from any entry of the soliton to $u_{\clfw_r}$ or for $\emptyset$.
Because the solitons are sufficiently far apart, we can make the difference between the riggings $J_i^{(r)}$ and the vacancy numbers $p_i^{(r)}$ as large as we want.
Therefore, we have precisely one singular row in $\nu^{(r)}$ for each soliton and the length corresponds to the number of elements of the currently soliton we have added.

Next, when $r \sim 0$, we have $x \otimes y$ in a soliton if and only if $x \leq y$ by Proposition~\ref{prop:soliton_H1}, Proposition~\ref{prop:soliton_H2}, and Proposition~\ref{prop:soliton_adjoint}.
Thus, when we add $b_0$, we add a row of length equal to the number of $r$-arrows in the path from $b_0 \to u$ to $\nu^{(r)}$ as there are no singular rows in $\nu^{(r)}$, where we add the first box under $\Phi^{-1}$.
Thus, the newly added row is the only singular row in $\nu^{(r)}$ as the solitons are far apart.
Therefore, for each subsequent $b_i$, we can select at most the same rows that were previously selected by $b_{i+1}$.
Hence, all boxes added to $\nu^{(r)}$ are in the same row and the claim follows.

For $r = 1$ in type $B_n^{(1)}$, the proof is similar to the minuscule case.
For $r = n$ in type $C_n^{(1)}$, we note that the only box in the column $(x_1, \dotsc, x_n)$ that would change $\nu^{(n)}$ is $n < x_n < \bon$ as $x_k < n$ for all $k < n$.
Moreover, the addition of $x_n$ only adds a single box to $\nu^{(n)}$, and hence the proof is similar to the minuscule case.
\end{proof}

Theorem~\ref{thm:RC_soliton_sizes} suggests the following as an equivalent definition of solitons.
This suggested equivalence also comes from the description of solitons from inverse scattering transform and the $R$-matrix invariance of rigged configurations (see Proposition~\ref{prop:top_commutes} below).

\begin{conj}
\label{conj:soliton_partition}
Let $p$ be a state and $(\nu, J) = \Phi^{-1}(p)$.
Then $p$ corresponds to a soliton of length $\ell$ if and only if $\nu^{(r)} = (\ell)$.
Moreover, Theorem~\ref{thm:RC_soliton_sizes} holds for all $r \in I_0$.
\end{conj}

We can now give a description of time evolution on rigged configurations.
Let $A_s$ denote the map on rigged configurations given by adding $\min(i, s)$ to all riggings $J_i^{(r)}$ for all $i \in \ZZ_{\geq 0}$.
The following was shown in~\cite[Prop.~2.6]{KOSTY06} (see also~\cite[Thm.~5.2]{KSY11}), and we include a proof for completeness.

\begin{prop}
\label{prop:top_commutes}
We have
\[
A_s \circ \Phi^{-1} = \Phi^{-1} \circ T_s.
\]
\end{prop}

\begin{proof}
Let $u_s = u(B^{r,s})$.
Consider a state $b$ and $(\nu, J) = \Phi^{-1}(b)$.
We claim $\Phi^{-1}(b \otimes u_s) = (\nu, J') = A_s(\nu, J)$, where $J'$ is formed by adding $\min(i, s)$ to all riggings in $J_i^{(r)}$ for all $i$.
This follows from the fact that $\Phi$ is only based upon the coriggings, and that after adding $u_s$, we still have the empty rigged configuration.
Thus, every subsequent step is the same as for $\Phi^{-1}$ applied to $b$ except $p_i^{(r)}$ has increased by $i$ for all $i \in \ZZ_{\geq 0}$, and the claim follows from the fact that we make every changed row singular.
Next, we have  $\Phi^{-1}(b \otimes u_s) = \Phi^{-1}\bigl(u_s \otimes T_{\infty}(b)\bigr)$ from Conjecture~\ref{conj:R_matrix_identity}.
Furthermore, from the definition of $\Phi^{-1}$, we have $\Phi^{-1}\bigl(u_s \otimes T_{\infty}(b)\bigr) = \Phi^{-1}\bigl(T_{\infty}(b)\bigr)$, and the claim follows from Theorem~\ref{thm:RC_soliton_sizes}.
\end{proof}

We use rigged configurations to give an alternative (and uniform) proof of Proposition~\ref{prop:conserved_quantities} and the conservation laws.
These were shown in~\cite[Thm.~3.2]{FOY00}, where the proof given is also uniform using the Yang--Baxter equation.

\begin{prop}
\label{prop:state_movement}
Suppose Conjecture~\ref{conj:soliton_partition} holds.
For $p \in \Sol^{(r)}(\ell)$, we have $E_s(p) = t_r^{\vee} \min(s, \ell)$ and $T_s(p)$ moves the soliton $\min(s, \ell)$ steps to the left.
\end{prop}

\begin{proof}
Let $(\nu, J) = \Phi(p)$, and note that $\nu^{(r)} = (\ell)$ from Conjecture~\ref{conj:soliton_partition}.
From Proposition~\ref{prop:top_commutes}, $\Phi^{-1}\bigl(T_s(p)\bigr)$ differs from $\Phi^{-1}(p)$ by adding $\min(s, \ell)$ to the rigging $J_{\ell}^{(r)}$.
Thus, we have nonvacuum factors $\min(s, \ell)$ steps earlier under $\Phi^{-1}$.
Moreover, the image of the soliton under $\Phi^{-1}$ must be the same as $\Phi^{-1}$ only depends on the coriggings.
Hence $T_s(p)$ moves the soliton $\min(s, \ell)$ steps to the left.

Next, we note that $\cc\bigl(A_s(\nu, J)\bigr) - \cc(\nu, J) = t_r^{\vee} \min(k, \ell)$.
We consider the state $p$ truncated to $K \gg 1$ factors, which we denote by $p_K$.
From the definition of energy, we have
\[
D\bigl(u_s \otimes T_s(p_K)\bigr) - D\bigl(T_s(p_K)\bigr) = E_s(p_K)
\]
since $R\bigl(u_s \otimes T_s(p_K)\bigr) = p_K \otimes u_s$ and $D(u_s) = 0$.
Note that $(\theta \circ \Phi^{-1}) \bigl(u_s \otimes T_s(p_K) \bigr) = (\nu, J)$ since $A_s$ preserves coriggings and the extra $u_s$ factor increases the vacancy numbers $p_i^{(r)}$ by $\min(i, s)$.
Hence, we have $E_s(p_K) = t_r^{\vee} \min(s, \ell)$ since $\cc = D \circ \Phi \circ \theta$.
\end{proof}

\begin{prop}[Conservation laws]
Suppose Conjecture~\ref{conj:soliton_partition} holds.
We have
\[
T_s \circ T_{s'} = T_{s'} \circ T_s,
\qquad\qquad
E_s \circ T_{s'} = E_s.
\]
\end{prop}

\begin{proof}
It is clear that $T_s \circ T_{s'} = T_{s'} \circ T_s$ from Proposition~\ref{prop:top_commutes} and the description of $A_s$.
Next, we have
\[
\cc\bigl((A_{s'} \circ A_s)(\nu, J)\bigr) - \cc\bigl(A_s(\nu, J)\bigr)
= \cc\bigl(A_s(\nu, J)\bigr) - \cc(\nu, J)
= E_s(p)
\]
by Proposition~\ref{prop:state_movement}.
Therefore, we have $E_s \circ T_{s'} = E_s$.
\end{proof}

We now prove the desired decoupling rules using rigged configurations.

\begin{prop}
\label{prop:RC_restriction}
Let $B = \bigotimes_{i=1}^N B^{r, s_i}$ of type $\g$.
Let $\mu = \{\widetilde{s}_1, \dotsc, \widetilde{s}_{\widetilde{N}}\}$ be a partition and
\[
\mathcal{B}_{\mu} = \bigotimes_{r' \sim r} \bigotimes_{i=1}^{\widetilde{N}} \begin{cases}
B^{r', \gamma_r \widetilde{s}_i} & \text{if $\modgamma_r / \gamma_{r'} \in \ZZ$}, \\
B^{r', \lfloor \widetilde{s}_i / 2 \rfloor} \otimes B^{r', \lceil \widetilde{s}_i / 2 \rceil} & \text{if $\modgamma_r / \gamma_{r'} = 1/2$}, \\
B^{r', \lfloor \widetilde{s}_i / 3 \rfloor} \otimes B^{r', \lfloor \widetilde{s}_i / 3 \rfloor + \sigma_i} \otimes B^{r', \lfloor \widetilde{s}_i / 3 \rfloor + \tau_i} & \text{if $\modgamma_r / \gamma_{r'} = 1/3$},
\end{cases}
\]
of type $\widehat{\g}_{r,0}$, where $\sigma_i = 1$ (resp.~$\tau_i = 1$) if the remainder of $\widetilde{s}_i / 3$ at least $1$ (resp.\ equals $2$) and $0$ otherwise.
Then, the map
\begin{equation}
\label{eq:decoupling_rule}
B^{(r)} \colon \RC(B) \to \bigoplus_{\mu} \RC(\mathcal{B}_{\mu})^{\oplus m_{\mu}},
\end{equation}
where $m_{\mu} = \lvert \{ (\nu, J) \in \hwRC(B) \mid \nu^{(r)} = \mu \} \rvert$,
given by deleting $\nu^{(r)}$ is a $U_q(\g_{0,r})$-crystal isomorphism.
\end{prop}

\begin{proof}
For $I_{0,r}$-highest weight rigged configurations, this follows from the definition of the vacancy numbers.
The remaining cases follow from the definition of the crystal operators on rigged configurations.
\end{proof}

We note that Proposition~\ref{prop:soliton_crystal_bijection} follows immediately from Theorem~\ref{thm:RC_soliton_sizes} and Proposition~\ref{prop:RC_restriction}.
Moreover, Equation~\eqref{eq:decoupling_rule} is the decoupling rule on rigged configurations.
Thus, we have a proof of Proposition~\ref{prop:soliton_crystal_bijection} using rigged configurations.

\begin{ex}
Consider soliton $p$ from Example~\ref{ex:decoupling_r=2}. The corresponding rigged configuration $(\nu, J)$ under $\Phi$ is
\[
\begin{tikzpicture}[scale=.35,baseline=-25]
\rpp{2}{-2}{-2}
 \begin{scope}[xshift=5cm]
 \rpp{5}{2}{2}
 \end{scope}
 \begin{scope}[xshift=13cm]
 \rpp{4,1}{1,-1}{1,-1}
 \end{scope}
 \begin{scope}[xshift=20cm]
 \rpp{7}{-5}{-5}
 \end{scope}
\end{tikzpicture}.
\]
Therefore, we have
\[
B^{(2)}(\nu, J) =
\begin{tikzpicture}[scale=.35,baseline=-23]
\rpp{2}{-2}{-2}
 \begin{scope}[xshift=6cm]
 \rpp{4,1}{1,-1}{1,-1}
 \end{scope}
 \begin{scope}[xshift=13cm]
 \rpp{7}{-5}{-5}
 \end{scope}
\end{tikzpicture}
\]
in $\RC(B^{1,5})$ in type $A^{(1)}_1 \times B_2^{(1)}$. Therefore, we have
\begin{align*}
\begin{tikzpicture}[scale=.35,baseline=-23]
\rpp{2}{-2}{-2}
\end{tikzpicture}
& \xrightarrow[\hspace{20pt}]{\Phi^{-1}_A} \young(11122)\ ,
\\
\begin{tikzpicture}[scale=.35,baseline=-23]
\rpp{4,1}{1,-1}{1,-1}
 \begin{scope}[xshift=7cm]
 \rpp{7}{-5}{-5}
 \end{scope}
\end{tikzpicture}
& \xrightarrow[\hspace{20pt}]{\Phi^{-1}_B} \young(10\btw\btw\bon)\ ,
\end{align*}
which agrees with $\Psi(p)$.
\end{ex}

\begin{lemma}
\label{lemma:description_irreducible_decoupled}
We have $\mathcal{B}^{(r)}(\ell) \iso B(\ell \varpi)$ as $U_q(\g_{0,r})$-crystals if and only if for all $r' \sim r$, we must have $r'$ a special node and $\modgamma_r / \gamma_{r'} \in \ZZ$.
\end{lemma}

\begin{proof}
This is a straightforward check.
\end{proof}

\begin{prop}
\label{prop:sides_commute}
Fix an $r$ such that $\mathcal{B}^{(r)}(\ell) \iso B(\ell\varpi)$ as $U_q(\g_{0,r})$-crystals.
Let $\Sol^{(r)}(\ell_1, \dotsc, \ell_k; K)$ denote the truncation of states in $\Sol^{(r)}(\ell_1, \dotsc, \ell_k)$ to $(B^{r,1})^{\otimes K}$ from $(B^{r,1})^{\otimes \infty}$.
For $K \gg 1$, the diagram
\[
\xymatrixcolsep{4em}
\xymatrixrowsep{3em}
\xymatrix{
\Sol^{(r)}(\ell_1, \dotsc, \ell_k; K) \ar[d]_{\Psi} \ar[r]^-{\Phi^{-1}} & \RC\bigl( (B^{1,1})^{\otimes K} \bigr) \ar[d]^{B^{(r)}}
\\ \mathcal{B}^{(r)}(\ell_1, \dotsc, \ell_k) \ar[r]_-{\Phi^{-1}} & \RC\bigl( \mathcal{B}^{(r)}(\ell_1, \dotsc, \ell_k) \bigr)
}
\]
commutes.
\end{prop}

\begin{proof}
For $k = 1$, this follows from Proposition~\ref{prop:RC_restriction}.

Next, assume it holds for $k-1$.
We note that $\Phi^{-1}$, $B^{(r)}$, and $\Psi$ are $U_q(\g_{0,r})$-crystal isomorphisms, so it is sufficient to consider the case when we have a $I_{0,r}$-highest weight element.
Additionally, note that there is a unique $I_{0,r}$-highest weight element $\widetilde{u}_{\ell_1} \in \mathcal{B}^{(r)}(\ell_1)$.
Consider a state $p \in \Sol^{(r)}(\ell_1, \dotsc, \ell_k; K)$.
Therefore, the rightmost soliton of $p$ is $u_{\varpi}^{\otimes \ell_k}$, which maps to $\widetilde{u}_{\ell_1}$ under $\Psi$.

Let $p = \overline{p} \otimes u_{\varpi}^{\otimes \ell_k}$.
Let
\begin{align*}
(\nu, J) & = \Phi^{-1}(p), & (\overline{\nu}, \overline{J}) & = \Phi^{-1}(\overline{p}),
\\ (\widetilde{\nu}, \widetilde{J}) & = \Phi^{-1}\bigl( \Psi(p) \bigr), & (\widetilde{\overline{\nu}}, \widetilde{\overline{J}}) & = \Phi^{-1}\bigl( \Psi(\overline{p}) \bigr).
\end{align*}
We obtain $(\widetilde{\nu}, \widetilde{J})$ from $(\widetilde{\overline{\nu}}, \widetilde{\overline{J}})$ by adding $\min(i, \ell_k)$ to each rigging of $J_i^{(r')}$ for all $r' \sim r$ as adding $u\bigl(\mathcal{B}^{(r)}(\ell_k)\bigr)$ does not change the rigged configuration.
Next we consider how $(\nu, J)$ differs from $(\overline{\nu}, \overline{J})$.
We note that adding $u_{\varpi}^{\otimes \ell_k}$ only adds a single row of length $\ell_k$ to $\nu^{(r)}$ similar to the proof of Proposition~\ref{prop:soliton_crystal_bijection}.
Since the solitons are far apart, we never change this row, so it does not affect the remaining steps of $\Phi^{-1}$ other than the final riggings.
Hence, the riggings $J_i^{(r')}$ are $\min(i, \ell_k)$ larger than the riggings $\overline{J}_i^{(r')}$ for all $r' \sim r$.
Hence, we have $B^{(r)} \circ \Phi^{-1} = \Phi^{-1} \circ \Psi$ by induction.
\end{proof}

\begin{conj}
\label{conj:sides_commute}
Proposition~\ref{prop:sides_commute} holds for all $r$. 
\end{conj}

We expect a similar proof to work in general, but describing the difference between $(\nu, J)$ and $(\overline{\nu}, \overline{J})$ is more complicated.

Next, we give our second main result, a proof of scattering using rigged configurations.

\begin{thm}
\label{thm:RC_scattering}
Suppose Conjecture~\ref{conj:sides_commute} holds.
Let $\Psi \colon \Sol^{(r)}(\ell_1, \ell_2) \to \mathcal{B}^{(r)}(\ell_1, \ell_2)$, where $\ell_1 > \ell_2$.
Then we have
\[
\Psi \circ T_{\infty}^k = R \circ \Psi
\]
for $k \gg 1$.
\end{thm}

\begin{proof}
Fix some $k \gg 1$.
Let $\Sol^{(r)}(\ell_1, \ell_2)_k$ denote the set of states $p$ such that $T_{\infty}^k(p) \in \Sol^{(r)}(\ell_2, \ell_1)$.
Let $\RC(\Sol)$ denote the image of the states $\Sol$ under $\Phi^{-1}$, which is well defined since adding the left factors of $u_1$ does not change the rigged configuration under $\Phi^{-1}$.
Consider the cube
\[
\xymatrixcolsep{2.1em}
\xymatrixrowsep{2.5em}
\xymatrix{
\Sol^{(r)}(\ell_1, \ell_2)_k \ar[rrr]^{T_{\infty}^k} \ar[ddd]_{\Psi} \ar[dr]^{\Phi^{-1}} & & & \Sol^{(r)}(\ell_2, \ell_1) \ar[ddd]^{\Psi} \ar[dl]_{\Phi^{-1}}
\\ & \RC\bigl( \Sol^{(r)}(\ell_1, \ell_2)_k \bigr) \ar[r]^{A_{\infty}^k} \ar[d]_{B^{(r)}} & \RC\bigl( \Sol^{(r)}(\ell_2, \ell_1) \bigr) \ar[d]^{B^{(r)}} &
\\ & \RC\bigl( \mathcal{B}^{(r)}(\ell_1,\ell_2) \bigr) \ar[r]_{\operatorname{id}} & \RC\bigl( \mathcal{B}^{(r)}(\ell_2,\ell_1) \bigr) &
\\ \mathcal{B}^{(r)}(\ell_1,\ell_2) \ar[rrr]_{R} \ar[ur]_{\Phi^{-1}} & & & \mathcal{B}^{(r)}(\ell_2,\ell_1) \ar[ul]^{\Phi^{-1}}
}
\]

We first show the back face commutes.
Considering the path of $B^{(r)} \circ A_{\infty}^k$, we first change the riggings associated with the partition $\nu^{(r)}$, and then $B$ deletes the partition $\nu^{(r)}$.
Therefore, we have the new rigged configurations $\RC\bigl( \mathcal{B}^{(r)}(\ell_2,\ell_1) \bigr)$ without  the partition $\nu^{(r)}$.
Next, if we begin with the path of $B^{(r)}$ and $\id$, we will delete the partition $\nu^{(r)}$ first by $B^{(r)}$ and then change nothing by $\id$, so we will also get the rigged configurations $\RC\bigl( \mathcal{B}^{(r)}(\ell_2,\ell_1) \bigr)$.
Hence, we have $B^{(r)} \circ A_{\infty}^k = \id \circ B^{(r)}$ as desired.

The top face commutes by Proposition~\ref{prop:top_commutes}.
The bottom face commutes by our assumption that Conjecture~\ref{conj:R_matrix_identity} holds.
The left and right faces commute by Conjecture~\ref{conj:sides_commute}.
Hence, by~\cite[Lemma~5.3]{KSS02} and that $\Phi$ is a bijection, the front face commutes as desired.
\end{proof}

Finally, we give our last main result: a description of the phase shift and a proof using rigged configurations.

\begin{thm}
\label{thm:RC_phase_shift}
Fix an $r$ such that $\mathcal{B}^{(r)}(\ell) \iso B(\ell\varpi)$ as $U_q(\g_{0,r})$-crystals.
Consider a two soliton state with solitons $s_1$ and $s_2$ of lengths $\ell_1 < \ell_2$.
The phase shift is
\[
\delta = 2\ell_1 + A_{rr'} H\bigl(\Psi(s_1) \otimes \Psi(s_2)\bigr),
\]
where $r' \sim r$.
\end{thm}

\begin{proof}
Note that for $\mathcal{B}^{(r)}(\ell) \iso B(\ell\varpi)$ as $U_q(\g_{0,r})$-crystals, for all $r', r'' \sim r$, we have $A_{rr'} = A_{rr''}$, so the statement is well defined.
Moreover, our assumptions satisfy the assumptions of Proposition~\ref{prop:sides_commute} (and hence, Conjecture~\ref{conj:sides_commute} holds).

Consider a state $b \in \Sol^{(r)}(\ell_1, \ell_2)$ with $\ell_1 < \ell_2$ and is $I_{0,r}$-highest weight, and let $(\nu, J) = \Phi^{-1}(b)$.
Let $s_1$ and $s_2$ denote the solitons of length $\ell_1$ and $\ell_2$ respectively.
Without loss of generality, assume there are no vacuums to the right of $s_2$.
By Proposition~\ref{prop:soliton_crystal_bijection}/Proposition~\ref{prop:RC_restriction}, it is sufficient to consider $p$ to be a $I_{0,r}$-highest weight states.

We modify the vacancy numbers by
\[
P_i^{(a)}(\nu) = P_i^{(a)} = - \sum_{(b,j) \in \HH_0} \frac{A_{ab}}{\gamma_b} \min(\modgamma_a i, \modgamma_b j) m_j^{(b)},
\]
and note that for $B = (B^{r,1})^{\otimes \kappa}$, we have $p_i^{(a)} = P_i^{(a)} + \delta_{ar} \kappa$.
Let $\zeta$ denote the starting position of the right soliton $s_2$ and $j$ denote the rigging of $J_{\ell_2}^{(r)}$.
By our assumptions, we have $\Psi(s_2) = u\bigl(\mathcal{B}^{(r)}(\ell_2)\bigr)$; in particular, the only nonempty partition of $\Phi(s_2)$ is $\nu^{(r)} = (\ell_2)$ with rigging $j = -\zeta = -\ell_2$.
Let
\[
\widetilde{\jmath} = j + k \ell_2,
\]
which is the rigging of $\widetilde{J}_{\ell_2}^{(r)}$ after $k$ time evolutions.
We choose $k \gg 1$ such that $\Phi(\nu, \widetilde{J}) \in \Sol^{(r)}(\ell_2, \ell_1)$: we have two solitons of length $\ell_2 > \ell_1$ that are not interacting.
Let $\xi$ be such that $\widetilde{\jmath} = \xi + P_{\ell_2}^{(r)}$, and so
\[
\xi = \widetilde{\jmath} - P_{\ell_2}^{(r)} = j + k \ell_2 - P_{\ell_2}^{(r)}
\]
is the position of the left soliton $\widetilde{s}_2$.
Note that
\[
P_{\ell_2}^{(r)} = -2(\ell_1 + \ell_2) + \sum_{r' \sim r} A_{rr'} \lvert \nu^{(r')} \rvert
\]
since $\ell_1 < \ell_2$ and by Lemma~\ref{lemma:description_irreducible_decoupled}.
Note that each box added to $\nu^{(r')}$ when performing $\Phi$ on $b \otimes u_{\ell_2 \varpi}$ in type $\widehat{\g}_{0,r}$ corresponds to an $r'$ arrow from $b$ to $u_{\ell_1 \varpi}$.
Moreover, we add one to the local energy for each of these $r'$ arrows by Theorem~\ref{thm:minuscule_local_energy}.
Hence, we have $H\bigl(\Psi(s_1) \otimes \Psi(s_2)\bigr) = \sum_{r' \sim r} \lvert \nu^{(r')} \rvert$.
Thus, the phase shift is
\begin{align*}
\delta & = \xi - (\ell_2 + k\ell_2) = j - P_{\ell_2}^{(r)} - \ell_2
\\ & = 2\ell_1 + \sum_{r' \sim r} A_{rr'} \lvert \nu^{(r')} \rvert = 2\ell_1 + A_{rr'} H\bigl(\Psi(s_1) \otimes \Psi(s_2)\bigr),
\end{align*}
for some $r' \sim r$ as desired.
\end{proof}

From the proof above, observe that the phase shift is determined by the change in the vacancy number $P_{\ell_2}^{(r)}$ by adding the soliton $s_1$ under $\Phi$.

\begin{ex}
Consider an SCA in type $E_6^{(1)}$ with an initial sate in $\Sol^{(1)}(1,2)$, where we separate tensor factors with `$|$':
\[
\scalebox{.93}{$
{\begin{array}{c|c@{}*{17}{@{{\color{gray}|}}c@{}}c}
t = 0 & \cdots \vace & \vace & \vace & \vace & \vace & \vace & \vace & \vace & \vace & \vace & \vace & \vace & \btw5 & \vace & \vace & \vace & 2\bfive6 & \btw4\bsix \\
t = 1 & \cdots \vace & \vace & \vace & \vace & \vace & \vace & \vace & \vace & \vace & \vace & \vace & \btw5 & \vace & \vace & 2\bfive6 & \btw4\bsix & \vace & \vace \\
t = 2 & \cdots \vace & \vace & \vace & \vace & \vace & \vace & \vace & \vace & \vace & \vace & \btw5 & \vace & 2\bfive6 & \btw4\bsix & \vace & \vace & \vace & \vace \\
t = 3 & \cdots \vace & \vace & \vace & \vace & \vace & \vace & \vace & \vace & \vace & \btw5 & 2\bfive6 & \btw4\bsix & \vace & \vace & \vace & \vace & \vace & \vace \\
t = 4 & \cdots \vace & \vace & \vace & \vace & \vace & \vace & \vace & \btw5 & \btw4\bfive6 & 2\bsix & \vace & \vace & \vace & \vace & \vace & \vace & \vace & \vace \\
t = 5 & \cdots \vace & \vace & \vace & \vace & \vace & \btw5 & \btw4\bfive6 & \vace & 2\bsix & \vace & \vace & \vace & \vace & \vace & \vace & \vace & \vace & \vace \\
t = 6 & \cdots \vace & \vace & \vace & \btw5 & \btw4\bfive6 & \vace & \vace & 2\bsix & \vace & \vace & \vace & \vace & \vace & \vace & \vace & \vace & \vace & \vace \\
t = 7 & \cdots \vace & \btw5 & \btw4\bfive6 & \vace & \vace & \vace & 2\bsix & \vace & \vace & \vace & \vace & \vace & \vace & \vace & \vace & \vace & \vace & \vace \\
\end{array}}
$}
\]
Here we have
\begin{align*}
\Psi(\btw5) & = \begin{array}{|c|} \hline \raisebox{-1pt}{$2\bfo$} \\\hline \end{array}\ ,
& \Psi(2\bfive6 \otimes \btw4\bsix) & = \begin{array}{|c|c|}\hline \raisebox{-1pt}{$1\btw4$} &  \raisebox{-1pt}{$\bon3\bfo$} \\\hline \end{array}\ ,
\\
\Psi(2\bsix) & = \begin{array}{|c|} \hline \raisebox{-1pt}{$\bon4$} \\\hline \end{array}\ ,
& \Psi(\btw5 \otimes \btw4\bfive6) & = \begin{array}{|c|c|}\hline \raisebox{-1pt}{$2\bfo$} &  \raisebox{-1pt}{$1\btw3\bfo$} \\\hline \end{array}\ ,
\end{align*}
see Figure~\ref{fig:crystal_graphs} for the crystal graphs.
We also have
\begin{align*}
R\bigl( \begin{array}{|c|} \hline \raisebox{-1pt}{$2\bfo$} \\\hline \end{array} \otimes \begin{array}{|c|c|}\hline \raisebox{-1pt}{$1\btw4$} &  \raisebox{-1pt}{$\bon3\bfo$} \\\hline \end{array} \bigr)
& = \begin{array}{|c|c|}\hline \raisebox{-1pt}{$2\bfo$} &  \raisebox{-1pt}{$1\btw3\bfo$} \\\hline \end{array} \otimes \begin{array}{|c|} \hline \raisebox{-1pt}{$\bon4$} \\\hline \end{array}\ ,
\\ H\bigl( \begin{array}{|c|} \hline \raisebox{-1pt}{$2\bfo$} \\\hline \end{array} \otimes \begin{array}{|c|c|}\hline \raisebox{-1pt}{$1\btw4$} &  \raisebox{-1pt}{$\bon3\bfo$} \\\hline \end{array} \bigr) & = 1,
\end{align*}
which agrees with the phase shift of $2 \cdot 1 - 1 = 1$.
Furthermore, the corresponding rigged configuration $(\nu, J_t)$ after $t$ time evolutions is
\[
\begin{tikzpicture}[scale=.35,baseline=-25]
\rpp{2,1}{2t,t+4}{}
 \begin{scope}[xshift=6cm]
 \rpp{2}{-1}{-1}
 \end{scope}
 \begin{scope}[xshift=11.7cm]
 \rpp{2,1}{0,0}{0,0}
 \end{scope}
 \begin{scope}[xshift=17cm]
 \rpp{2,1}{1,0}{1,0}
 \end{scope}
 \begin{scope}[xshift=22cm]
 \rpp{2}{-1}{0}
 \end{scope}
 \begin{scope}[xshift=28cm]
 \rpp{1}{-1}{-1}
 \end{scope}
\end{tikzpicture},
\]
where we write the vacancy number of the left and the rigging on the right of each partition (we omit the vacancy numbers $p_i^{(1)}$).
Note that $P_2^{(1)}(\nu) = -3$ and $P_1^{(1)}(\nu) = -2$.
In contrast, the rigged configuration $\Phi(2\bfive6 \otimes \btw4\bsix) = (\overline{\nu}, \overline{J})$ is
\[
\begin{tikzpicture}[scale=.35,baseline=-18]
\rpp{2}{0}{0}
 \begin{scope}[xshift=5cm]
 \rpp{1}{-1}{-1}
 \end{scope}
 \begin{scope}[xshift=9.3cm]
 \rpp{2}{0}{0}
 \end{scope}
 \begin{scope}[xshift=14.3cm]
 \rpp{2}{1}{1}
 \end{scope}
 \begin{scope}[xshift=19.3cm]
 \rpp{2}{-1}{-1}
 \end{scope}
 \begin{scope}[xshift=25.3cm]
 \rpp{1}{-1}{-1}
 \end{scope}
\end{tikzpicture},
\]
and we have $P_2^{(1)}(\overline{\nu}) = -2$.
Note that $P_2^{(1)}(\overline{\nu}) - P_2^{(1)}(\nu) = 1$, which agrees with the phase shift.
\end{ex}
%
%


\section{Summary}
\label{sec:summary}

In this section, we summarize the cases that are proven by our results.
We continue assuming that Conjecture~\ref{conj:bijection} and Conjecture~\ref{conj:R_matrix_identity} hold.

We note that Conjecture~\ref{conj:soliton_partition} is equivalent to Conjecture~\ref{conj:SCA}(\ref{conj:SCA_bijection}) by Proposition~\ref{prop:RC_restriction} and Proposition~\ref{prop:sides_commute}.
Conjecture~\ref{conj:soliton_partition}/Theorem~\ref{thm:RC_soliton_sizes} holds for all special nodes and $r \sim 0$ in all affine types.

Proposition~\ref{prop:sides_commute} and Theorem~\ref{thm:RC_phase_shift}  hold
\begin{itemize}
\item for all $I_0$ in every simply-laced type, $A_{2n-1}^{(2)}$, and $D_4^{(3)}$;
\item $r \neq n$ in type $B_n^{(1)}$;
\item $r=1,2,3$ in type $E_6^{(2)}$; and
\item $r = 2$ in type $G_2^{(1)}$.
\end{itemize}
Therefore, Conjecture~\ref{conj:SCA}(\ref{conj:SCA_scattering},~\ref{conj:SCA_phase_shift}) holds in these cases for those nodes which are adjoint or special in the above list.
If we additionally assume Conjecture~\ref{conj:soliton_partition}, then Conjecture~\ref{conj:SCA}(\ref{conj:SCA_scattering},~\ref{conj:SCA_phase_shift}) holds for all cases in the above list.

Next, we discuss some aspects of the phase shift.
As mentioned above, the phase shift is precisely how the vacancy number $P_{\ell_2}^{(r)}$ changes when adding the second soliton.
Furthermore, using the results of this paper, we can construct SCA where the phase shift is \emph{negative}: the larger soliton is shifted to the \emph{right} (equivalently, the smaller soliton is shifted to the \emph{left}) by 1 step after scattering.
This is a phenomenon that had only been previously observed in types $D_4^{(3)}$~\cite{Yamada07} and $G_2^{(1)}$~\cite{MOW12}.

\begin{ex}
\label{ex:SCA_E_6_neg_shift}
Consider a state $p \in \Sol^{(2)}(1, 3)$ in type $E_6^{(1)}$ such that
\[
\Psi(p) = \young(4,5,6) \otimes \young(111,222,333)\ .
\]
Note that $H\bigl(\Psi(p)\bigr) = 3$, and so the resulting phase shift is $\delta = 2 \cdot 2 + (-1) \cdot 3 = -1$.
Explicitly, we have the SCA
\[
{\begin{array}{c|c}
t = 0 & \cdots {\color{gray} u} {\color{gray} u} {\color{gray} u} {\color{gray} u} {\color{gray} u} {\color{gray} u} {\color{gray} u} {\color{gray} u} {\color{gray} u} {\color{gray} u} {\color{gray} u} {\color{gray} u} {\color{gray} u} y {\color{gray} u} {\color{gray} u} x x x {\color{gray} u} \\
t = 1 & \cdots {\color{gray} u} {\color{gray} u} {\color{gray} u} {\color{gray} u} {\color{gray} u} {\color{gray} u} {\color{gray} u} {\color{gray} u} {\color{gray} u} {\color{gray} u} {\color{gray} u} {\color{gray} u} y x x x {\color{gray} u} {\color{gray} u} {\color{gray} u} {\color{gray} u} \\
t = 2 & \cdots {\color{gray} u} {\color{gray} u} {\color{gray} u} {\color{gray} u} {\color{gray} u} {\color{gray} u} {\color{gray} u} {\color{gray} u} {\color{gray} u} {\color{gray} u} {\color{gray} u} z x {\color{gray} u} {\color{gray} u} {\color{gray} u} {\color{gray} u} {\color{gray} u} {\color{gray} u} {\color{gray} u} \\
t = 3 & \cdots {\color{gray} u} {\color{gray} u} {\color{gray} u} {\color{gray} u} {\color{gray} u} {\color{gray} u} {\color{gray} u} {\color{gray} u} x x \emptyset {\color{gray} u} {\color{gray} u} {\color{gray} u} {\color{gray} u} {\color{gray} u} {\color{gray} u} {\color{gray} u} {\color{gray} u} {\color{gray} u} \\
t = 4 & \cdots {\color{gray} u} {\color{gray} u} {\color{gray} u} {\color{gray} u} {\color{gray} u} x x y x {\color{gray} u} {\color{gray} u} {\color{gray} u} {\color{gray} u} {\color{gray} u} {\color{gray} u} {\color{gray} u} {\color{gray} u} {\color{gray} u} {\color{gray} u} {\color{gray} u} \\
t = 5 & \cdots {\color{gray} u} {\color{gray} u} x x y {\color{gray} u} {\color{gray} u} x {\color{gray} u} {\color{gray} u} {\color{gray} u} {\color{gray} u} {\color{gray} u} {\color{gray} u} {\color{gray} u} {\color{gray} u} {\color{gray} u} {\color{gray} u} {\color{gray} u} {\color{gray} u}
\end{array}}
\]
where
\begin{align*}
x & = f_2 u = u_{\varpi},
\\ y & = f_4 f_3 f_1 f_5 f_4 f_3 f_6 f_5 f_4 f_2 u,
\\ z & = f_2 f_2 f_4 f_3 f_1 f_5 f_4 f_3 f_6 f_5 f_4 f_2 u.
\end{align*}
\end{ex}

\begin{ex}
\label{ex:SCA_B_2_neg_shift}
Consider an SCA starting with a state in $\Sol^{(2)}(1,3)$ in type $D_4^{(1)}$, $B_4^{(1)}$, or $A_9^{(2)}$:
\[
{\begin{array}{c|c}
t = 0 & \cdots {\color{gray} \makebox[6pt]{$\begin{array}{c}1\\2\end{array}$}} {\color{gray} \makebox[6pt]{$\begin{array}{c}1\\2\end{array}$}} {\color{gray} \makebox[6pt]{$\begin{array}{c}1\\2\end{array}$}} {\color{gray} \makebox[6pt]{$\begin{array}{c}1\\2\end{array}$}} {\color{gray} \makebox[6pt]{$\begin{array}{c}1\\2\end{array}$}} {\color{gray} \makebox[6pt]{$\begin{array}{c}1\\2\end{array}$}} {\color{gray} \makebox[6pt]{$\begin{array}{c}1\\2\end{array}$}} {\color{gray} \makebox[6pt]{$\begin{array}{c}1\\2\end{array}$}} {\color{gray} \makebox[6pt]{$\begin{array}{c}1\\2\end{array}$}} {\color{gray} \makebox[6pt]{$\begin{array}{c}1\\2\end{array}$}} {\color{gray} \makebox[6pt]{$\begin{array}{c}1\\2\end{array}$}} {\color{gray} \makebox[6pt]{$\begin{array}{c}1\\2\end{array}$}} {\color{gray} \makebox[6pt]{$\begin{array}{c}1\\2\end{array}$}} {\color{gray} \makebox[6pt]{$\begin{array}{c}1\\2\end{array}$}} {\color{gray} \makebox[6pt]{$\begin{array}{c}1\\2\end{array}$}} {\color{gray} \makebox[6pt]{$\begin{array}{c}1\\2\end{array}$}} {\color{gray} \makebox[6pt]{$\begin{array}{c}1\\2\end{array}$}} {\color{gray} \makebox[6pt]{$\begin{array}{c}1\\2\end{array}$}} {\color{gray} \makebox[6pt]{$\begin{array}{c}1\\2\end{array}$}} {\color{gray} \makebox[6pt]{$\begin{array}{c}1\\2\end{array}$}} {\color{gray} \makebox[6pt]{$\begin{array}{c}1\\2\end{array}$}} \makebox[6pt]{$\begin{array}{c}2\\\overline{3}\end{array}$} {\color{gray} \makebox[6pt]{$\begin{array}{c}1\\2\end{array}$}} {\color{gray} \makebox[6pt]{$\begin{array}{c}1\\2\end{array}$}} {\color{gray} \makebox[6pt]{$\begin{array}{c}1\\2\end{array}$}} {\color{gray} \makebox[6pt]{$\begin{array}{c}1\\2\end{array}$}} \makebox[6pt]{$\begin{array}{c}1\\3\end{array}$} \makebox[6pt]{$\begin{array}{c}2\\4\end{array}$} \makebox[6pt]{$\begin{array}{c}2\\\overline{3}\end{array}$} {\color{gray} \makebox[6pt]{$\begin{array}{c}1\\2\end{array}$}} \\
t = 1 & \cdots {\color{gray} \makebox[6pt]{$\begin{array}{c}1\\2\end{array}$}} {\color{gray} \makebox[6pt]{$\begin{array}{c}1\\2\end{array}$}} {\color{gray} \makebox[6pt]{$\begin{array}{c}1\\2\end{array}$}} {\color{gray} \makebox[6pt]{$\begin{array}{c}1\\2\end{array}$}} {\color{gray} \makebox[6pt]{$\begin{array}{c}1\\2\end{array}$}} {\color{gray} \makebox[6pt]{$\begin{array}{c}1\\2\end{array}$}} {\color{gray} \makebox[6pt]{$\begin{array}{c}1\\2\end{array}$}} {\color{gray} \makebox[6pt]{$\begin{array}{c}1\\2\end{array}$}} {\color{gray} \makebox[6pt]{$\begin{array}{c}1\\2\end{array}$}} {\color{gray} \makebox[6pt]{$\begin{array}{c}1\\2\end{array}$}} {\color{gray} \makebox[6pt]{$\begin{array}{c}1\\2\end{array}$}} {\color{gray} \makebox[6pt]{$\begin{array}{c}1\\2\end{array}$}} {\color{gray} \makebox[6pt]{$\begin{array}{c}1\\2\end{array}$}} {\color{gray} \makebox[6pt]{$\begin{array}{c}1\\2\end{array}$}} {\color{gray} \makebox[6pt]{$\begin{array}{c}1\\2\end{array}$}} {\color{gray} \makebox[6pt]{$\begin{array}{c}1\\2\end{array}$}} {\color{gray} \makebox[6pt]{$\begin{array}{c}1\\2\end{array}$}} {\color{gray} \makebox[6pt]{$\begin{array}{c}1\\2\end{array}$}} {\color{gray} \makebox[6pt]{$\begin{array}{c}1\\2\end{array}$}} {\color{gray} \makebox[6pt]{$\begin{array}{c}1\\2\end{array}$}} \makebox[6pt]{$\begin{array}{c}2\\\overline{3}\end{array}$} {\color{gray} \makebox[6pt]{$\begin{array}{c}1\\2\end{array}$}} {\color{gray} \makebox[6pt]{$\begin{array}{c}1\\2\end{array}$}} \makebox[6pt]{$\begin{array}{c}1\\3\end{array}$} \makebox[6pt]{$\begin{array}{c}2\\4\end{array}$} \makebox[6pt]{$\begin{array}{c}2\\\overline{3}\end{array}$} {\color{gray} \makebox[6pt]{$\begin{array}{c}1\\2\end{array}$}} {\color{gray} \makebox[6pt]{$\begin{array}{c}1\\2\end{array}$}} {\color{gray} \makebox[6pt]{$\begin{array}{c}1\\2\end{array}$}} {\color{gray} \makebox[6pt]{$\begin{array}{c}1\\2\end{array}$}} \\
t = 2 & \cdots {\color{gray} \makebox[6pt]{$\begin{array}{c}1\\2\end{array}$}} {\color{gray} \makebox[6pt]{$\begin{array}{c}1\\2\end{array}$}} {\color{gray} \makebox[6pt]{$\begin{array}{c}1\\2\end{array}$}} {\color{gray} \makebox[6pt]{$\begin{array}{c}1\\2\end{array}$}} {\color{gray} \makebox[6pt]{$\begin{array}{c}1\\2\end{array}$}} {\color{gray} \makebox[6pt]{$\begin{array}{c}1\\2\end{array}$}} {\color{gray} \makebox[6pt]{$\begin{array}{c}1\\2\end{array}$}} {\color{gray} \makebox[6pt]{$\begin{array}{c}1\\2\end{array}$}} {\color{gray} \makebox[6pt]{$\begin{array}{c}1\\2\end{array}$}} {\color{gray} \makebox[6pt]{$\begin{array}{c}1\\2\end{array}$}} {\color{gray} \makebox[6pt]{$\begin{array}{c}1\\2\end{array}$}} {\color{gray} \makebox[6pt]{$\begin{array}{c}1\\2\end{array}$}} {\color{gray} \makebox[6pt]{$\begin{array}{c}1\\2\end{array}$}} {\color{gray} \makebox[6pt]{$\begin{array}{c}1\\2\end{array}$}} {\color{gray} \makebox[6pt]{$\begin{array}{c}1\\2\end{array}$}} {\color{gray} \makebox[6pt]{$\begin{array}{c}1\\2\end{array}$}} {\color{gray} \makebox[6pt]{$\begin{array}{c}1\\2\end{array}$}} {\color{gray} \makebox[6pt]{$\begin{array}{c}1\\2\end{array}$}} {\color{gray} \makebox[6pt]{$\begin{array}{c}1\\2\end{array}$}} \makebox[6pt]{$\begin{array}{c}2\\\overline{3}\end{array}$} \makebox[6pt]{$\begin{array}{c}1\\3\end{array}$} \makebox[6pt]{$\begin{array}{c}2\\4\end{array}$} \makebox[6pt]{$\begin{array}{c}2\\\overline{3}\end{array}$} {\color{gray} \makebox[6pt]{$\begin{array}{c}1\\2\end{array}$}} {\color{gray} \makebox[6pt]{$\begin{array}{c}1\\2\end{array}$}} {\color{gray} \makebox[6pt]{$\begin{array}{c}1\\2\end{array}$}} {\color{gray} \makebox[6pt]{$\begin{array}{c}1\\2\end{array}$}} {\color{gray} \makebox[6pt]{$\begin{array}{c}1\\2\end{array}$}} {\color{gray} \makebox[6pt]{$\begin{array}{c}1\\2\end{array}$}} {\color{gray} \makebox[6pt]{$\begin{array}{c}1\\2\end{array}$}} \\
t = 3 & \cdots {\color{gray} \makebox[6pt]{$\begin{array}{c}1\\2\end{array}$}} {\color{gray} \makebox[6pt]{$\begin{array}{c}1\\2\end{array}$}} {\color{gray} \makebox[6pt]{$\begin{array}{c}1\\2\end{array}$}} {\color{gray} \makebox[6pt]{$\begin{array}{c}1\\2\end{array}$}} {\color{gray} \makebox[6pt]{$\begin{array}{c}1\\2\end{array}$}} {\color{gray} \makebox[6pt]{$\begin{array}{c}1\\2\end{array}$}} {\color{gray} \makebox[6pt]{$\begin{array}{c}1\\2\end{array}$}} {\color{gray} \makebox[6pt]{$\begin{array}{c}1\\2\end{array}$}} {\color{gray} \makebox[6pt]{$\begin{array}{c}1\\2\end{array}$}} {\color{gray} \makebox[6pt]{$\begin{array}{c}1\\2\end{array}$}} {\color{gray} \makebox[6pt]{$\begin{array}{c}1\\2\end{array}$}} {\color{gray} \makebox[6pt]{$\begin{array}{c}1\\2\end{array}$}} {\color{gray} \makebox[6pt]{$\begin{array}{c}1\\2\end{array}$}} {\color{gray} \makebox[6pt]{$\begin{array}{c}1\\2\end{array}$}} {\color{gray} \makebox[6pt]{$\begin{array}{c}1\\2\end{array}$}} {\color{gray} \makebox[6pt]{$\begin{array}{c}1\\2\end{array}$}} {\color{gray} \makebox[6pt]{$\begin{array}{c}1\\2\end{array}$}} {\color{gray} \makebox[6pt]{$\begin{array}{c}1\\2\end{array}$}} \makebox[6pt]{$\begin{array}{c}4\\\overline{1}\end{array}$} \makebox[6pt]{$\begin{array}{c}2\\\overline{3}\end{array}$} {\color{gray} \makebox[6pt]{$\begin{array}{c}1\\2\end{array}$}} {\color{gray} \makebox[6pt]{$\begin{array}{c}1\\2\end{array}$}} {\color{gray} \makebox[6pt]{$\begin{array}{c}1\\2\end{array}$}} {\color{gray} \makebox[6pt]{$\begin{array}{c}1\\2\end{array}$}} {\color{gray} \makebox[6pt]{$\begin{array}{c}1\\2\end{array}$}} {\color{gray} \makebox[6pt]{$\begin{array}{c}1\\2\end{array}$}} {\color{gray} \makebox[6pt]{$\begin{array}{c}1\\2\end{array}$}} {\color{gray} \makebox[6pt]{$\begin{array}{c}1\\2\end{array}$}} {\color{gray} \makebox[6pt]{$\begin{array}{c}1\\2\end{array}$}} {\color{gray} \makebox[6pt]{$\begin{array}{c}1\\2\end{array}$}} \\
t = 4 & \cdots {\color{gray} \makebox[6pt]{$\begin{array}{c}1\\2\end{array}$}} {\color{gray} \makebox[6pt]{$\begin{array}{c}1\\2\end{array}$}} {\color{gray} \makebox[6pt]{$\begin{array}{c}1\\2\end{array}$}} {\color{gray} \makebox[6pt]{$\begin{array}{c}1\\2\end{array}$}} {\color{gray} \makebox[6pt]{$\begin{array}{c}1\\2\end{array}$}} {\color{gray} \makebox[6pt]{$\begin{array}{c}1\\2\end{array}$}} {\color{gray} \makebox[6pt]{$\begin{array}{c}1\\2\end{array}$}} {\color{gray} \makebox[6pt]{$\begin{array}{c}1\\2\end{array}$}} {\color{gray} \makebox[6pt]{$\begin{array}{c}1\\2\end{array}$}} {\color{gray} \makebox[6pt]{$\begin{array}{c}1\\2\end{array}$}} {\color{gray} \makebox[6pt]{$\begin{array}{c}1\\2\end{array}$}} {\color{gray} \makebox[6pt]{$\begin{array}{c}1\\2\end{array}$}} {\color{gray} \makebox[6pt]{$\begin{array}{c}1\\2\end{array}$}} {\color{gray} \makebox[6pt]{$\begin{array}{c}1\\2\end{array}$}} {\color{gray} \makebox[6pt]{$\begin{array}{c}1\\2\end{array}$}} \makebox[6pt]{$\begin{array}{c}2\\4\end{array}$} \makebox[6pt]{$\begin{array}{c}2\\\overline{3}\end{array}$} \makebox[6pt]{$\begin{array}{c}1\\\overline{1}\end{array}$} {\color{gray} \makebox[6pt]{$\begin{array}{c}1\\2\end{array}$}} {\color{gray} \makebox[6pt]{$\begin{array}{c}1\\2\end{array}$}} {\color{gray} \makebox[6pt]{$\begin{array}{c}1\\2\end{array}$}} {\color{gray} \makebox[6pt]{$\begin{array}{c}1\\2\end{array}$}} {\color{gray} \makebox[6pt]{$\begin{array}{c}1\\2\end{array}$}} {\color{gray} \makebox[6pt]{$\begin{array}{c}1\\2\end{array}$}} {\color{gray} \makebox[6pt]{$\begin{array}{c}1\\2\end{array}$}} {\color{gray} \makebox[6pt]{$\begin{array}{c}1\\2\end{array}$}} {\color{gray} \makebox[6pt]{$\begin{array}{c}1\\2\end{array}$}} {\color{gray} \makebox[6pt]{$\begin{array}{c}1\\2\end{array}$}} {\color{gray} \makebox[6pt]{$\begin{array}{c}1\\2\end{array}$}} {\color{gray} \makebox[6pt]{$\begin{array}{c}1\\2\end{array}$}} \\
t = 5 & \cdots {\color{gray} \makebox[6pt]{$\begin{array}{c}1\\2\end{array}$}} {\color{gray} \makebox[6pt]{$\begin{array}{c}1\\2\end{array}$}} {\color{gray} \makebox[6pt]{$\begin{array}{c}1\\2\end{array}$}} {\color{gray} \makebox[6pt]{$\begin{array}{c}1\\2\end{array}$}} {\color{gray} \makebox[6pt]{$\begin{array}{c}1\\2\end{array}$}} {\color{gray} \makebox[6pt]{$\begin{array}{c}1\\2\end{array}$}} {\color{gray} \makebox[6pt]{$\begin{array}{c}1\\2\end{array}$}} {\color{gray} \makebox[6pt]{$\begin{array}{c}1\\2\end{array}$}} {\color{gray} \makebox[6pt]{$\begin{array}{c}1\\2\end{array}$}} {\color{gray} \makebox[6pt]{$\begin{array}{c}1\\2\end{array}$}} {\color{gray} \makebox[6pt]{$\begin{array}{c}1\\2\end{array}$}} {\color{gray} \makebox[6pt]{$\begin{array}{c}1\\2\end{array}$}} \makebox[6pt]{$\begin{array}{c}2\\4\end{array}$} \makebox[6pt]{$\begin{array}{c}2\\\overline{3}\end{array}$} \makebox[6pt]{$\begin{array}{c}2\\\overline{3}\end{array}$} \makebox[6pt]{$\begin{array}{c}1\\3\end{array}$} {\color{gray} \makebox[6pt]{$\begin{array}{c}1\\2\end{array}$}} {\color{gray} \makebox[6pt]{$\begin{array}{c}1\\2\end{array}$}} {\color{gray} \makebox[6pt]{$\begin{array}{c}1\\2\end{array}$}} {\color{gray} \makebox[6pt]{$\begin{array}{c}1\\2\end{array}$}} {\color{gray} \makebox[6pt]{$\begin{array}{c}1\\2\end{array}$}} {\color{gray} \makebox[6pt]{$\begin{array}{c}1\\2\end{array}$}} {\color{gray} \makebox[6pt]{$\begin{array}{c}1\\2\end{array}$}} {\color{gray} \makebox[6pt]{$\begin{array}{c}1\\2\end{array}$}} {\color{gray} \makebox[6pt]{$\begin{array}{c}1\\2\end{array}$}} {\color{gray} \makebox[6pt]{$\begin{array}{c}1\\2\end{array}$}} {\color{gray} \makebox[6pt]{$\begin{array}{c}1\\2\end{array}$}} {\color{gray} \makebox[6pt]{$\begin{array}{c}1\\2\end{array}$}} {\color{gray} \makebox[6pt]{$\begin{array}{c}1\\2\end{array}$}} {\color{gray} \makebox[6pt]{$\begin{array}{c}1\\2\end{array}$}} \\
t = 6 & \cdots {\color{gray} \makebox[6pt]{$\begin{array}{c}1\\2\end{array}$}} {\color{gray} \makebox[6pt]{$\begin{array}{c}1\\2\end{array}$}} {\color{gray} \makebox[6pt]{$\begin{array}{c}1\\2\end{array}$}} {\color{gray} \makebox[6pt]{$\begin{array}{c}1\\2\end{array}$}} {\color{gray} \makebox[6pt]{$\begin{array}{c}1\\2\end{array}$}} {\color{gray} \makebox[6pt]{$\begin{array}{c}1\\2\end{array}$}} {\color{gray} \makebox[6pt]{$\begin{array}{c}1\\2\end{array}$}} {\color{gray} \makebox[6pt]{$\begin{array}{c}1\\2\end{array}$}} {\color{gray} \makebox[6pt]{$\begin{array}{c}1\\2\end{array}$}} \makebox[6pt]{$\begin{array}{c}2\\4\end{array}$} \makebox[6pt]{$\begin{array}{c}2\\\overline{3}\end{array}$} \makebox[6pt]{$\begin{array}{c}2\\\overline{3}\end{array}$} {\color{gray} \makebox[6pt]{$\begin{array}{c}1\\2\end{array}$}} {\color{gray} \makebox[6pt]{$\begin{array}{c}1\\2\end{array}$}} \makebox[6pt]{$\begin{array}{c}1\\3\end{array}$} {\color{gray} \makebox[6pt]{$\begin{array}{c}1\\2\end{array}$}} {\color{gray} \makebox[6pt]{$\begin{array}{c}1\\2\end{array}$}} {\color{gray} \makebox[6pt]{$\begin{array}{c}1\\2\end{array}$}} {\color{gray} \makebox[6pt]{$\begin{array}{c}1\\2\end{array}$}} {\color{gray} \makebox[6pt]{$\begin{array}{c}1\\2\end{array}$}} {\color{gray} \makebox[6pt]{$\begin{array}{c}1\\2\end{array}$}} {\color{gray} \makebox[6pt]{$\begin{array}{c}1\\2\end{array}$}} {\color{gray} \makebox[6pt]{$\begin{array}{c}1\\2\end{array}$}} {\color{gray} \makebox[6pt]{$\begin{array}{c}1\\2\end{array}$}} {\color{gray} \makebox[6pt]{$\begin{array}{c}1\\2\end{array}$}} {\color{gray} \makebox[6pt]{$\begin{array}{c}1\\2\end{array}$}} {\color{gray} \makebox[6pt]{$\begin{array}{c}1\\2\end{array}$}} {\color{gray} \makebox[6pt]{$\begin{array}{c}1\\2\end{array}$}} {\color{gray} \makebox[6pt]{$\begin{array}{c}1\\2\end{array}$}} {\color{gray} \makebox[6pt]{$\begin{array}{c}1\\2\end{array}$}} \\
t = 7 & \cdots {\color{gray} \makebox[6pt]{$\begin{array}{c}1\\2\end{array}$}} {\color{gray} \makebox[6pt]{$\begin{array}{c}1\\2\end{array}$}} {\color{gray} \makebox[6pt]{$\begin{array}{c}1\\2\end{array}$}} {\color{gray} \makebox[6pt]{$\begin{array}{c}1\\2\end{array}$}} {\color{gray} \makebox[6pt]{$\begin{array}{c}1\\2\end{array}$}} {\color{gray} \makebox[6pt]{$\begin{array}{c}1\\2\end{array}$}} \makebox[6pt]{$\begin{array}{c}2\\4\end{array}$} \makebox[6pt]{$\begin{array}{c}2\\\overline{3}\end{array}$} \makebox[6pt]{$\begin{array}{c}2\\\overline{3}\end{array}$} {\color{gray} \makebox[6pt]{$\begin{array}{c}1\\2\end{array}$}} {\color{gray} \makebox[6pt]{$\begin{array}{c}1\\2\end{array}$}} {\color{gray} \makebox[6pt]{$\begin{array}{c}1\\2\end{array}$}} {\color{gray} \makebox[6pt]{$\begin{array}{c}1\\2\end{array}$}} \makebox[6pt]{$\begin{array}{c}1\\3\end{array}$} {\color{gray} \makebox[6pt]{$\begin{array}{c}1\\2\end{array}$}} {\color{gray} \makebox[6pt]{$\begin{array}{c}1\\2\end{array}$}} {\color{gray} \makebox[6pt]{$\begin{array}{c}1\\2\end{array}$}} {\color{gray} \makebox[6pt]{$\begin{array}{c}1\\2\end{array}$}} {\color{gray} \makebox[6pt]{$\begin{array}{c}1\\2\end{array}$}} {\color{gray} \makebox[6pt]{$\begin{array}{c}1\\2\end{array}$}} {\color{gray} \makebox[6pt]{$\begin{array}{c}1\\2\end{array}$}} {\color{gray} \makebox[6pt]{$\begin{array}{c}1\\2\end{array}$}} {\color{gray} \makebox[6pt]{$\begin{array}{c}1\\2\end{array}$}} {\color{gray} \makebox[6pt]{$\begin{array}{c}1\\2\end{array}$}} {\color{gray} \makebox[6pt]{$\begin{array}{c}1\\2\end{array}$}} {\color{gray} \makebox[6pt]{$\begin{array}{c}1\\2\end{array}$}} {\color{gray} \makebox[6pt]{$\begin{array}{c}1\\2\end{array}$}} {\color{gray} \makebox[6pt]{$\begin{array}{c}1\\2\end{array}$}} {\color{gray} \makebox[6pt]{$\begin{array}{c}1\\2\end{array}$}} {\color{gray} \makebox[6pt]{$\begin{array}{c}1\\2\end{array}$}} \\
\end{array}}
\]
\end{ex}

Finally, we note that Conjecture~\ref{conj:SCA}(\ref{conj:SCA_phase_shift}) requires a little bit of care when defining what the positions of the solitons are.
We first consider an SCA in type $C_3^{(1)}$ with an initial state in $\Sol^{(1)}(1,2)$:
\[
{\begin{array}{c|c}
t = 0 & \cdots {\color{gray} 1} {\color{gray} 1} {\color{gray} 1} {\color{gray} 1} {\color{gray} 1} {\color{gray} 1} {\color{gray} 1} {\color{gray} 1} {\color{gray} 1} {\color{gray} 1} {\color{gray} 1} {\color{gray} 1} {\color{gray} 1} {\color{gray} 1} {\color{gray} 1} {\color{gray} 1} {\color{gray} 1} {\color{gray} 1} \underline{3} {\color{gray} 1} {\color{gray} 1} {\color{gray} 1} \underline{\underline{\overline{1}}} {\color{gray} 1} \\
t = 1 & \cdots {\color{gray} 1} {\color{gray} 1} {\color{gray} 1} {\color{gray} 1} {\color{gray} 1} {\color{gray} 1} {\color{gray} 1} {\color{gray} 1} {\color{gray} 1} {\color{gray} 1} {\color{gray} 1} {\color{gray} 1} {\color{gray} 1} {\color{gray} 1} {\color{gray} 1} {\color{gray} 1} {\color{gray} 1}\underline{3} {\color{gray} 1} {\color{gray} 1} \underline{\underline{\overline{1}}} {\color{gray} 1} {\color{gray} 1} {\color{gray} 1} \\
t = 2 & \cdots {\color{gray} 1} {\color{gray} 1} {\color{gray} 1} {\color{gray} 1} {\color{gray} 1} {\color{gray} 1} {\color{gray} 1} {\color{gray} 1} {\color{gray} 1} {\color{gray} 1} {\color{gray} 1} {\color{gray} 1} {\color{gray} 1} {\color{gray} 1} {\color{gray} 1} {\color{gray} 1} \underline{3} {\color{gray} 1} \underline{\underline{\overline{1}}} {\color{gray} 1} {\color{gray} 1} {\color{gray} 1} {\color{gray} 1} {\color{gray} 1} \\
t = 3 & \cdots {\color{gray} 1} {\color{gray} 1} {\color{gray} 1} {\color{gray} 1} {\color{gray} 1} {\color{gray} 1} {\color{gray} 1} {\color{gray} 1} {\color{gray} 1} {\color{gray} 1} {\color{gray} 1} {\color{gray} 1} {\color{gray} 1} {\color{gray} 1} 3 \underline{\overline{2}} \underline{\underline{2}} {\color{gray} 1} {\color{gray} 1} {\color{gray} 1} {\color{gray} 1} {\color{gray} 1} {\color{gray} 1} {\color{gray} 1} \\
t = 4 & \cdots {\color{gray} 1} {\color{gray} 1} {\color{gray} 1} {\color{gray} 1} {\color{gray} 1} {\color{gray} 1} {\color{gray} 1} {\color{gray} 1} {\color{gray} 1} {\color{gray} 1} {\color{gray} 1} {\color{gray} 1} 3 \overline{2} \underline{\underline{\underline{{\color{gray} 1}}}} 2 {\color{gray} 1} {\color{gray} 1} {\color{gray} 1} {\color{gray} 1} {\color{gray} 1} {\color{gray} 1} {\color{gray} 1} {\color{gray} 1} \\
t = 5 & \cdots {\color{gray} 1} {\color{gray} 1} {\color{gray} 1} {\color{gray} 1} {\color{gray} 1} {\color{gray} 1} {\color{gray} 1} {\color{gray} 1} {\color{gray} 1} {\color{gray} 1} 3 \overline{2} \underline{\underline{{\color{gray} 1}}} \underline{{\color{gray} 1}} 2 {\color{gray} 1} {\color{gray} 1} {\color{gray} 1} {\color{gray} 1} {\color{gray} 1} {\color{gray} 1} {\color{gray} 1} {\color{gray} 1} {\color{gray} 1} \\
t = 6 & \cdots {\color{gray} 1} {\color{gray} 1} {\color{gray} 1} {\color{gray} 1} {\color{gray} 1} {\color{gray} 1} {\color{gray} 1} {\color{gray} 1} 3 \overline{2} \underline{\underline{{\color{gray} 1}}} {\color{gray} 1} \underline{{\color{gray} 1}} 2 {\color{gray} 1} {\color{gray} 1} {\color{gray} 1} {\color{gray} 1} {\color{gray} 1} {\color{gray} 1} {\color{gray} 1} {\color{gray} 1} {\color{gray} 1} {\color{gray} 1} \\
t = 7 & \cdots {\color{gray} 1} {\color{gray} 1} {\color{gray} 1} {\color{gray} 1} {\color{gray} 1} {\color{gray} 1} 3 \overline{2} \underline{\underline{{\color{gray} 1}}} {\color{gray} 1} {\color{gray} 1} \underline{{\color{gray} 1}} 2 {\color{gray} 1} {\color{gray} 1} {\color{gray} 1} {\color{gray} 1} {\color{gray} 1} {\color{gray} 1} {\color{gray} 1} {\color{gray} 1} {\color{gray} 1} {\color{gray} 1} {\color{gray} 1} \\
\end{array}}
\]
In this case, note that
\[
H\bigl(\Psi(3) \otimes \Psi(\bon)\bigr) = H(2 \otimes \emptyset) = 1,
\]
but after scattering, the phase shift appears to be $2$ for the left soliton and $1$ for the right soliton.
However, we can fix this by considering the left soliton to consist of the right nonvacuum element, so the resulting phase shift is $2 \cdot 1 - 1 = 1$.

In this example, note that if the soliton $\bon$ was concentrated at one point, we should get to $t = 3$ before the solitons start to interact.
Instead, they begin to interact at $t = 3$, as if there is an additional $1$ linked together with the $\bon$; in other words, we should consider the soliton $\bon$ instead as $1\bon$.
With this modification, we see the phase shift would indeed be $1$ for the left soliton.

As another example, consider an initial state in $\Sol^{(1)}(3,4)$ in type $C_3^{(1)}$:
\[
{\begin{array}{c|c}
t = 0 & \cdots {\color{gray} 1} {\color{gray} 1} {\color{gray} 1} {\color{gray} 1} {\color{gray} 1} {\color{gray} 1} {\color{gray} 1} {\color{gray} 1} {\color{gray} 1} {\color{gray} 1} {\color{gray} 1} {\color{gray} 1} {\color{gray} 1} {\color{gray} 1} {\color{gray} 1} {\color{gray} 1} {\color{gray} 1} {\color{gray} 1} {\color{gray} 1} {\color{gray} 1} {\color{gray} 1} {\color{gray} 1} {\color{gray} 1} {\color{gray} 1} {\color{gray} 1} {\color{gray} 1} {\color{gray} 1} {\color{gray} 1} {\color{gray} 1} {\color{gray} 1} {\color{gray} 1} {\color{gray} 1} {\color{gray} 1} {\color{gray} 1} {\color{gray} 1} {\color{gray} 1} {\color{gray} 1} {\color{gray} 1} {\color{gray} 1} {\color{gray} 1} {\color{gray} 1} {\color{gray} 1} \underline{2 2 3} {\color{gray} 1} {\color{gray} 1} {\color{gray} 1} {\color{gray} 1} \underline{\underline{\overline{1} \overline{1}}} {\color{gray} 1} \\
t = 1 & \cdots {\color{gray} 1} {\color{gray} 1} {\color{gray} 1} {\color{gray} 1} {\color{gray} 1} {\color{gray} 1} {\color{gray} 1} {\color{gray} 1} {\color{gray} 1} {\color{gray} 1} {\color{gray} 1} {\color{gray} 1} {\color{gray} 1} {\color{gray} 1} {\color{gray} 1} {\color{gray} 1} {\color{gray} 1} {\color{gray} 1} {\color{gray} 1} {\color{gray} 1} {\color{gray} 1} {\color{gray} 1} {\color{gray} 1} {\color{gray} 1} {\color{gray} 1} {\color{gray} 1} {\color{gray} 1} {\color{gray} 1} {\color{gray} 1} {\color{gray} 1} {\color{gray} 1} {\color{gray} 1} {\color{gray} 1} {\color{gray} 1} {\color{gray} 1} {\color{gray} 1} {\color{gray} 1} {\color{gray} 1} {\color{gray} 1} \underline{2 2 3} {\color{gray} 1} {\color{gray} 1} {\color{gray} 1} \underline{\underline{\overline{1} \overline{1}}} {\color{gray} 1} {\color{gray} 1} {\color{gray} 1} {\color{gray} 1} {\color{gray} 1} \\
t = 2 & \cdots {\color{gray} 1} {\color{gray} 1} {\color{gray} 1} {\color{gray} 1} {\color{gray} 1} {\color{gray} 1} {\color{gray} 1} {\color{gray} 1} {\color{gray} 1} {\color{gray} 1} {\color{gray} 1} {\color{gray} 1} {\color{gray} 1} {\color{gray} 1} {\color{gray} 1} {\color{gray} 1} {\color{gray} 1} {\color{gray} 1} {\color{gray} 1} {\color{gray} 1} {\color{gray} 1} {\color{gray} 1} {\color{gray} 1} {\color{gray} 1} {\color{gray} 1} {\color{gray} 1} {\color{gray} 1} {\color{gray} 1} {\color{gray} 1} {\color{gray} 1} {\color{gray} 1} {\color{gray} 1} {\color{gray} 1} {\color{gray} 1} {\color{gray} 1} 2 \underline{2 3 \overline{2}} {\color{gray} 1} {\color{gray} 1} \underline{\underline{2 \overline{1}}} {\color{gray} 1} {\color{gray} 1} {\color{gray} 1} {\color{gray} 1} {\color{gray} 1} {\color{gray} 1} {\color{gray} 1} {\color{gray} 1} {\color{gray} 1} \\
t = 3 & \cdots {\color{gray} 1} {\color{gray} 1} {\color{gray} 1} {\color{gray} 1} {\color{gray} 1} {\color{gray} 1} {\color{gray} 1} {\color{gray} 1} {\color{gray} 1} {\color{gray} 1} {\color{gray} 1} {\color{gray} 1} {\color{gray} 1} {\color{gray} 1} {\color{gray} 1} {\color{gray} 1} {\color{gray} 1} {\color{gray} 1} {\color{gray} 1} {\color{gray} 1} {\color{gray} 1} {\color{gray} 1} {\color{gray} 1} {\color{gray} 1} {\color{gray} 1} {\color{gray} 1} {\color{gray} 1} {\color{gray} 1} {\color{gray} 1} {\color{gray} 1} {\color{gray} 1} 2 2 \underline{3 \overline{2} {\color{gray} 1}} {\color{gray} 1} \underline{\underline{{\color{gray} 1} 2}} \overline{1} {\color{gray} 1} {\color{gray} 1} {\color{gray} 1} {\color{gray} 1} {\color{gray} 1} {\color{gray} 1} {\color{gray} 1} {\color{gray} 1} {\color{gray} 1} {\color{gray} 1} {\color{gray} 1} {\color{gray} 1} \\
t = 4 & \cdots {\color{gray} 1} {\color{gray} 1} {\color{gray} 1} {\color{gray} 1} {\color{gray} 1} {\color{gray} 1} {\color{gray} 1} {\color{gray} 1} {\color{gray} 1} {\color{gray} 1} {\color{gray} 1} {\color{gray} 1} {\color{gray} 1} {\color{gray} 1} {\color{gray} 1} {\color{gray} 1} {\color{gray} 1} {\color{gray} 1} {\color{gray} 1} {\color{gray} 1} {\color{gray} 1} {\color{gray} 1} {\color{gray} 1} {\color{gray} 1} {\color{gray} 1} {\color{gray} 1} {\color{gray} 1} 2 2 3 \underline{\overline{2} {\color{gray} 1} {\color{gray} 1}} \underline{\underline{{\color{gray} 1} {\color{gray} 1}}} 2 \overline{1} {\color{gray} 1} {\color{gray} 1} {\color{gray} 1} {\color{gray} 1} {\color{gray} 1} {\color{gray} 1} {\color{gray} 1} {\color{gray} 1} {\color{gray} 1} {\color{gray} 1} {\color{gray} 1} {\color{gray} 1} {\color{gray} 1} {\color{gray} 1} {\color{gray} 1} \\
t = 5 & \cdots {\color{gray} 1} {\color{gray} 1} {\color{gray} 1} {\color{gray} 1} {\color{gray} 1} {\color{gray} 1} {\color{gray} 1} {\color{gray} 1} {\color{gray} 1} {\color{gray} 1} {\color{gray} 1} {\color{gray} 1} {\color{gray} 1} {\color{gray} 1} {\color{gray} 1} {\color{gray} 1} {\color{gray} 1} {\color{gray} 1} {\color{gray} 1} {\color{gray} 1} {\color{gray} 1} {\color{gray} 1} {\color{gray} 1} 2 2 3 \overline{2} \underline{{\color{gray} 1} {\color{gray} 1}} \underline{\underline{\underline{{\color{gray} 1}} {\color{gray} 1}}} {\color{gray} 1} 2 \overline{1} {\color{gray} 1} {\color{gray} 1} {\color{gray} 1} {\color{gray} 1} {\color{gray} 1} {\color{gray} 1} {\color{gray} 1} {\color{gray} 1} {\color{gray} 1} {\color{gray} 1} {\color{gray} 1} {\color{gray} 1} {\color{gray} 1} {\color{gray} 1} {\color{gray} 1} {\color{gray} 1} {\color{gray} 1} {\color{gray} 1} \\
t = 6 & \cdots {\color{gray} 1} {\color{gray} 1} {\color{gray} 1} {\color{gray} 1} {\color{gray} 1} {\color{gray} 1} {\color{gray} 1} {\color{gray} 1} {\color{gray} 1} {\color{gray} 1} {\color{gray} 1} {\color{gray} 1} {\color{gray} 1} {\color{gray} 1} {\color{gray} 1} {\color{gray} 1} {\color{gray} 1} {\color{gray} 1} {\color{gray} 1} 2 2 3 \overline{2} {\color{gray} 1} \underline{{\color{gray} 1}} \underline{\underline{\underline{{\color{gray} 1} {\color{gray} 1}}}} {\color{gray} 1} {\color{gray} 1} 2 \overline{1} {\color{gray} 1} {\color{gray} 1} {\color{gray} 1} {\color{gray} 1} {\color{gray} 1} {\color{gray} 1} {\color{gray} 1} {\color{gray} 1} {\color{gray} 1} {\color{gray} 1} {\color{gray} 1} {\color{gray} 1} {\color{gray} 1} {\color{gray} 1} {\color{gray} 1} {\color{gray} 1} {\color{gray} 1} {\color{gray} 1} {\color{gray} 1} {\color{gray} 1} {\color{gray} 1} \\
t = 7 & \cdots {\color{gray} 1} {\color{gray} 1} {\color{gray} 1} {\color{gray} 1} {\color{gray} 1} {\color{gray} 1} {\color{gray} 1} {\color{gray} 1} {\color{gray} 1} {\color{gray} 1} {\color{gray} 1} {\color{gray} 1} {\color{gray} 1} {\color{gray} 1} {\color{gray} 1} 2 2 3 \overline{2} {\color{gray} 1} {\color{gray} 1} \underline{\underline{\underline{{\color{gray} 1} {\color{gray} 1}}}} \underline{{\color{gray} 1}} {\color{gray} 1} {\color{gray} 1} 2 \overline{1} {\color{gray} 1} {\color{gray} 1} {\color{gray} 1} {\color{gray} 1} {\color{gray} 1} {\color{gray} 1} {\color{gray} 1} {\color{gray} 1} {\color{gray} 1} {\color{gray} 1} {\color{gray} 1} {\color{gray} 1} {\color{gray} 1} {\color{gray} 1} {\color{gray} 1} {\color{gray} 1} {\color{gray} 1} {\color{gray} 1} {\color{gray} 1} {\color{gray} 1} {\color{gray} 1} {\color{gray} 1} {\color{gray} 1} {\color{gray} 1} \\
t = 8 & \cdots {\color{gray} 1} {\color{gray} 1} {\color{gray} 1} {\color{gray} 1} {\color{gray} 1} {\color{gray} 1} {\color{gray} 1} {\color{gray} 1} {\color{gray} 1} {\color{gray} 1} {\color{gray} 1} 2 2 3 \overline{2} {\color{gray} 1} {\color{gray} 1} \underline{\underline{{\color{gray} 1} \underline{{\color{gray} 1}}}} \underline{{\color{gray} 1} {\color{gray} 1}} {\color{gray} 1} {\color{gray} 1} 2 \overline{1} {\color{gray} 1} {\color{gray} 1} {\color{gray} 1} {\color{gray} 1} {\color{gray} 1} {\color{gray} 1} {\color{gray} 1} {\color{gray} 1} {\color{gray} 1} {\color{gray} 1} {\color{gray} 1} {\color{gray} 1} {\color{gray} 1} {\color{gray} 1} {\color{gray} 1} {\color{gray} 1} {\color{gray} 1} {\color{gray} 1} {\color{gray} 1} {\color{gray} 1} {\color{gray} 1} {\color{gray} 1} {\color{gray} 1} {\color{gray} 1} {\color{gray} 1} {\color{gray} 1} {\color{gray} 1} \\
t = 9 & \cdots {\color{gray} 1} {\color{gray} 1} {\color{gray} 1} {\color{gray} 1} {\color{gray} 1} {\color{gray} 1} {\color{gray} 1} 2 2 3 \overline{2} {\color{gray} 1} {\color{gray} 1} \underline{\underline{{\color{gray} 1} {\color{gray} 1}}} \underline{{\color{gray} 1} {\color{gray} 1} {\color{gray} 1}} {\color{gray} 1} {\color{gray} 1} 2 \overline{1} {\color{gray} 1} {\color{gray} 1} {\color{gray} 1} {\color{gray} 1} {\color{gray} 1} {\color{gray} 1} {\color{gray} 1} {\color{gray} 1} {\color{gray} 1} {\color{gray} 1} {\color{gray} 1} {\color{gray} 1} {\color{gray} 1} {\color{gray} 1} {\color{gray} 1} {\color{gray} 1} {\color{gray} 1} {\color{gray} 1} {\color{gray} 1} {\color{gray} 1} {\color{gray} 1} {\color{gray} 1} {\color{gray} 1} {\color{gray} 1} {\color{gray} 1} {\color{gray} 1} {\color{gray} 1} {\color{gray} 1} {\color{gray} 1} {\color{gray} 1} \\
t = 10 & \cdots {\color{gray} 1} {\color{gray} 1} {\color{gray} 1} 2 2 3 \overline{2} {\color{gray} 1} {\color{gray} 1} \underline{\underline{{\color{gray} 1} {\color{gray} 1}}} {\color{gray} 1} \underline{{\color{gray} 1} {\color{gray} 1} {\color{gray} 1}} {\color{gray} 1} {\color{gray} 1} 2 \overline{1} {\color{gray} 1} {\color{gray} 1} {\color{gray} 1} {\color{gray} 1} {\color{gray} 1} {\color{gray} 1} {\color{gray} 1} {\color{gray} 1} {\color{gray} 1} {\color{gray} 1} {\color{gray} 1} {\color{gray} 1} {\color{gray} 1} {\color{gray} 1} {\color{gray} 1} {\color{gray} 1} {\color{gray} 1} {\color{gray} 1} {\color{gray} 1} {\color{gray} 1} {\color{gray} 1} {\color{gray} 1} {\color{gray} 1} {\color{gray} 1} {\color{gray} 1} {\color{gray} 1} {\color{gray} 1} {\color{gray} 1} {\color{gray} 1} {\color{gray} 1} {\color{gray} 1} {\color{gray} 1} {\color{gray} 1} \\
\end{array}}
\]
Here, we have
\[
H\bigl(\Psi(223) \otimes \Psi(\bon\bon)\bigr) = H(112 \otimes \emptyset) = 2.
\]
If, after scattering, we consider the left soliton to consist of the two right nonvacuum elements and add a nonvacuum to the right soliton, then the phase shift is precisely $2 \cdot 3 - 2 = 4$.

As in the previous example, we see them interacting at $t = 2$, when the two solitons are still separated by multiple vacuum elements.
This can be seen by the fact that the carrier has not returned the maximal element when it reaches the next soliton (see also Example~\ref{ex:carrier_type_C}).
This is indicating that every $\bon$ should be linked with a $1$ as part of the soliton.
With this interpretation, we obtain the description of the solitons in type $C_n^{(1)}$ given in~\cite[Sec.~2.3]{HKOTY02} (see also~\cite[Sec.~3.2]{HKT00}).
This would also yield the desired phase shift of $2$.
Contrast this with the (conjectural) definition of solitons from Definition~\ref{def:soliton} and the rigged configurations under $\Phi$.
Indeed, under $\Phi$ the $\bon$ should be a soliton with doubled weighting as removing the $\bon$ removes two boxes from $\nu^{(1)}$ (assuming $\nu^{(1)}$ is a single row as we consider the single soliton case here).

We give one more example with an initial state in $\Sol^{(1)}(2, 4)$ in type $C_3^{(1)}$:
\[
{\begin{array}{c|c}
t = 0 & \cdots {\color{gray} 1} {\color{gray} 1} {\color{gray} 1} {\color{gray} 1} {\color{gray} 1} {\color{gray} 1} {\color{gray} 1} {\color{gray} 1} {\color{gray} 1} {\color{gray} 1} {\color{gray} 1} {\color{gray} 1} {\color{gray} 1} {\color{gray} 1} {\color{gray} 1} {\color{gray} 1} {\color{gray} 1} {\color{gray} 1} {\color{gray} 1} {\color{gray} 1} {\color{gray} 1} {\color{gray} 1} {\color{gray} 1} {\color{gray} 1} {\color{gray} 1} {\color{gray} 1} {\color{gray} 1} \underline{2 3} {\color{gray} 1} {\color{gray} 1} {\color{gray} 1} {\color{gray} 1} \underline{\underline{\overline{1} \overline{1}}} {\color{gray} 1} \\
t = 1 & \cdots {\color{gray} 1} {\color{gray} 1} {\color{gray} 1} {\color{gray} 1} {\color{gray} 1} {\color{gray} 1} {\color{gray} 1} {\color{gray} 1} {\color{gray} 1} {\color{gray} 1} {\color{gray} 1} {\color{gray} 1} {\color{gray} 1} {\color{gray} 1} {\color{gray} 1} {\color{gray} 1} {\color{gray} 1} {\color{gray} 1} {\color{gray} 1} {\color{gray} 1} {\color{gray} 1} {\color{gray} 1} {\color{gray} 1} {\color{gray} 1} {\color{gray} 1} \underline{2 3} {\color{gray} 1} {\color{gray} 1} \underline{\underline{\overline{1} \overline{1}}} {\color{gray} 1} {\color{gray} 1} {\color{gray} 1} {\color{gray} 1} {\color{gray} 1} \\
t = 2 & \cdots {\color{gray} 1} {\color{gray} 1} {\color{gray} 1} {\color{gray} 1} {\color{gray} 1} {\color{gray} 1} {\color{gray} 1} {\color{gray} 1} {\color{gray} 1} {\color{gray} 1} {\color{gray} 1} {\color{gray} 1} {\color{gray} 1} {\color{gray} 1} {\color{gray} 1} {\color{gray} 1} {\color{gray} 1} {\color{gray} 1} {\color{gray} 1} {\color{gray} 1} {\color{gray} 1} 2 3 \underline{\overline{2} \overline{2}} \underline{\underline{2 2}} {\color{gray} 1} {\color{gray} 1} {\color{gray} 1} {\color{gray} 1} {\color{gray} 1} {\color{gray} 1} {\color{gray} 1} {\color{gray} 1} {\color{gray} 1} \\
t = 3 & \cdots {\color{gray} 1} {\color{gray} 1} {\color{gray} 1} {\color{gray} 1} {\color{gray} 1} {\color{gray} 1} {\color{gray} 1} {\color{gray} 1} {\color{gray} 1} {\color{gray} 1} {\color{gray} 1} {\color{gray} 1} {\color{gray} 1} {\color{gray} 1} {\color{gray} 1} {\color{gray} 1} {\color{gray} 1} 2 3 \overline{2} \overline{2} \underline{\underline{\underline{{\color{gray} 1} {\color{gray} 1}}}} 2 2 {\color{gray} 1} {\color{gray} 1} {\color{gray} 1} {\color{gray} 1} {\color{gray} 1} {\color{gray} 1} {\color{gray} 1} {\color{gray} 1} {\color{gray} 1} {\color{gray} 1} {\color{gray} 1} \\
t = 4 & \cdots {\color{gray} 1} {\color{gray} 1} {\color{gray} 1} {\color{gray} 1} {\color{gray} 1} {\color{gray} 1} {\color{gray} 1} {\color{gray} 1} {\color{gray} 1} {\color{gray} 1} {\color{gray} 1} {\color{gray} 1} {\color{gray} 1} 2 3 \overline{2} \overline{2} \underline{\underline{{\color{gray} 1} {\color{gray} 1}}} \underline{{\color{gray} 1} {\color{gray} 1}} 2 2 {\color{gray} 1} {\color{gray} 1} {\color{gray} 1} {\color{gray} 1} {\color{gray} 1} {\color{gray} 1} {\color{gray} 1} {\color{gray} 1} {\color{gray} 1} {\color{gray} 1} {\color{gray} 1} {\color{gray} 1} {\color{gray} 1} \\
t = 5 & \cdots {\color{gray} 1} {\color{gray} 1} {\color{gray} 1} {\color{gray} 1} {\color{gray} 1} {\color{gray} 1} {\color{gray} 1} {\color{gray} 1} {\color{gray} 1} 2 3 \overline{2} \overline{2} \underline{\underline{{\color{gray} 1} {\color{gray} 1}}} {\color{gray} 1} {\color{gray} 1} \underline{{\color{gray} 1} {\color{gray} 1}} 2 2 {\color{gray} 1} {\color{gray} 1} {\color{gray} 1} {\color{gray} 1} {\color{gray} 1} {\color{gray} 1} {\color{gray} 1} {\color{gray} 1} {\color{gray} 1} {\color{gray} 1} {\color{gray} 1} {\color{gray} 1} {\color{gray} 1} {\color{gray} 1} {\color{gray} 1} \\
t = 6 & \cdots {\color{gray} 1} {\color{gray} 1} {\color{gray} 1} {\color{gray} 1} {\color{gray} 1} 2 3 \overline{2} \overline{2} \underline{\underline{{\color{gray} 1} {\color{gray} 1}}} {\color{gray} 1} {\color{gray} 1} {\color{gray} 1} {\color{gray} 1} \underline{{\color{gray} 1} {\color{gray} 1}} 2 2 {\color{gray} 1} {\color{gray} 1} {\color{gray} 1} {\color{gray} 1} {\color{gray} 1} {\color{gray} 1} {\color{gray} 1} {\color{gray} 1} {\color{gray} 1} {\color{gray} 1} {\color{gray} 1} {\color{gray} 1} {\color{gray} 1} {\color{gray} 1} {\color{gray} 1} {\color{gray} 1} {\color{gray} 1} \\
t = 7 & \cdots {\color{gray} 1} 2 3 \overline{2} \overline{2} \underline{\underline{{\color{gray} 1} {\color{gray} 1}}} {\color{gray} 1} {\color{gray} 1} {\color{gray} 1} {\color{gray} 1} {\color{gray} 1} {\color{gray} 1} \underline{{\color{gray} 1} {\color{gray} 1}} 2 2 {\color{gray} 1} {\color{gray} 1} {\color{gray} 1} {\color{gray} 1} {\color{gray} 1} {\color{gray} 1} {\color{gray} 1} {\color{gray} 1} {\color{gray} 1} {\color{gray} 1} {\color{gray} 1} {\color{gray} 1} {\color{gray} 1} {\color{gray} 1} {\color{gray} 1} {\color{gray} 1} {\color{gray} 1} {\color{gray} 1} {\color{gray} 1} \\
\end{array}}
\]
where we have
\[
H\bigl(\Psi(23) \otimes \Psi(\bon\bon)\bigr) = H(12 \otimes \emptyset) = 2
\]
and a phase shift of $2 \cdot 2 - 2 = 2$.
For additional examples for this soliton in type $C_n^{(1)}$, see~\cite[Ex.~4.6]{HKOTY02} and~\cite[Ex.~3.5,3.6]{HKT00}.

\subsection*{Acknowledgments}

The authors thank Atsuo Kuniba and Masato Okado for valuable discussions and comments on an earlier version of this paper.
The authors thank Ben Salisbury for comments on an earlier version of this paper.
This work benefited from computations using {\sc SageMath}~\cite{sage, combinat}.
The authors would like to thank the referee for comments on our paper.


\appendix
\section{Examples with {\sc SageMath}}

We give some examples using {\sc SageMath}~\cite{sage}, which has been implemented by the second author (the examples given here are using~\cite{ScrimshawSCACode}).
We begin with the code used to construct Example~\ref{ex:SCA_Dtwisted}.

\begin{lstlisting}
sage: initial = [[1],[1],['E'],[1],[1],[1],
....:            [3],[0],[1],[1],[1],[1],[2],[-3],[-1]]
sage: SCA = SolitonCellularAutomata(initial, ['D',4,2])
sage: view(SCA.latex_states(10))
\end{lstlisting}

%
%
Next, we construct Example~\ref{ex:SCA_B_2_neg_shift} by using the following code.

\begin{lstlisting}
sage: I = [[2,1],[-3,2],[2,1],[2,1],[2,1],[2,1],
....:      [3,1],[4,2],[-4,2],[-3,2]]
sage: SCA = SolitonCellularAutomata(I, ['B',4,1], vacuum=2)
sage: view(SCA.latex_states(7))
\end{lstlisting}

The following code is used to construct Example~\ref{ex:SCA_B_spin}.

\begin{lstlisting}
sage: K = crystals.KirillovReshetikhin(['B',3,1], 3,1)
sage: u = K.module_generators[0]
sage: v = u.f(3)
sage: w = u.f_string([3,2,3])
sage: I = [u,v.f(2).f(1),u,u,v,w]
sage: SCA = SolitonCellularAutomata(I, ['B',3,1], vacuum=3)
sage: view(SCA.latex_states(6))
\end{lstlisting}

Continuing with the same variables, we construct Example~\ref{ex:SCA_B_spin2}.

\begin{lstlisting}
sage: I = [u,v.f(2).f(1),u,u,v,v.f(2),v.f(2).f(1)]
sage: SCA = SolitonCellularAutomata(I, ['B',3,1], vacuum=3)
sage: view(SCA.latex_states(6))
\end{lstlisting}

\section{Classical single row crystals for types $E_{6,7}$}
\label{sec:classical}

\begin{prop}
The classically highest weight elements in $B^{1,s} \otimes B^{1,s'}$, where $B^{1,s}$ and $B^{1,s'}$ are of type $E^{(1)}_6$, are given by
\[
[\underbrace{\bze1, \dotsc, \bze1}_{k_1}, \underbrace{\bze\bon3, \dotsc, \bze\bon3}_{k_2}, \underbrace{\bon6, \dotsc, \bon6}_{k_3}] \otimes [\underbrace{\bze1, \dotsc, \bze1}_{s'}]
\]
where $k_1 + k_2 + k_3 = s$ and $k_2 + k_3 \leq s'$.
\end{prop}

\begin{proof}
As noted in the proof~\cite[Lemma~3.1]{JS10}, the $\{2,3,4,5,6\}$-highest weight elements in $B(s\Lambda_1)$ are of the form
\[
b = [\underbrace{\bze1, \dotsc, \bze1}_{k_1}, \underbrace{\bze\bon3, \dotsc, \bze\bon3}_{k_2}, \underbrace{\bon6, \dotsc, \bon6}_{k_3}],
\]
with $\varepsilon_1(b) = k_2 + k_3$.
Note that $u_{s'\Lambda_1} = [\bze1, \dotsc, \bze1]$ is the unique classically highest weight element in $B^{1,s'}$ and that $\varphi_1(u_{s'\Lambda_1}) = s$.
By the tensor product rule, $e_1(b\otimes u_{s'\Lambda_1}) = 0$ if and only if $k_2+k_3 \leq s'$.
Similarly, we have $e_i(b\otimes u_{s\Lambda_1}) = 0$ for all $i \in \{2,3,4,5,6\}$.

Next, by the tensor product rule, any highest element in $B^{1,s} \otimes B^{1,s'}$ must be of the form $b \otimes u_{s'\Lambda_1}$.
Note that $\varepsilon_i(u_{s\Lambda_1}) = 0$ for all $i \in \{2,3,4,5,6\}$, so $b$ must be a $\{2,3,4,5,6\}$-highest weight element.
Therefore, we have described all classically highest weight elements above.
\end{proof}

\begin{prop}
The classically highest weight elements in $B^{7,s} \otimes B^{7,s'}$, where are $B^{7,s} \otimes B^{7,s'}$ are of type $E^{(1)}_7$, are given by
\[
[\underbrace{7, \dotsc, 7}_{k_1}, \underbrace{\bseven6, \dotsc, \bseven6}_{k_2}, \underbrace{\bseven1, \dotsc, \bseven1}_{k_3}, \underbrace{\bseven, \dotsc, \bseven}_{k_4}] \otimes [\underbrace{7, \dotsc, 7}_{s'}]
\]
where $k_1 + k_2 + k_3 + k_4= s$ and $k_2 + k_3 + k_4 \leq s'$.
\end{prop}

\begin{proof}
As noted in the proof~\cite[Lemma~3.1]{JS10}, the $\{1,2,3,4,5,6\}$-highest weight elements in $B(s\Lambda_7)$ are of the form
\[
b = [\underbrace{7, \dotsc, 7}_{k_1}, \underbrace{\bseven6, \dotsc, \bseven6}_{k_2}, \underbrace{\bseven1, \dotsc, \bseven1}_{k_3}, \underbrace{\bseven, \dotsc, \bseven}_{k_4}],
\]
with $\varepsilon_1(b) = k_2 + k_3 + k_4$.
Note that $u_{s'\Lambda_7} = [7, \dotsc, 7]$ is the unique classically highest weight element in $B^{7,s'}$ and that $\varphi_1(u_{s'\Lambda_7}) = s$.
By the tensor product rule, $e_7(b\otimes u_{s'\Lambda_7}) = 0$ if and only if $k_2+k_3 + k_4 \leq s'$.
Similarly, we have $e_i(b\otimes u_{s\Lambda_7}) = 0$ for all $i \in \{1,2,3,4,5,6\}$.

Next, by the tensor product rule, any highest element in $B^{7,s} \otimes B^{7,s'}$ must be of the form $b \otimes u_{s'\Lambda_7}$.
Note that $\varepsilon_i(u_{s\Lambda_7}) = 0$ for all $i \in \{1,2,3,4,5,6\}$, so $b$ must be a $\{1,2,3,4,5,6\}$-highest weight element.
Therefore, we have described all classically highest weight elements above.
\end{proof}

\bibliographystyle{alpha}
\bibliography{solitons}{}
\end{document}